\documentclass{ieeeaccess-arxiv}
\usepackage[numbers]{natbib}
\usepackage{amsmath,amssymb,amsfonts}
\usepackage[noend]{algorithmic}
\usepackage{graphicx}
\usepackage{textcomp}
\usepackage{enumitem}

\usepackage{hyperref}       
\usepackage{url}            
\usepackage{booktabs}       
\usepackage{array}
\usepackage{nicefrac}       
\usepackage{algorithm}
\usepackage{makecell}

\DeclareFontFamily{U}{mathx}{\hyphenchar\font45}
\DeclareFontShape{U}{mathx}{m}{n}{<-> mathx10}{}
\DeclareSymbolFont{mathx}{U}{mathx}{m}{n}
\DeclareMathAccent{\widebar}{0}{mathx}{"73}

\usepackage{xcolor}

\newcommand{\todoan}[1]{}

\usepackage{agafonov}

\newtheorem{definition}{Definition}
\newtheorem{assumption}{Assumption}
\newtheorem{lemma}{Lemma}
\newtheorem{theorem}{Theorem}
\newtheorem{innercustomthm}{Theorem}
\newenvironment{customthm}[1]
  {\renewcommand\theinnercustomthm{#1}\innercustomthm}
  {\endinnercustomthm}
\newtheorem{innercustomlemma}{Lemma}
\newenvironment{customlemma}[1]
  {\renewcommand\theinnercustomlemma{#1}\innercustomlemma}
{\endinnercustomlemma}

\newtheorem{remark}{Remark}

\newcommand{\opnorm}[1]{\|#1\|_{\mathrm{op}}}

\newcommand{\hx}{\hat{x}}
\newcommand{\hu}{\hat{u}}
\newcommand{\hv}{\hat{v}}
\newcommand{\hg}{\hat{g}}
\newcommand{\hH}{\hat{H}}
\newcommand{\hphi}{\hat{\phi}}

\newcommand{\home}{\hat{\omega}}
\newcommand{\hnu}{\hat{\nu}}
\newcommand{\hmu}{\hat{\mu}}
\newcommand{\err}{\mathrm{err}}
\newcommand{\berr}{\widebar{\mathrm{err}}}
\newcommand{\bx}{\bar{x}}
\newcommand{\bv}{\bar{v}}
\newcommand{\bgg}{\bar{g}}
\newcommand{\bH}{\bar{H}}
\newcommand{\bR}{\bar{R}}
\newcommand{\mW}{{W}}
\newcommand{\mX}{{X}}
\newcommand{\mU}{{U}}
\newcommand{\mG}{{G}}

\newcommand{\mA}{{A}}
\newcommand{\hDe}{\hat \Delta}

\newcommand{\Lavg}{\bar{L}}
\newcommand{\Lmax}{L^{\max}}
\newcommand{\muavg}{\bar{\mu}}
\newcommand{\mumin}{\hmu}

\newcommand{\econs}{\hat{\err}}

\newcommand{\Done}[2]{\Delta_{1 \vert #1, #2}}
\newcommand{\Dtwo}[2]{\Delta_{2 \vert #1, #2}}
\newcommand{\Dtwoo}[1]{\Delta_{2 \vert #1}}

\newcommand{\hDex}[2]{\hDe_{#1, #2}}
\newcommand{\hDeg}[2]{\hDe_{g \vert #1, #2}}
\newcommand{\hDeH}[2]{\hDe_{H \vert #1, #2}}

\newcommand{\dimd}{d}
\newcommand{\numIter}{N}
\newcommand{\Tcons}[2]{T_{#1, #2}}
\newcommand{\Tx}[1]{\Tcons{x}{#1}}
\newcommand{\Tg}[1]{\Tcons{g}{#1}}
\newcommand{\THh}[1]{\Tcons{H\!}{#1}}

\newcommand{\Tz}[2]{\Tcons{#1}{#2}}
\newcommand{\Tgz}[2]{T_{g \vert #1, #2}}
\newcommand{\THz}[2]{T_{H \vert #1, #2}}

\newcommand{\numMach}{m}

\usepackage{bm}
\makeatletter
\AtBeginDocument{\DeclareMathVersion{bold}
\SetSymbolFont{operators}{bold}{T1}{times}{b}{n}
\SetSymbolFont{NewLetters}{bold}{T1}{times}{b}{it}
\SetMathAlphabet{\mathrm}{bold}{T1}{times}{b}{n}
\SetMathAlphabet{\mathit}{bold}{T1}{times}{b}{it}
\SetMathAlphabet{\mathbf}{bold}{T1}{times}{b}{n}
\SetMathAlphabet{\mathtt}{bold}{OT1}{pcr}{b}{n}
\SetSymbolFont{symbols}{bold}{OMS}{cmsy}{b}{n}
\renewcommand\boldmath{\@nomath\boldmath\mathversion{bold}}}
\makeatother

\def\BibTeX{{\rm B\kern-.05em{\sc i\kern-.025em b}\kern-.08em
    T\kern-.1667em\lower.7ex\hbox{E}\kern-.125emX}}

\begin{document}

\title{Decentralized Inexact Cubic Newton Method with Consensus Procedure}
\author{\uppercase{Artem Agafonov}\authorrefmark{1, 2, 3},
\uppercase{Anton Novitskii}\authorrefmark{3, 4}, \uppercase{Alexander Rogozin}\authorrefmark{3}, \uppercase{Yury Sokolov}\authorrefmark{3}, \uppercase{Dmitry Kamzolov}\authorrefmark{5}, \uppercase{Alexander G. Dyakonov}\authorrefmark{2}, \uppercase{Martin Tak\'a\v{c}}\authorrefmark{1}, \uppercase{Alexander Gasnikov}\authorrefmark{3,4,6}.
}

\address[1]{Mohamed bin Zayed University of Artificial Intelligence (MBZUAI), Abu Dhabi, UAE}
\address[2]{AI VK, Moscow, Russia}
\address[3]{Moscow Independent Research Institute of Artificial Intelligence, Moscow, Russia}
\address[4]{Trusted AI Research Center, RAS, Moscow, Russia}
\address[5]{Independent Researcher}
\address[6]{Innopolis University, Kazan, Russia}

\tfootnote{This work was supported by a grant, provided by the Ministry of Economic Development of the Russian Federation (agreement  dated June 20, 2025 No. 139-15-2025-011, identifier 000000C313925P4G0002).}

\markboth
{Agafonov \headeretal: Decentralized Inexact Cubic Newton method with Consensus Procedure}
{Agafonov \headeretal: Decentralized Inexact Cubic Newton method with Consensus Procedure}

\corresp{Corresponding author: Artem Agafonov (e-mail: artem.agafonov@mbzuai.ac.ae).}


\begin{abstract}
    Distributed optimization is widely used in large-scale and privacy-preserving machine learning, where each agent stores a local objective and communicates only with its neighbors in a connected network. We study decentralized second-order optimization and focus on consensus procedures that approximately average local iterates, gradients, and Hessians through neighbor-to-neighbor communications. We propose a general Decentralized Cubic Newton method for convex optimization under $L_1$-smoothness of gradients and $L_2$-Lipschitz continuity of Hessians, and develop a theory that accurately tracks the inaccuracies caused by consensus and by disagreement between local iterates. Under these assumptions, the method matches the iteration complexity of the exact Cubic Newton method and requires only additional polylogarithmic communication-round overhead to reach the necessary consensus accuracy. We further propose an Accelerated Decentralized Cubic Newton method for strongly convex objectives and show that it matches the iteration complexity of the exact Accelerated Cubic Newton method, again with only additional polylogarithmic communication-round overhead. Finally, although the general method requires exchanging full $\dimd \times \dimd$ Hessian matrices, we show how it can be implemented for generalized linear models by transmitting only vectors, making the approach substantially more practical in high dimensions.
\end{abstract}

\titlepgskip=-21pt

\maketitle

\section{Introduction}


We consider decentralized optimization over a time-varying communication network represented by a sequence of undirected graphs with $\numMach$ nodes.
Each node~$i\in\{1,\dots,\numMach\}$ has access to a local objective function $f_i:\R^\dimd\to\R$.
The global optimization problem has the form
\begin{equation}
    \label{eq:problem_intro}
    \min_{x\in\R^\dimd}~ f(x) \eqdef \tfrac{1}{\numMach}\tsum_{i=1}^\numMach~f_i(x).
\end{equation}
Distributed optimization problems arise in federated learning \cite{kairouz2021advances}, control of power systems \cite{necoara2011parallel,necoara2014distributed}, resource allocation \cite{doan2017distributed,nedic2018improved} and other topics in multi-agent systems \cite{ren2006consensus}. The topic of distributed optimization emerged in seminal works \cite{tsitsiklis1984problems,bertsekas1989parallel}, form \eqref{eq:problem_intro} was formulated in a series of papers \cite{nedic2009distributed,nedic2009saddle}.

\textbf{First order methods}. Decentralized optimization problems can be solved by different techniques. First-order methods use only gradient information about local functions, i.e. these methods perform local computations of gradients $\nabla f_i(x)$ and function values $f_i(x)$. Therefore, such algorithms interleave local optimization steps and communications with their neighbors. First-order algorithms include several variants of distributed gradient descent \cite{nedic2009distributed,shi2015extra,koloskova2020unified,li2020decentralized}, gradient tracking \cite{nedic2017achieving,pu2020push,pu2021distributed,li2024accelerated,alghunaim2020decentralized} and primal-dual methods \cite{scaman2017optimal,kovalev2020optimal,kovalev2021adom,kovalev2021lower}. The theoretical study of first-order algorithms was summed up by deriving lower complexity bounds in the spirit of classical convex optimization theory \cite{nemirovski1983problem,nesterov2004introduction}, which was performed in \cite{scaman2017optimal,kovalev2021lower}, and building corresponding algorithms that achieve the lower bounds.

\textbf{Proximal methods}. A different approach to decentralized optimization is based on ADMM-type primal-dual proximal methods \cite{boyd2011distributed}. ADMM is applied when local functions have simple structure that allows cheap computation of proximal operator $\prox_{f}^{\gamma}(x) = \argmin_y \cbraces{f(y) + 1/(2\gamma) \norm{y - x}_2^2}$. ADMM-type methods were studied in \cite{boyd2011distributed,chang2016proximal,falsone2020tracking,wu2022distributed,gong2023decentralized}. Their main limitation is prox-friendly local loss functions, which is the case i.e. for simple quadratic loss functions. In other words, ADMM algorithm is not typically applied to functions of general form.

\textbf{Implementing a distributed second order method is nontrivial}. We now turn to second-order methods. They are algorithms that can compute not only function values and gradients, but also second order information, that are Hessians $\nabla^2 f_{i}(x)$. Developing a distributed Newton method even in a centralized setting (i.e. with full aggregation) turns out to be nontrivial. Consider two quadratic functions $f_1(x) = (x + 1)^2$ and $f_2(x) = (2x - 1)^2$. Let a distributed Newton method make two local steps with a following aggregation, i.e. $x_1^{k+1/2} = x^k - (\nabla^2 f_1(x^k))^{-1} \nabla f_1(x^k)$, $x_2^{k+1/2} = x^k - (\nabla^2 f_2(x^k))^{-1} \nabla f_2(x^k)$ and $x^{k+1} = (x_1^{k+1/2} + x_2^{k+1/2}) / 2$. However, such a method will jump out of the global minimum. Indeed, local minima are $x_1^* = -1$ and $x_2^* = 1/2$ and global minimum is $x^* = 1/5$. Starting from $x^0 = x^*$ we get to $x^1 = -1/4$. This simple example shows that a straightforward implementation of the Newton method is not convergent.

\textbf{Non-distributed Newton}. A standard Newton iteration $x^{k+1} = x^k - \gamma_k (\nabla^2 f(x^k))^{-1} \nabla f(x^k)$ with unit stepsize $\gamma_k = 1$ is known to achieve local quadratic convergence, but is non-convergent if started from an arbitrary point \cite{nesterov2018lectures}. Global convergence is achieved via line search for $\gamma_k$. In the modern literature, regularized versions of Newton method are mostly used. Cubic regularization of Newton method was proposed in \cite{nesterov2006cubic}. Gradient regularization of Hessian was studied in \cite{doikov2024gradient,mishchenko2023regularized}. In \cite{jiang2026beyond}, a trust-region approach combined with gradient regularization was proposed. Under Lipschitz continuity of Hessian (i.e. $\norm{\nabla^2 f(y) - \nabla^2 f(x)}\leq \bar L_2\norm{y - x}$), all mentioned regularized versions of Newton method achieve $\cO(\sqrt{\bar L_2 D^3/\e})$ convergence rate, where $D$ characterizes the distance to solution from starting point. The same convergence rate holds for damped Newton method, as shown in \cite{hanzely2022damped}.

\textbf{Centralized Newton}. We first discuss centralized second order methods. Exchanging Hessians between the nodes is the communication bottleneck of distributed methods. In the case when local functions are loss functions of generalized linear models \cite{NL2021,wang2018giant}, Hessian exchange can be made significantly cheaper. A different approach is applying a random compression operator to the Hessians before exchange \cite{safaryan2021fednl,Newton-3PC}. Under a data similarity assumption ($\norm{\nabla^2 f(x) - \nabla^2 f_i(x)}_2\leq \beta$ for some constant $\beta > 0$), a distributed Cubic Newton method was proposed in \cite{agafonov2021accelerated}. It is also possible to replace a Newton update with a low-rank update \cite{agafonov2025flecs}, which reduces the amount of exchanged information.

\textbf{Decentralized Newton}. Considering decentralized second order methods, the results in the current literature are not well developed. Newton method with assumption on closeness of local summands was derived in \cite{daneshmand2021newton}, employing a gradient tracking technique over gradients without any Hessian exchange. While this approach elegantly avoids communicating Hessians, it relies  on the second-order $\beta$-similarity assumption. 
In contrast, our proposed method relies only on standard Lipschitz continuity, making it robust to data heterogeneity.
Decentralized quasi-Newton methods were proposed in \cite{eisen2017decentralized,zhang2023variance}. Newton method with gradient tracking technique achieving linear convergence rate was proposed in \cite{zhang2021newton}. An approach proposed in \cite{uribe2020distributed} approximately computes a Newton step via a decentralized procedure. The outer method is accelerated Newton and the inner method computes the Newton step with accuracy $\delta$. The best estimate given in the paper is $\delta = \mathcal{O}(\e^3)$, resulting in a total communication complexity of $\mathcal{O}\left((\bar L_2 D^3/\e)^{1/3} \sqrt{\chi \bar L_1/\e^3} \ln(1/\e)\right)$. In comparison with \cite{uribe2020distributed}, our method makes local Newton steps and performs a decentralized aggregation afterwards. As a results, we have a better communication complexity in comparison to \cite{uribe2020distributed}. Moreover, lazy Hessians can be applied to distributed second-order optimization \cite{zhang2024communication}. In \cite{mokhtari2016network}, a second-order method with $\mathcal{O}(\bar L_2 / (\mu\e)\ln(1/\e))$ complexity is proposed. In lazy Hessian approach, local Hessians are communicated and updated once in several local iterations. Hessian exchange can also be cheaper in case of specific structure of local functions \cite{NL2021}.


\textbf{Our contribution}.

We develop a unified (strongly) convex theory of consensus-based decentralized cubic Newton methods that explicitly accounts for errors due to consensus and disagreement between local iterates. This theory is established without additional similarity assumptions such as $\beta$-similarity between local objectives, but relies on standard $L_1$,~$L_2$ smoothness assumptions (Assumptions~\ref{as:L1},~\ref{as:L2}).

\begin{table}[h!]
\centering
\begin{tabular}{|p{1.5cm}|p{6.5cm}|}
\hline
Method & Complexity \\
\hline
Di-Reg-INA \cite{daneshmand2021newton} &
Convex: 
$\cO\cbraces{\cbraces{\sqrt{\frac{\bar L_2 D^3}{\e}} + \beta \frac{\bar L_2 D^3}{\e}} \sqrt\chi \log\cbraces{\frac{1}{\e}}}$ \\
& Str. Conv: 
$\cO\cbraces{\sqrt{\frac{\bar L_2 D}{\mu}}\cbraces{1 + m^{1/4}\sqrt{\frac{\beta}{\mu}} + \frac{\beta}{\mu} \sqrt\chi \ln\cbraces{\frac1\e}}}$ \\
\hline
Distr. cubic solver \cite{uribe2020distributed} &
Convex: $\cO\cbraces{\cbraces{\frac{\bar L_2^{1/3} \bar L_1^{1/2} D}{\e^{11/6}}} \sqrt\chi\ln\cbraces{\frac1\e}}$ \\
& Str. Convex: -- \\
\hline
This paper &
Convex: $\cO \ls \sqrt{\tfrac{\Lavg_2 D^3}{\e}} \tfrac{\tau}{\lambda}\log \tfrac{1}{\e} \rs$ \\
& Str. Convex: $\cO\cbraces{\left(\tfrac{\bar L_2 \bR}{\muavg}\right)^{1/3}
    \tfrac{\tau}{\lambda}\log^2 \tfrac{f(\bx^0)-f(x^*)}{\e}}$ \\
\hline
\end{tabular}
\caption{Decentralized Newton complexity bounds for static graphs. $\overline L_2$ denotes the average Lipschitz constant of $\nabla^2 f_i(x)$; $\overline \mu$ is the average strong convexity constant defined in Assumption \ref{as:strcnvxty}, $\beta$ is the second-order similarity parameter used in \cite{daneshmand2021newton} ($\|\nabla^2 f_i(x) - \nabla^2 f(x)\|\leq \beta$), $D$ is the distance to solution defined in \eqref{eq:lebesgue_set_diameter} and $\overline R$ is the maximum deviation from solution over algorithm trajectory defined in \eqref{eq:boundnes}. Graph condition number is denoted $\chi$ that equals $\chi = 1/\lambda$ for static graphs (see Remark~\ref{rem:chebyshev}), where $\lambda$ is mixing matrix eigengap defined in Assumption \ref{assum:mixing_matrix_sequence}.}
\label{tab:decentr_newton_complexity}
\end{table}

Our contributions are summarized as follows.

$\bullet$ We propose Decentralized Cubic Newton for the convex setting. The method achieves the total communication complexity to reach accuracy $\e$
\begin{equation*}
    \cO \ls \sqrt{\tfrac{\Lavg_2 D^3}{\e}} \tfrac{\tau}{\lambda}\log \tfrac{1}{\e} \rs,
\end{equation*}
where diameter $D$ and average $L_2$-Lipschitz constant $\Lavg_2$ are defined in~\eqref{eq:lebesgue_set_diameter},~\eqref{eq:l2_avg} respectively, and $\tau,\lambda$ are defined in Assumption~\ref{assum:mixing_matrix_sequence}. This rate matches the convergence of exact Cubic Newton method~\cite{nesterov2006cubic} up to logarithmic factor.

$\bullet$ We propose decentralized cubic Newton method and its accelerated version for the strongly convex setting. The non-accelerated method achieves the total communication complexity to reach accuracy $\e$ of
\begin{equation*}
    \cO(1)\max\lb
        1,\,
        \left(\tfrac{\Lavg_2 D^3}{\muavg}\right)^{1/2}
    \rb
    \tfrac{\tau}{\lambda}\log^2 \tfrac{f(\bx^0)-f(x^*)}{\e},
\end{equation*}
while the accelerated variant improves this to
\begin{equation*}
    \cO(1)\max\lb
        1,\,
        \left(\tfrac{L\bR}{\muavg}\right)^{1/3}
    \rb
    \tfrac{\tau}{\lambda}\log^2 \tfrac{f(\bx^0)-f(x^*)}{\e},
\end{equation*}
where $\muavg$ is the averaged strong convexity constant, and $\bR$ is defined in~\eqref{eq:boundnes}. This rate matches the convergence of exact Accelerated Cubic Newton method~\cite{nesterov2008accelerating} up to logarithmic factor. To the best of our knowledge, accelerated Newton for decentralized optimization previously was proposed only in \cite{uribe2020distributed}. However, their approach includes the solution of a non-strongly convex cubic problem on each iteration, which results in a $\cO(1/\e^{11/6} \ln(1/\e))$ communication complexity. Our approach uses a consensus subroutine and the total communication complexity is $\cO(1/\e^{1/3} \ln^2(1/\e))$.

$\bullet$ In the general form, the method requires exchanging full $\dimd \times \dimd$ Hessian matrices. To address this bottleneck, we discuss how the same consensus-based scheme can be implemented for generalized linear models using only vector transmissions, which makes the method substantially more practical in high-dimensional problems.

\textbf{Organization.} The rest of the paper is organized as follows. Section~\ref{sec:setup} introduces the problem setup and notation. Section~\ref{sec:method} presents the Decentralized Cubic Newton method together with its convergence guarantees. Section~\ref{sec:acc_method} introduces the accelerated variant (Accelerated Decentralized Cubic Newton) and establishes its convergence guarantees. Section~\ref{sec:implementation} discusses implementation details for generalized linear models. All proofs are deferred to Appendix~\ref{app:proofs}.

\section{Problem setup and notation}
\label{sec:setup}

We consider decentralized optimization over a deterministic time-varying undirected communication network with $m$ nodes. Individual graphs in the sequence need not be connected; connectivity is imposed through the contraction condition in Assumption~\ref{assum:mixing_matrix_sequence}.
Each node~$i\in\{1,\dots,\numMach\}$ has access to a local objective function $f_i:\R^\dimd\to\R$. The node can compute function values $f_i(x_i)$, gradient values $\nabla f_i(x_i)$ and Hessian values $\nabla^2 f_i(x_i)$.
The global optimization objective is given in \eqref{eq:problem_intro}. The notation used in the paper (for non-accelerated Decentralized Cubic Newton method) is summarized in Table~\ref{tab:basic_notation}. Throughout this paper, functions are assumed to be either convex or strongly convex.
\begin{assumption}
    \label{as:cnvxty}
    We assume that each function $f_i$ is convex. Then the global objective $f(x) = \tfrac{1}{m}\tsum_{i=1}^\numMach f_i(x)$ is convex as well. 
\end{assumption}
\begin{assumption}
    \label{as:strcnvxty}
    We assume that each function $f_i$ is $\mu_i$-strongly convex with $\mu_i > 0$. Then the global objective $f(x) = \tfrac{1}{m}\tsum_{i=1}^\numMach f_i(x)$ is $\muavg$-strongly convex with 
$\muavg = \tfrac{1}{\numMach}\tsum_{i=1}^\numMach \mu_i$. 
\end{assumption}


\textbf{Averaging and disagreement.}
For a collection of local vectors $\{u_i\}_{i=1}^\numMach \subset \R^\dimd$, we denote their average by $\bar u \eqdef \tfrac{1}{\numMach}\tsum_{i=1}^\numMach u_i$. 
To write the network updates compactly, we associate with any such collection a matrix $\mU \in \R^{\numMach \times \dimd}$ whose $i$-th row is $u_i^\top$. In particular, for local iterates $\{x_i^k\}_{i=1}^\numMach$ we define their average $\bx^k$ and the corresponding matrix $\mX^k \in \R^{\numMach \times \dimd}$:
\begin{equation}
    \label{eq:mean_x}
    \bx^k \eqdef \tfrac{1}{\numMach}\tsum_{i=1}^\numMach x_i^k, \quad \mX^k \eqdef \begin{bmatrix} (x_1^k)^\top \\ \vdots \\ (x_\numMach^k)^\top \end{bmatrix}, \quad 
    \bar{X}^k = \begin{pmatrix}
        (\bar{x}^k)^\top \\ \vdots \\ (\bar{x}^k)^\top
    \end{pmatrix}.
\end{equation}

\textbf{Consensus routine.}   
We assume that nodes are connected via a deterministic time-varying network represented by a sequence of graphs $\{\mathcal{G}^k = (\mathcal{V}, \mathcal{E}^k)\}_{k=0}^\infty$ with a common vertex set $\mathcal{V} = \{1, \dots, \numMach\}$. With each graph, we associate a mixing matrix $\mW^k \in \R^{\numMach \times \numMach}$. 

    
\begin{assumption}
\label{assum:mixing_matrix_sequence}
    The sequence of mixing matrices $\{\mW^k\}_{k=1}^\infty$ satisfies the following properties:
    \begin{enumerate}
        \item (Network compatibility) For each $k \ge 1$, $[\mW^k]_{ij} = 0$ if $(i, j) \notin \mathcal{E}^k$ and $i \neq j$.
        \item (Double stochasticity) For each $k \ge 1$, we have $[W^k]_{ij}\geq 0$ for any $i, j = 1, \ldots, \numMach$, $\mW^k \mathbf{1}_\numMach= \mathbf{1}_\numMach$, $\mathbf{1}_\numMach^\top \mW^k = \mathbf{1}_\numMach^\top$, where $\mathbf{1}_\numMach$ is a vector of all ones.
        \item (Contraction property) There exist an integer $\tau \ge 1$ and $\lambda \in (0, 1)$ such that for all $k \ge \tau - 1$ and any matrix $\mU \in \R^{\numMach \times d}$, it holds that

        \begin{equation}
            \left\| \mW_{\tau}^k \mU - \bar{\mU} \right\|_F \le (1 - \lambda)\|\mU - \bar{\mU}\|_F,
        \end{equation}
        where $\mW_{\tau}^k \eqdef \mW^k \mW^{k-1} \dots \mW^{k-\tau+1}$ and $\bar{\mU} \eqdef \frac{1}{\numMach}\mathbf{1}_\numMach\mathbf{1}_\numMach^\top \mU$ (a matrix where each row is the average vector $(\bar{u}^k)^\top$).
    \end{enumerate}
\end{assumption}

The contraction property in Assumption \ref{assum:mixing_matrix_sequence} generalizes several assumptions in the literature.
\begin{itemize}
    \item Time-static connected graph: $\mW^k = \mW$. In this classical case, the network is connected at every iteration, meaning $\tau = 1$. The one-step contraction factor is given by $\lambda = 1 - \sigma_2(\mW)$, where $\sigma_2(\mW)$ denotes the second largest singular value of $\mW$.

    \item Sequence of connected graphs: every $\mathcal{G}_k$ is connected. Here we also have $\tau = 1$, and the parameter is defined by the worst-case spectral gap: $\lambda = 1 - \sup_{k \ge 0} \sigma_2(\mW^k)$.
    \item $\tau$-connected graph sequence (i.e. for every $k \ge 0$ graph $\mathcal{G}_\tau^k = (\mathcal{V}, \mathcal{E}^k \cup \mathcal{E}^{k+1} \cup \dots \cup \mathcal{E}^{k+\tau-1})$ is connected \cite{nedic2017achieving}.
    For $\tau$-connected graph sequences it holds $1 - \lambda = \sup_{k \ge 0} \sigma_{\max}(\mW_\tau^k - \frac{1}{\numMach} \mathbf{1}_\numMach \mathbf{1}_\numMach^\top)$. In this general setting, the graphs $\mathcal{G}^k$ in the sequence are not required to be connected individually; connectivity is only required over the union of edges across any $\tau$ consecutive rounds.
\end{itemize}

\begin{remark}\label{rem:chebyshev}
    For time-static graphs, mixing matrix $W$ can be replaced by a matrix polynomial $P_K(W)$ of degree $K = \lceil\sqrt\chi\rceil$. Taking $P_K(W)$ as a Chebyshev polynomial with correct coefficients enables to make it well-conditioned, i.e. $\chi = 1/(1 - \sigma_2(P_K(W))) = \cO(1)$ \cite{scaman2017optimal,salim2022optimal}. Thus, one communication round will be replaced by $\cO(\sqrt\chi)$ communication rounds, but the effective condition number will be reduced to $\cO(1)$. As a result, for static graphs a communication complexity factor $\tau/\lambda = 1/\lambda = \chi$ in Table \ref{tab:decentr_newton_complexity} reduces to $\cO(\sqrt\chi)$. The only difference is to replace Consensus procedure \eqref{eq:consensus} with accelerated consensus (see i.e. Algorithm 2 in \cite{scaman2017optimal}).
\end{remark}

During every single communication round, the agents exchange information according to the rule $\mU^{k+1} = \mW^k \mU^k$. In particular, the contraction property holds for $\tau$-connected graphs with Metropolis weights choice for $\mW^k$, i.e.
\begin{equation*}
    [\mW^k]_{ij} = 
    \begin{cases}
        1/(1 + \max\{q_i^k, q_j^k\}) & \text{if } (i, j) \in \mathcal{E}^k, \\
        0 & \text{if } (i, j) \notin \mathcal{E}^k, \\
        1 - \tsum\limits_{(i,m) \in \mathcal{E}^k} [\mW^k]_{im} & \text{if } i = j,
    \end{cases}
\end{equation*}
where $q_i^k$ denotes the degree of node $i$ in graph $\mathcal{G}^k$.

The contraction property above ensures that repeated application of the mixing matrices drives local variables to their average. 
Accordingly, in the convergence analysis we model consensus as an abstract routine and capture its effect via suitable accuracy bounds introduced below. 
The contraction property will later be used to relate these bounds to the number of communication rounds.

\textbf{Consensus errors.}
Given local input vectors $\{u_i\}_{i=1}^\numMach$, we denote by 
\begin{equation}\label{eq:consensus}
    \hu_i \eqdef \mathrm{Consensus}(\{u_j\}_{j=1}^\numMach, T) \in \R^\dimd
\end{equation}
the output of $T$ consensus communication rounds on node $i$ (the node number $i$ is known from left-hand side). Equivalently, in matrix form, letting $\mU \in \R^{\numMach \times \dimd}$ be the matrix
with rows $u_i^\top$, the output of consensus can be written as
\begin{equation*}
    \hat{\mU} = \mW^k \mW^{k-1} \dots \mW^{k-T+1} \mU.
\end{equation*}

Given a sequence of local iterates $\{x_i^k\}_{k \ge 0}$, we denote by $\hx_i^k$ the result of applying consensus to $\{x_j^k\}_{j=1}^\numMach$ on node $i$. We assume that the consensus error is bounded as
\begin{equation}
\label{eq:cons_bounds_x}
    \max_{i\in[\numMach]} \|\hx^k_i - \bx^k\| \le \hDex{\hx}{k}.
\end{equation}
We use the notation $\hDe_{u,k}$ to indicate the quantity (e.g., $x$, $\hx$) and the iteration index at which the consensus error is evaluated. Such bounds can be ensured by performing a sufficient number of communication rounds, due to the contraction property in Assumption~\ref{assum:mixing_matrix_sequence}.

Similarly, applying consensus to local gradients $\{\nabla f_j(x_j^k)\}_{j=1}^\numMach$ 
and Hessians $\{\nabla^2 f_j(x_j^k)\}_{j=1}^\numMach$, we denote the outputs on node $i$ by 
$\hg_i(x_i^k)$ and $\hH_i(x_i^k)$, respectively,
where the argument indicates the points at which local derivatives are evaluated. For the Hessians, the input matrix $\mU \in \R^{\numMach \times \dimd^2}$ is formed such that each row is a flattened local Hessian. We also assume that:
\begin{subequations}
    \label{eq:cons_bounds_derivatives}
    \begin{align}
        \max_{i\in[\numMach]} \|\hg_i (x_i^k) - \tfrac{1}{\numMach} \tsum_{j=1}^\numMach \nabla f_j(x_j^k)\| & \le    \hDeg{x}{k}  
        \label{eq:cons_bounds_gradient}\\
        \max_{i\in[\numMach]} \opnorm{\hH_i (x_i^k) - \tfrac{1}{\numMach} \tsum_{j=1}^\numMach \nabla^2 f_j(x_j^k)} & \le \hDeH{x}{k}  .
        \label{eq:cons_bounds_Hessian}
    \end{align}
\end{subequations}

Since consensus is applied to local derivatives evaluated at possibly different points (e.g., $x_i^k$ or $\hx_i^k$), we explicitly indicate these evaluation points in the notation. 
In particular, expressions such as $\hg_i(x_i^k)$ and $\hg_i(\hx_i^k)$ denote the outputs of consensus applied to 
$\{\nabla f_j(x_j^k)\}_{j=1}^{\numMach}$ and $\{\nabla f_j(\hx_j^k)\}_{j=1}^{\numMach}$, respectively. In the second case we denote consensus bound error by $\hDeg{\hx}{k}$.

Here $\|\cdot\|$ and $\|\cdot\|_{\mathrm{op}}$ denote the Euclidean and operator norms, respectively.

\textbf{Smoothness assumptions.}
We assume both first- and second-order Lipschitz properties (with possibly different constants per node).

\begin{assumption}[Local $L_1$-smoothness]
\label{as:L1}
Each $f_i$ is differentiable and its gradient is $L_{1,i}$-Lipschitz:
\begin{equation}
    \label{eq:L1}
    \|\nabla f_i(x) - \nabla f_i(y)\| \le L_{1,i}\|x-y\| \qquad \forall x,y\in\R^\dimd.
\end{equation}
\end{assumption}

\begin{assumption}[Local $L_2$-Lipschitz Hessian]
\label{as:L2}
Each $f_i$ is twice differentiable and its Hessian is $L_{2,i}$-Lipschitz:
\begin{equation}
    \label{eq:L2}
    \opnorm{\nabla^2 f_i(x) - \nabla^2 f_i(y)} \le L_{2,i}\|x-y\| \qquad \forall x,y\in\R^\dimd.
\end{equation}
\end{assumption}

Let us introduce second-order Taylor polynomial of local objectives $f_i$
\begin{equation*}
    \phi_i (y; x) = f_i(x) + \langle \nabla f_i(x), y - x \rangle + \tfrac{1}{2} \langle \nabla^2 f_i(x)(y - x), (y - x) \rangle.
\end{equation*}
Assumption \ref{as:L2} allows to control the quality of the approximation of the objective and its gradient by its Taylor polynomial~\cite{nesterov2018lectures}: 
\begin{equation*}
    |f_i(y) - \phi_i(y; x)| \le \tfrac{L_{2, i}}{6}\|y - x\|^{3}, x, y \in  \mathbb{R}^\dimd.
\end{equation*}
\begin{equation*}
    \|\nabla f_i(y) - \nabla \phi_i(y; x)\| \le \tfrac{L_{2, i}}{2}\|y - x\|^{2}, x, y \in  \mathbb{R}^\dimd.
\end{equation*}
Let us define global constants
\begin{equation}
    \label{eq:l2_avg}
    \Lavg_1 \eqdef \tfrac{1}{\numMach}\tsum_{i=1}^\numMach L_{1,i},
    \quad
    \Lavg_2 \eqdef \tfrac{1}{\numMach}\tsum_{i=1}^\numMach L_{2,i}.
\end{equation}
Then, the global objective $f$ has an $\Lavg_1$-Lipschitz gradient
\begin{equation*}
    \begin{aligned}
        \|\nabla f(x)-\nabla f(y)\|
        \le & \tfrac{1}{\numMach}\tsum_{i=1}^\numMach \opnorm{\nabla f_i(x)-\nabla f_i(y)} \\
        \le & ~\Lavg_1\|x-y\|.
    \end{aligned}
\end{equation*}
and $\Lavg_2$-Lipschitz Hessian 
\begin{equation*}
    \begin{aligned}
        \opnorm{\nabla^2 f(x)-\nabla^2 f(y)} 
        \le & \tfrac{1}{\numMach}\tsum_{i=1}^\numMach \opnorm{\nabla^2 f_i(x)-\nabla^2 f_i(y)} \\
        \le & ~\Lavg_2\|x-y\|.
    \end{aligned}
\end{equation*}
Let us introduce second-order Taylor polynomial of global objective $f$
\begin{equation*}
    \phi (y; x) = f(x) + \langle \nabla f(x), y - x \rangle + \tfrac{1}{2} \langle \nabla^2 f(x)(y - x), (y - x) \rangle,
\end{equation*}
which satisfies
\begin{equation}
    \label{eq:taylor_bound_global}
    |f(y) - \phi(y; x)| \le \tfrac{\Lavg_2}{6}\|y - x\|^{3}, x, y \in  \mathbb{R}^\dimd.
\end{equation}
\begin{equation}
    \label{eq:taylor_grad_bound_global}
    \|\nabla f(y) - \nabla \phi(y; x)\| \le \tfrac{\Lavg_2}{2}\|y - x\|^{2}, x, y \in  \mathbb{R}^\dimd.
\end{equation}






Table~\ref{tab:basic_notation} summarizes the notation used in the non-accelerated method.

\begin{table}[t]
\caption{Core notation for the basic decentralized cubic Newton method.}
\label{tab:basic_notation}
\centering
\scriptsize
\renewcommand{\arraystretch}{1.2}
\begin{tabular}{@{} >{\raggedright\arraybackslash}p{0.25\linewidth} >{\raggedright\arraybackslash}p{0.7\linewidth} @{}}
\toprule
\multicolumn{2}{@{}l}{\textbf{Problem, network, and global constants}} \\
\midrule
\quad $m$, $d$ & number of nodes, dimension \\
\quad $f_i$, $f$ & local objective, global objective~\eqref{eq:problem_intro} \\
\quad $x^*$ & global minimizer \\
\midrule
\multicolumn{2}{@{}l}{\textbf{Smoothness and strong convexity parameters}} \\
\midrule
\quad $L_{p,i}$, $\bar{L}_{p}$ & local and global $p$-th order Lipschitz constants~\eqref{eq:l2_avg} \\
\quad $\mu_i,~\bar{\mu},~\hat{\mu}$ & strong convexity parameters \\
\midrule
\multicolumn{2}{@{}l}{\textbf{Consensus errors and algorithmic parameters}} \\
\midrule
\quad $\hat{\Delta}_{\hat{x},k}$ & iterate consensus error~\eqref{eq:cons_bounds_x} \\
\quad $\hat{\Delta}_{g|\hat{x},k}$, $\hat{\Delta}_{H|\hat{x},k}$ & gradient and Hessian consensus errors at $\hat{x}_i^k$~\eqref{eq:cons_bounds_gradient},~\eqref{eq:cons_bounds_Hessian} \\
\quad $\Delta_{1|\hat{x},k}$, $\Delta_{2|\hat{x},k}$ & aggregated consensus errors~\eqref{eq:agg_deltas} \\
\quad $\delta_{1,k}$, $\delta_{2,k}$, $L$ & regularization parameters in $\hat{\omega}_i^k(x;y)$ \\
\quad $\Tx{k}$, $\Tg{k}$, $\THh{k}$ & numbers of consensus rounds for iterates, gradients, and Hessians in Algorithm~\ref{alg:basic} \\
\midrule
\multicolumn{2}{@{}l}{\textbf{Iterates, averages, and local models}} \\
\midrule
\quad $x_i^k$, $\bar{x}^k$, $\hat{x}_i^k$ & local iterate, averaged iterate~\eqref{eq:mean_x}, and consensus iterate \\
\quad $\hat{g}_i(\hat{x}_i^k)$, $\hat{H}_i(\hat{x}_i^k)$ & consensus gradient and Hessian at $\hat{x}_i^k$ \\
\quad $\phi_i(y;x)$, $\phi(y;x)$ & exact local and global Taylor models \\
\quad $\hat{\phi}_i^k(x;y)$, $\hat{\omega}_i^k(x;y)$ & inexact local Taylor model~\eqref{eq:inexact_taylor_2ord} and local cubic model~\eqref{eq:model_2ord} \\
\bottomrule
\end{tabular}
\end{table}

\section{Decentralized Cubic Newton}
\label{sec:method}

In this section, we propose a decentralized variant of the Cubic Regularized Newton method.

\textbf{Local model of the objective.}
We begin by introducing the local model used by each node to perform the update. This model is constructed using inexact local estimates of the first- and second-order derivatives obtained via the consensus procedure.

\begin{definition}[Local inexact cubic model]
\label{def:model}
Given the local point $\hx_i^k$ and local inexact derivatives $\hg_i(\hx_i^k),\hH_i(\hx_i^k)$, define the local smoothed model of the objective
\begin{equation}
\label{eq:model_2ord}
    \begin{aligned}
        \home^{k}_i(x; \hx^k_i)
        \eqdef&
        \la \hg_i(\hx_i^k), x-\hx_i^k\ra  
        + \tfrac12\la \hH_i(\hx_i^k)(x-\hx_i^k), x-\hx_i^k\ra \\
        & + \tfrac{\delta_{1,k}}{2\gamma} 
        + \ls\tfrac{\gamma\delta_{1,k} + \delta_{2,k}}{2}\rs\|x-\hx_i^k\|^2
        + \tfrac{L}{6}\|x-\hx_i^k\|^3
    \end{aligned}
\end{equation}
where $\delta_{1,k},\delta_{2,k}\ge 0$ are regularization parameters and $\gamma > 0$ is a smoothing parameter.
\end{definition}

Compared to standard Taylor models (e.g., $\phi(x; y)$ or the inexact tensor model~\cite[(15)]{agafonov2023inexact}), the model~\eqref{eq:model_2ord} does not include the function value term.
This constant does not affect the minimizer in~\eqref{eq:step} and is therefore omitted from the local model. It will be introduced only in the analysis when needed.

\begin{algorithm}
  \caption{Decentralized Cubic Newton}\label{alg:basic}
  \begin{algorithmic}[1]
    \STATE \textbf{Input:} initial parameters on each node $\{x_i^0\}_{i=1}^\numMach$, total number of steps $\numIter$, sequence of regularization parameters $\{\delta_{1,k}\}_{k=0}^\numIter$, $\{\delta_{2,k}\}_{k=0}^\numIter$, regularization parameter $L > 0$, smoothing parameter $\gamma > 0$, sequence of number of consensus steps $\{\Tx{k}\}_{k=0}^\numIter$, $\{\Tg{k}\}_{k=0}^\numIter$, $\{\THh{k}\}_{k=0}^\numIter$.
    \FOR{$k \in [0, \numIter]$ in parallel on each machine $i \in [\numMach]$} 
    \STATE Perform consensus on parameters:
        \begin{equation}
        \label{eq:algo_cons_x}
            \hx_i^k \eqdef \mathrm{Consensus}(\{x_j^k\}_{j=1}^\numMach, \Tx{k}).
        \end{equation}
    \STATE  Compute local derivatives $\nabla f_i(\hx_i^k)$ and $\nabla^2 f_i(\hx_i^k)$.
    \STATE Perform consensus on derivatives:
        \begin{equation}
            \label{eq:consensus_derivatives}
            \begin{gathered}
                \hg_i(\hx_i^k) \eqdef \mathrm{Consensus}(\{\nabla f_j(\hx_j^k)\}_{j=1}^\numMach, \Tg{k}),\\
                \hH_i(\hx_i^k) \eqdef \mathrm{Consensus}(\{\nabla^2 f_j(\hx_j^k)\}_{j=1}^\numMach, \THh{k}).
            \end{gathered}
        \end{equation}
    \STATE Make the local step:
        \begin{equation}
            \label{eq:step}
            x_{i}^{k+1} = \argmin_{x\in\R^\dimd}  \home_i^k(x; \hx^k_i),
        \end{equation}
    \ENDFOR
  \end{algorithmic}
\end{algorithm}

\begin{definition}[Local inexact Taylor approximation]
\label{def:inexact_taylor}
Given the local point $\hx_i^k$ and local inexact derivatives $\hg_i(\hx_i^k),\hH_i(\hx_i^k)$, define the local Taylor approximation
\begin{equation}
\label{eq:inexact_taylor_2ord}
    \begin{aligned}
        \hphi^{k}_i(x; \hx^k_i)
        \eqdef &
        ~f(\hx_i^k) + \la \hg_i(\hx_i^k), x-\hx_i^k\ra \\ 
        & + \tfrac12\la \hH_i(\hx_i^k)(x-\hx_i^k), x-\hx_i^k\ra.
    \end{aligned}
\end{equation}
\end{definition}

The next lemma shows that under the consensus procedure, the quality of approximation of global objective $f$ can be controlled by $\hphi^{k}_i$.
\begin{lemma}[Inexact Taylor bounds under consensus]
\label{lem:inexact_taylor}
Let Assumptions~\ref{as:L1},~\ref{as:L2} hold. Let $\hg_i(\hx_i^k)$ and $\hH_i(\hx_i^k)$ be the local gradient and Hessian approximations obtained after the consensus procedure at node $i \in [\numMach]$, with consensus errors defined in~\eqref{eq:cons_bounds_x},~\eqref{eq:cons_bounds_derivatives}. Define the aggregated consensus errors 
\begin{equation}
    \label{eq:agg_deltas}
    \begin{aligned}
       \Done{\hx}{k} &\eqdef \hDeg{\hx}{k} + 2\Lavg_1\hDex{\hx}{k}, \\
       \Dtwo{\hx}{k} &\eqdef \hDeH{\hx}{k} + 2\Lavg_2\hDex{\hx}{k}.
    \end{aligned}
\end{equation}
Then, for any $x\in\R^\dimd$,
\begin{equation}
    \label{eq:inexact_taylor_value}
    \begin{aligned}
        \bigl| f(x) - \hphi_i^k(x; \hx_i^k) \bigr| \le~ &\Done{\hx}{k}\|x-\hx_i^k\| + \tfrac{\Dtwo{\hx}{k}}{2} \|x-\hx_i^k\|^2\\
        & + \tfrac{\Lavg_2}{6}\|x-\hx_i^k\|^3  .
    \end{aligned}
\end{equation}
Moreover,
\begin{equation}
\label{eq:inexact_taylor_grad}
\begin{aligned}
    \|\nabla f(x) -  \nabla \hphi_i^k(x; \hx_i^k)\|
    \le~&\Done{\hx}{k} +\Dtwo{\hx}{k} \|x-\hx_i^k\| \\
    &+ \tfrac{\Lavg_2}{2}\|x-\hx_i^k\|^2.
\end{aligned}
\end{equation}
\end{lemma}


Thus, Lemma~\ref{lem:inexact_taylor} provides a bound on the inexactness of the local Taylor model $\hphi_i^k$ in terms of the consensus errors. This bound is similar to those obtained for inexact tensor methods (see, e.g.,~\cite[Lemmas~1--2]{agafonov2023inexact}) and suggests that similar analysis techniques could be applied.

However, a key difference in our setting is that the cubic step is performed locally at each node, while the convergence guarantees must be established for the global objective evaluated at the averaged iterate $\bx^k$. As a result, the standard analysis cannot be applied directly, since it typically controls the descent of a single sequence of iterates, whereas here multiple local sequences are generated and only their average is used in the analysis. This leads to additional error terms, which, as we show, can be controlled by increasing the number of consensus steps.

\textbf{The Method.}
The proposed method is a decentralized variant of the Cubic Regularized Newton. At every iteration, each node first approximates the network average iterate by consensus, obtaining $\hx_i^k$. It then computes local first- and second-order derivatives at $\hx_i^k$ and applies consensus again to obtain approximations of the global gradient $\hg_i(\hx_i^k)$ and Hessian $\hH_i(\hx_i^k)$. Using these quantities, node $i$ builds the local inexact cubic model $\home_i^k(x; \hx_i^k)$ and sets $x_i^{k+1}$ to its minimizer. The resulting method is summarized in Algorithm~\ref{alg:basic}. All scalar parameters, namely $L$, $\gamma$, $\delta_{1,k}$, $\delta_{2,k}$, and the prescribed numbers of consensus rounds, are common algorithmic parameters known to all nodes. The local computations are node-specific only through the local functions $f_i$ and the local consensus outputs. 

\begin{figure}[htbp]
    \centering
    \includegraphics[width=0.5\textwidth]{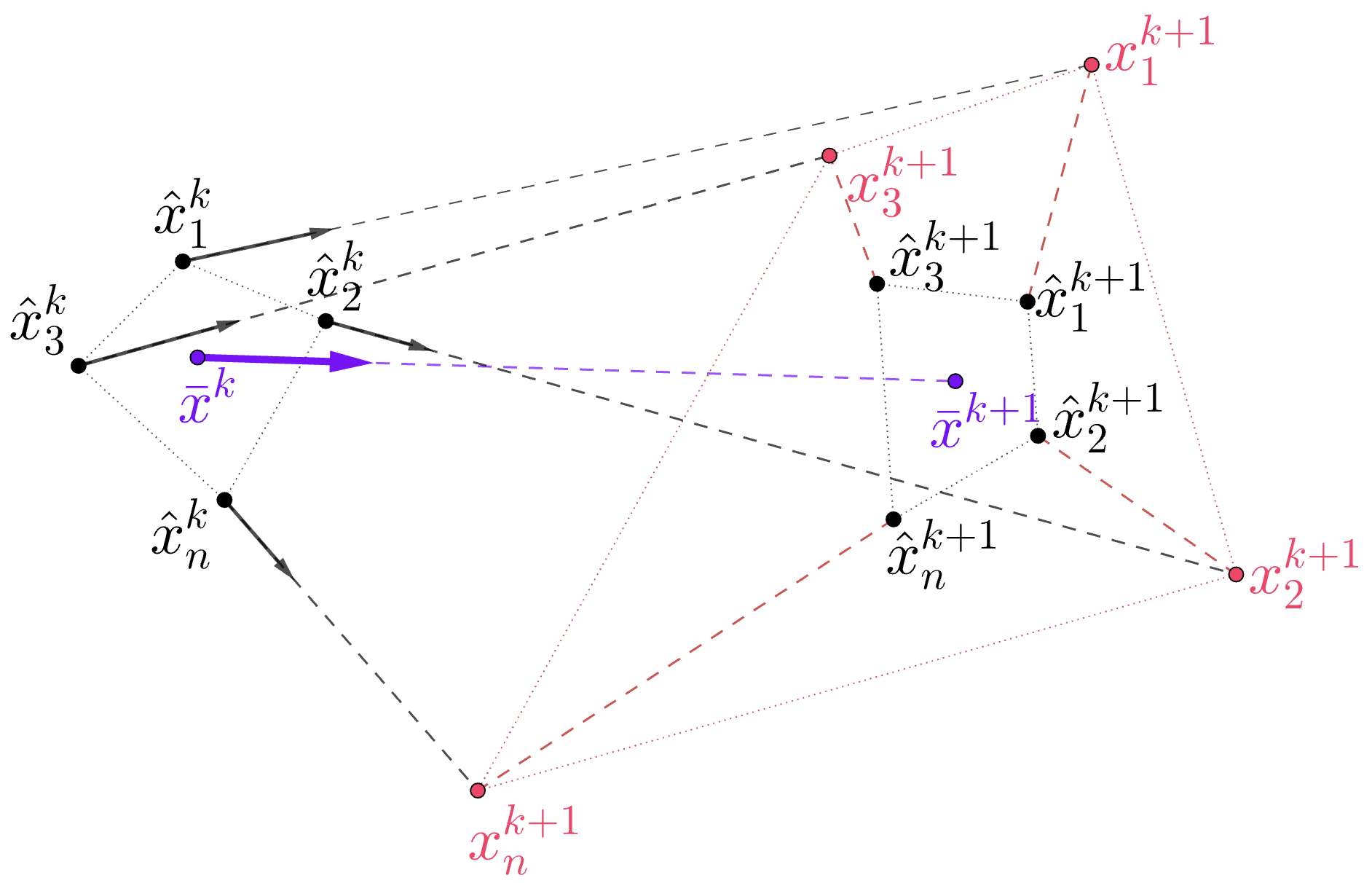} 
    \caption{Geometric interpretation of a decentralized cubic Newton step.}
    \label{fig:newton_step}
\end{figure}


An iteration of the method can be visualized as in Figure~\ref{fig:newton_step}.

Starting from approximately synchronized states $\hx_i^k$, obtained via consensus~\eqref{eq:algo_cons_x}, each node computes local first- and second-order derivatives at $\hx_i^k$ and applies consensus to obtain approximations of the global gradient $\hg_i(\hx_i^k)$ and Hessian $\hH_i(\hx_i^k)$. Using these quantities, node $i$ constructs its local inexact cubic model $\home_i^k(x; \hx_i^k)$.

Each node then independently minimizes its local model, producing the new iterate $x_i^{k+1}$ (see Figure~\ref{fig:newton_step}, red points). Since these models rely on approximate global information, the resulting points $x_i^{k+1}$ are generally different across nodes and become spatially dispersed.

This dispersion is the primary source of error introduced by decentralization. At the next iteration, the consensus step contracts these points toward their average, producing new approximately synchronized states $\hx_i^{k+1}$. Throughout this process, our theoretical analysis tracks the trajectory of the exact average $\bar{x}^k$ (purple trajectory), demonstrating that it effectively performs a global cubic Newton step despite the local approximations.

The cubic subproblem~\eqref{eq:step} is solved locally at each node using standard routines for cubic regularization~\cite{nesterov2006cubic}. In practice, one can employ implementations such as~\cite{kamzolov2024optami}.  The computational cost of solving this subproblem is typically of order $\widetilde{\cO}(\dimd^3)$ per iteration.

As a result, the method is fully decentralized: all nodes operate in parallel using only local computations and communication with neighbors.

The convergence analysis below quantifies how these consensus errors affect the quality of the local model and the descent of the global objective.

\textbf{Convergence Analysis.} We begin with a one-step descent lemma.

\begin{lemma}[Descent Lemma]
\label{lem:node_descent}
Let Assumptions~\ref{as:L1},\ref{as:L2} 
hold. 
Let $x_i^{k+1}$ be defined by~\eqref{eq:step} with parameters $\delta_{1,k},\delta_{2,k}, L$ satisfying
\begin{equation*}
    \delta_{1,k} \ge \Done{\hx}{k},\quad \delta_{2,k} \ge \Dtwo{\hx}{k}, \quad L \ge \Lavg_2,
\end{equation*}
where $\Done{\hx}{k}$ and $\Dtwo{\hx}{k}$ are defined in~\eqref{eq:agg_deltas}.
Then, for every node $i$,
\begin{equation}
\label{eq:node_descent_final}
    \begin{aligned}
        f(x_i^{k+1})
        \le
        \min_{x\in\R^\dimd}
    \left\{\right.&\left.
            f(x)
            + \tfrac{L+\Lavg_2}{6}\|x-\hx_i^k\|^3
    \right.\\ 
        &\left. 
            + \ls\gamma\delta_{1,k} + \delta_{2,k}\rs\|x-\hx_i^k\|^2 
\right\} + \tfrac{\delta_{1,k}}{\gamma}.
    \end{aligned}
\end{equation}
\end{lemma}

Next, we transform the result of Lemma~$\ref{lem:node_descent}$ into a bound on the averaged iterate $\bx^{k+1}$, which will be used to derive convergence guarantees for $f(\bx^{k+1}) - f(x^*)$. Averaging~\eqref{eq:node_descent_final} over all nodes and applying standard inequalities yields
\begin{equation}
\label{eq:avg_descent_form}
\begin{aligned}
    f(\bx^{k+1}) &\le  \tfrac{1}{\numMach}\tsum_{i=1}^\numMach f(x_i^{k+1}) 
    \le \min_{x\in\R^\dimd} \biggl\{ f(x) + \tfrac{4(L+\Lavg_2)}{6}\|x-\bx^k\|^3 \\
    & + 2(\gamma\delta_{1,k} + \delta_{2,k})\|x-\bx^k\|^2 \biggr\} + \tfrac{\delta_{1,k}}{\gamma}\\
    & + \underbrace{\tfrac{2(L+\Lavg_2)}{3}\hDex{\hx}{k}^3 + 2(\gamma\delta_{1,k} + \delta_{2,k})\hDex{\hx}{k}^2}_{\substack{\eqdef \econs_k}},
\end{aligned}
\end{equation}
where the first inequality follows from Jensen inequality due to convexity of $f$.

Based on this inequality, the process guarantees that the sequence of averaged iterates $\bx^{k+1}$ is almost monotone, i.e., monotone up to an accuracy of $\frac{\delta_{1,k}}{\gamma} + \econs_k$. Indeed, by upper bounding the minimum on the right-hand side with the value of its objective at $x = \bx^k$, we obtain:
\begin{equation*}
    f(\bx^{k+1}) \le f(\bx^k) + \tfrac{\delta_{1,k}}{\gamma} + \econs_k.
\end{equation*}
Suppose that the inaccuracy regularization parameters $\delta_{1, k}$ and consensus errors $\econs_k$ are bounded from above throughout the execution of the algorithm: $\delta_{1,k} \le \delta_1$ and $\econs_k \le \econs$ for all $k$. 
Therefore, for any iteration $\numIter \ge 1$, we have:
\begin{equation}
\label{eq:error_accumulation}
    f(\bx^\numIter) \le f(\bx^0) + \numIter \ls \tfrac{\delta_1}{\gamma} + \econs \rs.
\end{equation}
Further we show that we can control parameters $\delta_{1}$ and $\hDex{\hx}{k}$ such that $\numIter \ls \tfrac{\delta_1}{\gamma} + \econs \rs \leq \e$ (see Remark~\ref{rem:lebesgue_diameter_convex} and Remark~\ref{rem:sc_lebesgue_set_bound}).
In this regard, we define the enlarged Lebesgue set for the starting point $\bx^0$ as:
\begin{equation*}
    \mathcal{L}'(\bx^0) = \lb x \in \R^\dimd \;\middle|\; f(x) \le f(\bx^0) + \e \rb.
\end{equation*}
Let us denote by $D$ the maximum distance on this set to the global optimum $x^*$:
\begin{equation}
\label{eq:lebesgue_set_diameter}
    D = \max_{x \in \mathcal{L}'(\bx^0)} \|x - x^*\|.
\end{equation}

\begin{theorem}[Convex case]
    \label{thm:basic_delta_choice}
    Let Assumptions~\ref{as:cnvxty}, \ref{as:L1}, and \ref{as:L2} hold. Let $\numIter$ be the total number of iterations, 
    \begin{equation*}
        \delta_{1, k} = \delta_1 \ge \max_{0 \leq j \leq \numIter} \Done{\hx}{j},~~\delta_{2, k} = \delta_2 \ge \max_{0 \leq j \leq \numIter} \Dtwo{\hx}{j},
    \end{equation*}
    and $L \ge \Lavg_2$ in Algorithm~\ref{alg:basic}.
    Let $\e>0$ be the desired accuracy and let the gradient and Hessian inexactness levels be chosen for every $k = 0, \ldots, N$ as
    \begin{equation}
        \label{eq:non_acc_cons_bound_g_H_main}
        \hDeg{\hx}{k} \eqdef \hDe_{g \vert \hx } \le \tfrac{\sqrt{2}}{144}\tfrac{\e}{D}, \quad
        \hDeH{\hx}{k} \eqdef \hDe_{H \vert \hx } \le \tfrac{\sqrt{3}}{72}\sqrt{\tfrac{\e(L+\Lavg_2)}{D}}.
    \end{equation}

    Depending on the target accuracy $\e$, we choose the number of iterations $\numIter$ and the state inexactness level $\hDe_{\hx \vert k} \eqdef \hDe_{\hx}$ as follows:

    \begin{enumerate}[noitemsep,topsep=0pt,leftmargin=12pt]
        \item If $\e$ is sufficiently small, namely $\e \le 12(L+\Lavg_2)D^3$, we set:
    \begin{equation}
    \label{eq:num_iter_small_main}
        \numIter = \left\lceil \sqrt{\tfrac{108(L+\Lavg_2) D^3}{\e}} \right\rceil - 2,
    \end{equation}
    \begin{equation}
        \label{eq:non_acc_cons_bound_x_small_main}
        \hDe_{\hx} \le \min \lb
            \tfrac{\sqrt{2}\e}{288 \Lavg_1 D}, 
            \tfrac{\sqrt{3\e(L+\Lavg_2)}}{144 \Lavg_2\sqrt{D}} , 
            \tfrac{\sqrt\e}{3\sqrt{(L+\Lavg_2)D}}
        \rb.
    \end{equation}
        \item Otherwise, if $\e > 12(L+\Lavg_2)D^3$, we set $N = 1$ and:
    \begin{equation}
        \label{eq:non_acc_cons_bound_x_large_main}
        \hDe_{\hx} \le \min \lb
            \tfrac{\sqrt{2}\e}{288 \Lavg_1 D}, 
            \tfrac{\sqrt{3\e(L+\Lavg_2)}}{144 \Lavg_2\sqrt{D}}, 
            \tfrac{\e^{1/3}}{(6(L+\Lavg_2))^{1/3}}, 
            D
        \rb.
    \end{equation}
    \end{enumerate}

    We also set $\gamma = \tfrac{\sqrt{(\numIter+1)(\numIter+2)}}{6D}$. Then after $\numIter+1$ iterations Algorithm~\ref{alg:basic} outputs an $\e$-solution, i.e.,
    \begin{equation*}
        f(\bx^{\numIter+1})-f(x^*) \le \e.
    \end{equation*}
\end{theorem}

\begin{remark}
\label{rem:lebesgue_diameter_convex}
    In ~\eqref{eq:lebesgue_set_diameter} we defined $D = \max_{x \in \mathcal{L}'(\bx^0)} \|x - x^*\|$ where $\mathcal{L}'(\bx^0) = \lb x \in \R^\dimd \;\middle|\; f(x) \le f(\bx^0) + \e \rb.$ However,  we have $f(\bx^\numIter) \le f(\bx^0) + \numIter \ls \tfrac{\delta_1}{\gamma} + \econs \rs$  from~\eqref{eq:error_accumulation}. Choosing regularization parameter $\delta_1$, number of iterations $N$ and consensus errors as in Theorem \ref{thm:basic_delta_choice}, we have:
    \begin{align*}
        \numIter \ls \tfrac{\delta_1}{\gamma} + \econs \rs \leq \e
    \end{align*}
    Then, our method is monotone up to a small inaccuracy $\e$. The detailed proof is provided in \ref{prf:remark_convex_proof}
\end{remark}

Thus, if the consensus inaccuracies are controlled according to~\eqref{eq:non_acc_cons_bound_x_small_main}, Theorem~\ref{thm:basic_delta_choice} yields the iteration complexity
\begin{equation*}
     N =\cO \left( \sqrt{\tfrac{(L+\Lavg_2) D^3}{\e}} \right).
\end{equation*}
This rate matches the convergence rate of the exact Cubic Regularized Newton method~\cite{nesterov2006cubic}.

\begin{lemma}[Communication Complexity]
\label{lem:comm_complexity}
    Suppose Assumptions~\ref{assum:mixing_matrix_sequence},~\ref{as:L1} and~\ref{as:L2} hold. To guarantee the target consensus accuracies for points, gradients, and Hessians defined in\eqref{eq:non_acc_cons_bound_x_small_main} and ~\eqref{eq:non_acc_cons_bound_g_H_main}, it is sufficient to perform the following number of consensus steps at each iteration $k = 0, \ldots, N$:
    \begin{subequations}
    \label{eq:T_complexities_main}
    \begin{align}
        \Tx{k} 
        & \geq 
        \tfrac{\tau}{\lambda} \log 
            \left( 
                2D\sqrt{\numMach} \cdot \max 
                    \left\{ 
                        \tfrac{288 \Lavg_1 D}{\sqrt{2}\e}, 
                    \right. 
            \right. 
        \notag \\
        & \qquad 
            \left. 
                \left. 
                    \tfrac{144 \Lavg_2\sqrt{D}}{\sqrt{3\e(L+\Lavg_2)}}, 
                    \tfrac{3\sqrt{(L+\Lavg_2)D}}{\sqrt{\e}} 
                \right\} 
            \right), 
        \label{eq:Tx_final_main} 
        \\[2ex]
        \Tg{k} 
        & \geq 
        \tfrac{\tau}{\lambda} \log 
            \left( 
                \tfrac{144 D \sqrt{\numMach} \big( \zeta_g + 2\Lmax_1 D \big)}{\sqrt{2}\e} 
            \right), 
        \label{eq:Tg_final_main} 
        \\[2ex]
        \THh{k} 
        & \geq 
        \tfrac{\tau}{\lambda} \log 
            \left( 
                \tfrac{72\sqrt{\numMach D} \big( \zeta_H + 2\Lmax_2\sqrt{\dimd} D \big)}{\sqrt{3\e(L+\Lavg_2)}} 
            \right),
        \label{eq:TH_final_main}
    \end{align}
    \end{subequations}
    where for brevity we define $\zeta_g \eqdef \sqrt{\tfrac{1}{\numMach} \sum_{i=1}^\numMach \|\nabla f_i(x^*)\|^2}$ and $\zeta_H \eqdef \sqrt{\tfrac{1}{\numMach} \sum_{i=1}^\numMach \|\nabla^2 f_i(x^*) - \nabla^2 f(x^*)\|_F^2}$.
\end{lemma}

    Combining the iteration complexity~\eqref{eq:num_iter_small_main} with the communication bounds~\eqref{eq:T_complexities_main}, we obtain that the total communication complexity required to reach an $\e$-solution:
    \begin{align*}
        \cO(1)  \tfrac{\tau}{\lambda} \sqrt{\tfrac{(L+\Lavg_2) D^3}{\e}} d^2\log\tfrac{1}{\e}.
    \end{align*}

\begin{theorem}[Strongly Convex Case]
    \label{thm:strongly_convex_complexity}
    Let the regularization parameter be chosen as $\gamma = 1/D$. Let $\e > 0$ be the desired accuracy. Also let Assumptions~\ref{as:strcnvxty},~\ref{as:L1},~\ref{as:L2} hold.

    Let inexactness levels be chosen for every $k=0, \ldots, \numIter$ as:
    \begin{equation}
    \label{eq:sc_cons_bound_main}
    \begin{aligned}
        \hDex{\hx}{k} &\eqdef \hDe_{\hx} 
         \le 
        \min 
            \left\{ 
                \tfrac{\alpha \e}{24 \Lavg_1 D}, 
                \sqrt[3]{\tfrac{\alpha\e}{4(L+\Lavg_2)}}, 
            \right. 
        \\
        & \qquad
            \left. 
                2D \sqrt{ \tfrac{\alpha\e \Lavg_1}{3\muavg D^2\Lavg_1 + 4\alpha\e(2\Lavg_1 + D\Lavg_2)} }, 
                \tfrac{\muavg}{64(\Lavg_1/D + \Lavg_2)} 
            \right\}, 
        \\[3ex]
        \hDeg{\hx}{k} &\eqdef \hDe_{g \vert \hx} 
         \le 
        \min 
            \left\{ 
                \tfrac{\alpha \e}{12 D}, \; 
                \tfrac{\muavg D}{32} 
            \right\}, 
        \\[1.5ex]
        \hDeH{\hx}{k} &\eqdef \hDe_{H \vert \hx} 
         \le 
        \tfrac{\muavg}{16},
    \end{aligned}
    \end{equation}
    where we define $\alpha$ as:
    \begin{equation}
        \label{eq:sc_alpha_const_main}
        \alpha = \min \left\{ \tfrac{1}{2}, \sqrt{\tfrac{3\muavg}{16(L+\Lavg_2)D}} \right\}.
    \end{equation}
    Assume that the parameters of Algorithm~\ref{alg:basic} satisfy $\delta_{1, k} = \delta_1 \geq \hDe_{g\vert \hx} + 2\Lavg_1\hDe_{\hx}$, $\delta_{2, k} = \delta_2 \geq \hDe_{H\vert \hx} + 2\Lavg_2\hDe_{\hx}$, and $L \ge \Lavg_2$. 
    
    Then after $\numIter + 1$ iterations with 
    \begin{equation}
        \label{eq:sc_iterations_bound_main}
        \numIter = \left\lceil \tfrac{1}{\alpha} \log \left( \tfrac{2\bigl(f(\bx^0) - f(x^*)\bigr)}{\e} \right) \right\rceil - 1,
    \end{equation}
    the method outputs an $\e$-solution, i.e.:
    \begin{equation*}
        f(\bx^{N+1}) - f(x^*) \le \e.
    \end{equation*}
\end{theorem}

\begin{remark}
\label{rem:sc_lebesgue_set_bound}
Unlike the general analysis~\eqref{eq:error_accumulation} where the additive error accumulates linearly, in the strongly convex case it is bounded uniformly. This guarantees that our method remains monotone up to a small inaccuracy $\e$, meaning $f(\bx^N) \le f(\bx^0) + \e$. The detailed proof is deferred to Appendix~\ref{prf:sc_remark_proof}.
\end{remark}

Thus, if the consensus inaccuracies are controlled by ~\eqref{eq:sc_cons_bound_main}, Theorem~\ref{thm:strongly_convex_complexity}  yields the iteration complexity
\begin{equation}
\label{eq:sc_iteration_complexity}
    \cO(1) \left( \max \left\{ 1, \left(\tfrac{(L+\Lavg_2)D}{\muavg}\right)^{1/2} \right\} \log \left( \tfrac{f(\bx^0) - f(x^*)}{\e} \right) \right).
\end{equation}

\begin{lemma}
\label{lem:strcvx_comm_complexity}
    Suppose Assumptions~\ref{as:strcnvxty},~\ref{assum:mixing_matrix_sequence},~\ref{as:L1} and~\ref{as:L2} hold. To guarantee the target consensus accuracies for points, gradients, and Hessians defined in~\eqref{eq:sc_cons_bound_main}, it is sufficient to perform the following number of consensus steps at each iteration $k = 0, \ldots, N$:

    \begin{subequations}
    \label{eq:sc_T_complexities}
    \begin{align}
        \Tx{k} 
        & \geq 
        \tfrac{\tau}{\lambda} \log 
            \left( 
                2D\sqrt{\numMach} \cdot \max 
                    \left\{ 
                        \tfrac{24 \Lavg_1 D}{\alpha \e}, 
                        \sqrt[3]{ \tfrac{4(L+\Lavg_2)}{\alpha \e} }, 
                    \right. 
            \right. 
        \notag \\
        & \qquad 
            \left. 
                \left. 
                    \sqrt{ \tfrac{3\muavg}{4\alpha\e} + \tfrac{1}{D}\bigl(\tfrac{2}{D} + \tfrac{\Lavg_2}{\Lavg_1}\bigr) } 
                \right\} 
            \right), 
        \label{eq:sc_Tx_final} 
        \\[2ex]
        \Tg{k} 
        & \geq 
        \tfrac{\tau}{\lambda} \log 
            \left( 
                \tfrac{12 D \sqrt{\numMach} \big( \zeta_g + 2\Lmax_1 D \big)}{\alpha \e} 
            \right), 
        \label{eq:sc_Tg_final} 
        \\[2ex]
        \THh{k} 
        & \geq 
        \tfrac{\tau}{\lambda} \log 
            \left( 
                \tfrac{16\sqrt{\numMach} \big( \zeta_H + 2\Lmax_2\sqrt{\dimd} D \big)}{\muavg} 
            \right), 
        \label{eq:sc_TH_final}
    \end{align}
    \end{subequations}
    where $\alpha$ is defined in~\eqref{eq:sc_alpha_const_main}.
\end{lemma}

Combining the iteration complexity~\eqref{eq:sc_iteration_complexity} with the communication bounds~\eqref{eq:sc_T_complexities}, we obtain that the total communication round complexity required to reach an $\e$-solution:
\begin{equation*}
    \cO(1)
    \max\lb
        1,\,
        \left(\tfrac{LD}{\muavg}\right)^{1/2}
    \rb
    \tfrac{\tau}{\lambda}\log^2 \tfrac{1}{\e}
\end{equation*}
up to absolute constants.

Since each consensus round on Hessians requires transmitting $\cO(\dimd^2)$ scalars, the overall communication cost is
\begin{equation*}
    \cO(1)
    \max\lb
        1,\,
        \left(\tfrac{LD}{\muavg}\right)^{1/2}
    \rb
    \tfrac{\tau}{\lambda}\dimd^2  \log^2 \tfrac{1}{\e}.
\end{equation*}

\section{Accelerated Decentralized Cubic Newton}
\label{sec:acc_method}

In this section, we present the accelerated variant of the decentralized Cubic Newton method, listed as Algorithm~\ref{alg:acc_cubic}.
Unlike the basic method (Algorithm~\ref{alg:basic}), the accelerated scheme maintains three local sequences
$\{x_i^k\}_{k\ge 0}$, $\{y_i^k\}_{k\ge 0}$, and $\{v_i^k\}_{k\ge 0}$.
Here $x_i^k$ is the main iterate, obtained as the minimizer of the local cubic model, $y_i^k$ is the minimizer of a local estimating function, and
$v_i^k$ is an interpolation point at which the local cubic model is constructed.

At iteration $k$, each node first forms the interpolation point $v_i^k$, then runs consensus on $\{v_j^k\}_{j=1}^{\numMach}$ and obtains $\hv_i^k$.
Next, node $i$ computes the local derivatives at $\hv_i^k$ and uses consensus to approximate the averaged gradient and Hessian by
$\hg_i(\hv_i^k)$ and $\hH_i(\hv_i^k)$.
These quantities are then used to define a local inexact cubic model $\hnu^{k}_i(x; \hv^k_i)$. Following~\cite{agafonov2023inexact}, the accelerated cubic model contains only quadratic additional regularization:
\begin{equation}
\label{eq:model_2ord_acc}
    \begin{aligned}
        \hnu^{k}_i(x; \hv^k_i)
        \eqdef&
        \la \hg_i (\hv_i^k), x-\hv_i^k\ra
        + \tfrac12\la \hH_i(\hv_i^k) (x-\hv_i^k), x-\hv_i^k\ra \\
        &+ \tfrac{\delta_{2,k}}{2}\|x-\hv_i^k\|^2
        + \tfrac{L}{6}\|x-\hv_i^k\|^3.
    \end{aligned}
\end{equation}

After the cubic step is computed, the method performs one additional consensus procedure on the gradients
$\{\nabla f_j(x_j^{k+1})\}_{j=1}^{\numMach}$. This quantity is used to update the local estimating sequence $\psi_i^k(x)$, initialized as
\begin{equation}
    \label{eq:acc_psi1}
    \psi_i^1(x)
    \eqdef
    f(x_i^1) + \tfrac{\bk_{2,0}}{2}\|x - \hv_i^0\|^2 + \tfrac{\bk_{3,0}}{6}\|x - \hv_i^0\|^3,
\end{equation}
and updated according to
\begin{equation}
\label{eq:acc_psi_kp}
    \begin{aligned}
        \psi_{i}^{k+1}(x)
        \eqdef&~ \psi_{i}^{k}(x)
        + \tfrac{\bk_{2,k} - \bk_{2,k-1}}{2}\|x - \hv_i^0\|^2 \\
        &+ \tfrac{\bk_{3,k} - \bk_{3, k-1}}{6}\|x - \hv_i^0\|^3 \\
        &+ \tfrac{\alpha_k}{A_k}
       \left(
            f(x_i^{k+1}) + \la \hg_i(x_i^{k+1}), x - x_i^{k+1} \ra \right. \\
        \quad&+ \left. \tfrac{\muavg}{2}\|x - x_i^{k+1}\|^2
        \right),
    \end{aligned}
\end{equation}
where 
\begin{equation}
    \label{eq:As}
    A_0 = 1,
    ~~
    A_k = \textstyle{\prod}_{j=1}^k (1 - \alpha_j).
\end{equation}
The particular choice of coefficients $\alpha_k$ is specified later in the convergence analysis.

Although $\psi_i^k(x)$ explicitly contains function values, the algorithm only requires its minimizer.
Therefore, additive constants in $\psi_i^k(x)$ do not affect the actual updates.




\begin{algorithm}
  \caption{Accelerated Decentralized Cubic Newton}\label{alg:acc_cubic}
  \begin{algorithmic}[1]
    \STATE \textbf{Input:} initial local points $\{x_i^0\}_{i=1}^{\numMach}$; set $y_i^0 \eqdef x_i^0$ for all $i\in[\numMach]$; total number of iterations $\numIter$; $L > 0$, $\{\delta_{2,k}\}_{k = 0}^{\numIter}$, $\{\bk_{2, k}\}_{k = 0}^\numIter$, $\{\bk_{3, k}\}_{k=0}^\numIter$; sequences of consensus rounds $\{\Tz{v}{k}\}_{k=0}^{\numIter}$, $\{\Tgz{v}{k}\}_{k=0}^{\numIter}$, $\{\Tgz{x}{k}\}_{k=0}^{\numIter}$, $\{\THz{v}{k}\}_{k=0}^{\numIter}$; coefficients $\{\alpha_k\}_{k=0}^\numIter,~\{A_k\}_{k=0}^\numIter$; strong convexity parameter $\muavg \ge 0$.
    
    \STATE \textbf{Precomputation:} in parallel on each machine $i\in[\numMach]$.
    Set 
    \begin{equation*}
        v_i^0 \eqdef x_i^0.
    \end{equation*}
    Perform consensus on parameters
    \begin{equation*}
        \hv_i^0 \eqdef \mathrm{Consensus}(\{v_j^0\}_{j=1}^\numMach, \Tz{v}{0}).
    \end{equation*}
    On each node $i\in[\numMach]$, compute local derivatives $\nabla f_i(\hv_i^0), \nabla^2 f_i(
    \hv_i^0)$ and perform consensus on derivatives:
    \begin{equation*}
        \begin{gathered}
            \hg_i(\hv_i^0)
            \eqdef
            \mathrm{Consensus}\ls\{\nabla f_j(\hv_j^0)\}_{j=1}^{\numMach}, \Tgz{\hv}{0}\rs,\\
            \hH_i(\hv_i^0)
            \eqdef
            \mathrm{Consensus}\ls\{\nabla^2 f_j(\hv_j^0)\}_{j=1}^{\numMach}, \THz{\hv}{0}\rs.
        \end{gathered}
    \end{equation*}
    Compute:
    \begin{equation*}
        x_i^1
        =
        \argmin_{x\in\R^\dimd} \hnu^0_i(x; \hv_i^0).
    \end{equation*}
    Compute local gradients $\nabla f_i(x_i^{1})$ and perform consensus on gradients:
    \begin{equation*}
        \hg_i(x_i^1)
        \eqdef
        \mathrm{Consensus}\ls \{\nabla f_j(x_j^1)\}_{j=1}^{\numMach}, \Tgz{x}{0}\rs.
    \end{equation*}
    Initialize the estimating function by~\eqref{eq:acc_psi1} and compute its minimizer
    \begin{equation}
        y_i^1 = \argmin \limits_{x \in \R^\dimd} \psi_i^1(x). 
    \end{equation}

    \FOR{$k=1,\dots,\numIter$ in parallel on each machine $i\in[\numMach]$}

    \STATE Set
    \begin{equation}
        \label{eq:acc_vk}
        v_i^k = (1-\alpha_k)x_i^k + \alpha_k y_i^k.
    \end{equation}

    \STATE Perform consensus on parameters $v_i^k$:
        \begin{equation*}
            \hv_i^k \eqdef \mathrm{Consensus}(\{v_j^k\}_{j=1}^\numMach, \Tz{v}{k}).
        \end{equation*}

    \STATE Compute local derivatives $\nabla f_i(\hv_i^k), \nabla^2 f_i(\hv_i^k)$.
    \STATE Perform consensus on derivatives:
        \begin{equation*}
            \begin{gathered}
                \hg_i(\hv_i^k) \eqdef \mathrm{Consensus}(\{\nabla f_j(\hv_j^k)\}_{j=1}^\numMach, \Tgz{\hv}{k}),\\
                \hH_i(\hv_i^k) \eqdef \mathrm{Consensus}(\{\nabla^2 f_j(\hv_j^k)\}_{j=1}^\numMach, \THz{\hv}{k}).
            \end{gathered}
        \end{equation*}

    \STATE Make the local step:
        \begin{equation}
            \label{eq:acc_step}
                x_{i}^{k+1} = \argmin_{x\in\R^\dimd} \hnu_i^k(x; \hv_i^k),
        \end{equation}

    \STATE Compute local gradients $\nabla f_i(x_i^{k+1})$ and perform consensus on gradients:
    \begin{equation*}
        \hg_i(x_i^{k+1})
        \eqdef
        \mathrm{Consensus}\bigl(\{\nabla f_j(x_j^{k+1})\}_{j=1}^{\numMach}, \Tgz{x}{k}\bigr).
    \end{equation*}

    \STATE  Update the estimating function by~\eqref{eq:acc_psi_kp} and compute its minimizer
    \begin{equation*}
            y_i^{k+1} = \argmin \limits_{x \in \R^\dimd} \psi_{i}^{k+1}(x) 
    \end{equation*}
    \ENDFOR
  \end{algorithmic}
\end{algorithm}

\textbf{Convergence Guarantees and Proof Sketch.}
In this subsection, we focus on the strongly-convex setting and provide convergence guarantee for Accelerated Decentralized Cubic Newton (Algorithm~\ref{alg:acc_cubic}). Our analysis follows the estimating-sequence framework for accelerated second-order and high-order methods
\cite{nesterov2008accelerating,ghadimi2017second,agafonov2023inexact}.

The main goal is to establish lower and upper bounds on the local estimating sequences $\psi_i^k(x)$ defined in~\eqref{eq:acc_psi_kp}. In the exact single-node setting, estimating-sequence arguments typically yield bounds of the form~\cite{nesterov2008accelerating}
\begin{equation*}
    \begin{aligned}
        \tfrac{f(x^{k+1})}{A_k}
        \le
        \psi^{k+1}(y^{k+1})
        & \le
        \psi^{k+1}(x^*)\\
        & \le
        \tfrac{f(x^*)}{A_k}
        + \cO \ls L\|x^0-x^*\|^3 \rs,
    \end{aligned}
\end{equation*}
which imply
\begin{equation*}
    f(x^{k+1})-f(x^*)
    \le
    A_k \cO \ls L\|x^0-x^*\|^3 \rs.
\end{equation*}
Thus, the convergence rate is determined by the choice of the weights $A_k$. In the strongly-convex case, one may choose $A_k=(1-\alpha)^k$, which leads to a linear rate of the form
\begin{equation*}
    f(x^{k+1})-f(x^*)
    \le
    (1-\alpha)^k \cO \ls L\|x^0-x^*\|^3 \rs.
\end{equation*}
Using the $\mu$-strong convexity bound
\[
\|x^0-x^*\|^2 \le \tfrac{2}{\mu}\bigl(f(x^0)-f(x^*)\bigr),
\]
this can be rewritten as
\begin{equation*}
    f(x^{k+1})-f(x^*)
    \le
    (1-\alpha)^k\,\cO\!\left(
        \tfrac{L\bR}{\mu}\ls f(x^0)-f(x^*)\rs
    \right),
\end{equation*}
where $\bR$ is a bound on the distance to the solution. Our decentralized analysis recovers the same estimating-sequence mechanism at the local level, with additional error terms caused by consensus inaccuracies and disagreement between local iterates. 

At a high level, the proof proceeds in three steps. First, we establish upper and lower bounds on the local estimating sequences $\psi_i^k(x)$. Second, we control the additional consensus and disagreement errors. Third, we combine these results to obtain an intermediate linear-rate estimate with accumulated inexactness terms and then specialize it to explicit per-iteration conditions ensuring an $\e$-solution.

The detailed proofs are deferred to Appendix~\ref{app:accelerated_proofs}. In particular, Appendix~\ref{app:accelerated_proofs} first proves an intermediate estimate (Theorem~\ref{thm:acc_convergence}) with accumulated consensus and disagreement terms, and then derives the explicit sufficient conditions stated below in Theorem~\ref{thm:acc_delta_choice}.

To obtain an explicit accelerated convergence bound, we additionally assume that iterates remain in a bounded neighborhood of the solution. Assumptions of this type are standard in the analysis of accelerated globally convergent high-order methods with inexact gradients; see, e.g.,~\cite[Assumption~2]{agafonov2023inexact} and~\cite[(Opt)]{antonakopoulos2022extra}. In contrast, for related acceleration schemes with unbiased stochastic gradients, such a boundedness assumption is not required; see~\cite{agafonov2024advancing}.

\begin{assumption}
    \label{as:boundnes}
    Let $\{x_i^k, y_i^k, v_i^k, \hv_i^k\}_{k \geq 0}$ be generated by Algorithm~\ref{alg:acc_cubic}. Assume that there exists $\bR > 0$ such that for all $k \ge 1$,
    \begin{equation}
        \label{eq:boundnes}
        \max_{i\in[\numMach]}
        \Bigl\{
        \|x_i^k-x^*\|,
        \|y_i^k-x^*\|,
        \|\hat v_i^k-x^*\|,
        \|v_i^k-x^*\|
        \Bigr\}
        \le \bR.
    \end{equation}
\end{assumption}

We initialize the method from a common point $x_i^0 \eqdef x^0$ for all $i\in[\numMach]$. Since $v_i^0=x_i^0$ and consensus leaves identical inputs unchanged, we have $\hv_i^0=x^0$ for all $i\in[\numMach]$. Hence,
\begin{equation*}
        \|\hv_i^0-x^*\|=\|v_i^0-x^*\|=\|x^0-x^*\| \eqdef R.
\end{equation*}

We now state the main strongly-convex convergence result for the Algorithm~\ref{alg:acc_cubic}.

\begin{theorem}
\label{thm:acc_delta_choice}
Let Assumptions~\ref{as:strcnvxty},~\ref{as:L1},~\ref{as:L2},~\ref{as:boundnes} hold. Let $\e>0$ be desired accuracy, $\numIter$ be total number of iterations, and 
\begin{equation*}
\alpha = \alpha_k = \min \lb
    \tfrac45,\,
    \ls\tfrac{3\muavg\bR^{-1}}{160\Lavg_2}\rs^{1/3}
\rb.
\end{equation*}
Let inexactness levels be chosen for every $k=0,\dots,\numIter$ as
\begin{equation}
    \label{eq:acc_inacc_main}
    \begin{gathered}
            \hDeH{\hv}{k} \eqdef \hDe_{H \vert \hv} 
            \le
            \min \lb
                \tfrac{\muavg}{60\sqrt5 \alpha^2},
                \tfrac{\alpha\mumin\e}{320\Lavg_1\bR^2}
            \rb,
            \\
            \hDeg{\hv}{k} \eqdef \hDe_{g \vert \hv}
            \le
            \min \lb
                \tfrac{\alpha\mumin\e}{160\Lavg_1\bR},
                \tfrac{\alpha\e}{160\bR}
            \rb,
            \\
            \hDex{\hv}{k}  \eqdef \hDe_{\hv}
            \le
            \min \lb
                \tfrac{\muavg}{120\sqrt5 \alpha^2\Lavg_2},
                \tfrac{\alpha\e}{320\Lavg_1\bR}
            \rb,
            \\
            \hDeg{x}{k+1}  \eqdef \hDe_{g \vert x}
            \le
            \tfrac{\e}{8\bR}.
    \end{gathered}
\end{equation}
Let Algorithm~\ref{alg:acc_cubic} run with parameters chosen as
\begin{equation*}
    L = 3\Lavg_2,
    \quad 
    \delta_{2,k}=3\Dtwo{\hv}{k},
    \quad 
    \bk_{2,k}=\tfrac{\muavg}{2},
    \quad
    \bk_{3,k}=\tfrac{3}{2}\muavg\bR^{-1},
\end{equation*}
for all $k=0,\dots,\numIter$, Then after $\numIter + 1$ iterations with 
\begin{equation*}
    \numIter
    \ge
    \left\lceil
        \frac{\log\!\ls2C(f(\bx^0) - f(x^*))/\e\rs}{\log\ls1/(1-\alpha)\rs}
    \right\rceil,
\end{equation*}
where
\begin{equation*}
    C
    \eqdef
    \tfrac{4\Done{\hv}{0}}{\muavg R}
    + \tfrac{4\Done{\hv}{0}\bR}{\muavg R^2}
    + \tfrac{4\Dtwo{\hv}{0}}{\muavg}
    + \tfrac{1}{2}
    + \tfrac{8\Lavg_2 + 3\muavg\bR^{-1}}{6\muavg}R,
\end{equation*}
the method outputs an $\e$-solution, i.e.,
\begin{equation*}
    f(\bx^{\numIter+1})-f(x^*) \le \e.
\end{equation*}
\end{theorem}

Thus, if the consensus inaccuracies are controlled according to~\eqref{eq:acc_inacc_main}, Theorem~\ref{thm:acc_delta_choice} yields the iteration complexity
\begin{equation}
    \label{eq:acc_complexity}
    \numIter
    =
    \cO(1)
    \max\lb
        1,\,
        \left(\tfrac{L\bR}{\muavg}\right)^{1/3}
    \rb
    \log \tfrac{f(\bx^0)-f(x^*)}{\e}
\end{equation}
Hence, under~\eqref{eq:acc_inacc_main}, the decentralized accelerated method achieves, up to absolute constants, the same iteration complexity as the exact accelerated cubic Newton method. It remains to estimate how many consensus rounds are sufficient to guarantee~\eqref{eq:acc_inacc_main}; we do this in the next lemma.
\begin{lemma}
\label{lem:acc_comm_complexity}
    Suppose Assumptions~\ref{as:strcnvxty},~\ref{assum:mixing_matrix_sequence},~\ref{as:L1},~\ref{as:L2},~\ref{as:boundnes} hold. To guarantee the target consensus accuracies for points, gradients, and Hessians defined in~\eqref{eq:acc_inacc_main}, it is sufficient to perform the following number of consensus steps at each iteration $k = 0, \ldots, N$:
    \begin{equation}
    \label{eq:acc_T_complexities}
    \begin{aligned}
        \Tz{v}{k} 
        & \geq 
        \tfrac{\tau}{\lambda} \log 
            \ls 
                2\bR\sqrt{\numMach} 
                \max 
                    \lb 
                        \tfrac{120\sqrt{5} \alpha^2\Lavg_2}{\muavg}, 
                        \tfrac{320\Lavg_1\bR}{\alpha\e}
                    \rb 
            \rs, 
        \\
        \Tgz{\hv}{k} 
        & \geq 
        \tfrac{\tau}{\lambda} \log 
            \left( 
                \sqrt{\numMach} 
                    \ls 
                        \zeta_g + 2\Lmax_1 \bR 
                    \rs  
            \right. 
        \\
        & \qquad \qquad \qquad 
                \left. 
                    \times \max 
                        \lb 
                            \tfrac{160\Lavg_1\bR}{\alpha\mumin\e}, 
                            \tfrac{160\bR}{\alpha\e} 
                        \rb 
                    \right),
        \\
        \THz{\hv}{k} 
        & \geq 
        \tfrac{\tau}{\lambda} \log 
            \left( 
                \sqrt{\numMach} 
                    \ls 
                        \zeta_H + 2\Lmax_2\sqrt{\dimd} \bR 
                    \rs 
            \right.
        \\
        & \qquad \qquad \qquad 
            \left. 
                \times \max 
                        \lb 
                            \tfrac{60\sqrt{5} \alpha^2}{\muavg}, 
                            \tfrac{320\Lavg_1\bR^2}{\alpha\mumin\e} 
                        \rb 
            \right), 
        \\
        \Tgz{x}{k} 
        & \geq \tfrac{\tau}{\lambda} \log 
            \ls 
                \tfrac{8\bR\sqrt{\numMach} 
                    \ls 
                        \zeta_g + 2\Lmax_1 \bR 
                    \rs}{\e} 
            \rs,
    \end{aligned}
    \end{equation}
    where for brevity we define 
    \begin{gather*}
        \zeta_g \eqdef \sqrt{\tfrac{1}{\numMach} \sum_{i=1}^\numMach \|\nabla f_i(x^*)\|^2}, \\
        \zeta_H \eqdef \sqrt{\tfrac{1}{\numMach} \sum_{i=1}^\numMach \|\nabla^2 f_i(x^*) - \nabla^2 f(x^*)\|_F^2}.
    \end{gather*}
\end{lemma}

Combining the iteration complexity~\eqref{eq:acc_complexity} with the communication bounds~\eqref{eq:acc_T_complexities}, we obtain that the total communication round complexity required to reach an $\e$-solution is
\begin{equation*}
    \cO(1)
    \max\lb
        1,\,
        \left(\tfrac{L\bR}{\muavg}\right)^{1/3}
    \rb
    \tfrac{\tau}{\lambda}\log^2 \tfrac{1}{\e}
\end{equation*}
up to absolute constants. 

Consequently, the total communication cost, measured in the number of transmitted scalar values, is given by
\begin{equation*}
    \cO(1)
    \max\lb
        1,\,
        \left(\tfrac{L\bR}{\muavg}\right)^{1/3}
    \rb
    \tfrac{\tau}{\lambda}\dimd^2  \log^2 \tfrac{1}{\e}.
\end{equation*}

This quadratic dependence on the dimension $\dimd$ is due to the transmission of Hessian matrices. This observation motivates the design of more communication-efficient implementations, in particular via structured or compressed Hessian representations. Such improvements are possible, for example, in the case of generalized linear models, which we discuss in the next section.

\section{Implementation Aspects for Generalized Linear Models}
\label{sec:implementation}

The main limitation of proposed Algorithms \ref{alg:basic} and \ref{alg:acc_cubic} is the exchange of Hessians between the nodes of the network. Exchanging matrices incurs a high communication cost. A possible way to overcome communication limitations is running the method on loss functions of generalized linear models \cite{NL2021}. Consider a problem of decentralized training. Each node locally holds a dataset characterized by $\braces{a_{ij}}_{j=1}^\ell$, where $a_{ij}\in\R^d$ are feature vectors. Let $x\in\R^d$ denote model weights and let local functions have form $f_i(x) = \sum_{j=1}^\ell f_{ij}(x)$ and $f_{ij}(x) = \varphi_{ij}(a_{ij}^\top x)$. Matrices of second derivatives are computed as $\nabla^2 f_{ij}(x) = \varphi_{ij}''(a_{ij}^\top x) a_{ij} a_{ij}^\top$. Therefore, it is sufficient to once communicate datasets $\braces{a_{ij}}_{j=1}^\ell$ to the neighboring nodes and then only exchange vectors $h_i(x) = (\varphi_{i1}''(a_{i1}^\top x) \ldots \varphi_{i\ell}''(a_{i\ell}^\top x))$. Therefore, the cost of a single communication will be reduced from $\cO(d^2)$ to $\cO(\ell)$. Moreover, work \cite{NL2021} also proposed to exchange compressed vectors $\mathcal{C}(h_i(x))$, where the compression operator leaves only several nonzero components of $h_i(x)$ and zeroes out all the rest entries. Using compressed information exchange additionally decreases communication complexity.

\section{Conclusion and Discussion}

In this work, we studied consensus-based decentralized second-order optimization and proposed Cubic Regularized Newton methods for both convex and strongly convex settings. Our analysis explicitly tracks the inexactness arising from consensus errors and disagreement between local iterates, and shows that these effects can be controlled to recover the iteration complexity of exact Cubic Newton methods up to logarithmic communication overhead.

We deliberately do not focus on communication efficiency and restrict our attention to a relatively simple setting of generalized linear models, which allows us to highlight the core algorithmic and analytical ideas without additional technical overhead.

Several directions remain open. First, it would be interesting to develop communication-efficient extensions based on Hessian compression, low-rank or quasi-Newton surrogates, and lazy or randomized Hessian updates. Second, our accelerated analysis relies on a boundedness assumption on the iterates; removing this assumption would further strengthen the theory. Finally, it would be natural to extend the present framework beyond the deterministic convex setting, including stochastic objectives, nonconvex problems, and broader classes of decentralized second-order methods.

\bibliographystyle{IEEEtran}
\bibliography{agafonov}



\appendices
\section{\break Notation}
\label{app:notation}
This appendix collects the recurring notation used throughout the paper. Table~\ref{tab:notation_summary} extends Table~\ref{tab:basic_notation} by including the accelerated method, communication-complexity quantities, and auxiliary symbols that appear repeatedly in the appendix proofs. One-off constants introduced only inside a single theorem or proof are defined locally where they first appear.

\begin{table*}[!t]
\caption{Complete notation summary used throughout the paper.}
\label{tab:notation_summary}
\centering
\scriptsize
\setlength{\tabcolsep}{4pt}
\renewcommand{\arraystretch}{1.08}
\begin{tabular}{@{}
>{\raggedright\arraybackslash}p{0.17\textwidth}
>{\raggedright\arraybackslash}p{0.34\textwidth}
>{\raggedright\arraybackslash}p{0.41\textwidth}
@{}}
\toprule
Symbol & Formal definition & Description \\
\midrule
\multicolumn{3}{@{}l}{\textbf{Problem, network, and global constants}} \\
$\numMach$, $\dimd$ & --- & Number of nodes and dimension. \\
$\mathcal{G}^k$, $\mW^k$, $\mW_\tau^k$ & $\mathcal{G}^k = (\mathcal{V}, \mathcal{E}^k)$, $\mW_\tau^k \eqdef \mW^k \mW^{k-1} \cdots \mW^{k-\tau+1}$ & Communication graph, mixing matrix, and $\tau$-step mixing product used by the consensus routine. \\ 
$\lambda$, $\tau$ & see Assumption~\ref{assum:mixing_matrix_sequence} & Contraction parameters of the consensus process. \\
$f_i$, $f$, $x^*$ & $f(x) \eqdef \tfrac{1}{\numMach}\tsum_{i=1}^{\numMach} f_i(x)$,~ $x^* = \argmin_x f(x)$ & Local objectives, global objective, and global minimizer. \\
$L_{1,i}$, $L_{2,i}$, $\Lavg_1$, $\Lavg_2$, $\Lmax_1$, $\Lmax_2$ & $\Lavg_p \eqdef \tfrac{1}{\numMach}\tsum_{i=1}^{\numMach} L_{p,i}$, $\Lmax_p \eqdef \max_{i\in[\numMach]} L_{p,i}$ & Local, averaged, and maximal smoothness constants for gradients and Hessians. \\
$\mu_i$, $\muavg$, $\mumin$ & $\muavg \eqdef \tfrac{1}{\numMach}\tsum_{i=1}^{\numMach} \mu_i$, $\mumin \eqdef \min_{i\in[\numMach]} \mu_i$ & Local, averaged, and minimal strong-convexity constants. \\
$D$, $R$, $\bR$ & see~\eqref{eq:lebesgue_set_diameter} and Assumption~\ref{as:boundnes} & Radius bounds used in the non-accelerated and accelerated analyses. \\
$\zeta_g$, $\zeta_H$ & $\zeta_g \eqdef \sqrt{\tfrac{1}{\numMach} \sum_{i=1}^\numMach \|\nabla f_i(x^*)\|^2}$ and $\zeta_H \eqdef \sqrt{\tfrac{1}{\numMach} \sum_{i=1}^\numMach \|\nabla^2 f_i(x^*) - \nabla^2 f(x^*)\|_F^2}$ & Gradient and Hessian heterogeneity constants at the optimum. \\
\midrule
\multicolumn{3}{@{}l}{\textbf{Iterates, averages, and local models}} \\
$x_i^k$, $\bx^k$, $\hx_i^k$ & $\bx^k \eqdef \tfrac{1}{\numMach}\tsum_{i=1}^{\numMach} x_i^k$, $\hx_i^k \eqdef \mathrm{Consensus}(\{x_j^k\}_{j=1}^{\numMach}, \Tx{k})$ & Local iterate, exact network average, and approximate averaged iterate. \\
$y_i^k$, $v_i^k$, $\hv_i^k$, $\bv^k$ & $v_i^k = (1-\alpha_k)x_i^k + \alpha_k y_i^k$, $\hv_i^k \eqdef \mathrm{Consensus}(\{v_j^k\}_{j=1}^{\numMach}, \Tz{v}{k})$, $\bv^k \eqdef \tfrac{1}{\numMach}\tsum_{i=1}^{\numMach} v_i^k$ & Auxiliary iterates of the accelerated method, their consensus approximation, and the corresponding exact network average. \\
$\hg_i(\hx_i^k)$, $\hH_i(\hx_i^k)$ & consensus outputs for $\{\nabla f_j(\hx_j^k)\}_{j=1}^{\numMach}$ and $\{\nabla^2 f_j(\hx_j^k)\}_{j=1}^{\numMach}$ & Approximate averaged gradient and Hessian in the basic method. \\
$\hg_i(\hv_i^k)$, $\hH_i(\hv_i^k)$, $\hg_i(x_i^{k+1})$ & consensus outputs for $\{\nabla f_j(\hv_j^k)\}_{j=1}^{\numMach}$, $\{\nabla^2 f_j(\hv_j^k)\}_{j=1}^{\numMach}$, and $\{\nabla f_j(x_j^{k+1})\}_{j=1}^{\numMach}$ & Approximate averaged derivatives in the accelerated method. \\
$\phi_i(y;x)$, $\phi(y;x)$ & second-order Taylor polynomials of $f_i$ and $f$ at $x$ & Exact local and global Taylor models. \\
$\hphi_i^k(x;\hx_i^k)$, $\hphi_i^k(x;\hv_i^k)$ & built from consensus derivatives at $\hx_i^k$ or $\hv_i^k$ & Inexact Taylor approximations used in the basic and accelerated analyses. \\
$\home_i^k(x;\hx_i^k)$, $\hnu_i^k(x;\hv_i^k)$, $\psi_i^k(x)$ & see~\eqref{eq:model_2ord},~\eqref{eq:model_2ord_acc},~\eqref{eq:acc_psi1}, and~\eqref{eq:acc_psi_kp} & Basic and accelerated inexact cubic models, and the accelerated estimating sequence. \\
$s_i^k$, $M_i^k$ & $s_i^k \eqdef x_i^{k+1}-\hv_i^k$, $M_i^k \eqdef \hH_i(\hv_i^k) + \delta_{2,k} I$ & Auxiliary proof quantities used in the accelerated analysis. \\
\midrule
\multicolumn{3}{@{}l}{\textbf{Consensus counters, errors, and algorithmic parameters}} \\
$\Tx{k}$, $\Tg{k}$, $\THh{k}$ & --- & Numbers of consensus rounds for iterates, gradients, and Hessians in Algorithm~\ref{alg:basic}. \\
$\Tz{
\hv}{k}$, $\Tgz{\hv}{k}$, $\Tgz{x}{k}$, $\THz{v}{k}$ & --- & Numbers of consensus rounds for $v$-iterates, gradients at $\hv_i^k$, gradients at $x_i^{k+1}$, and Hessians at $\hv_i^k$ in Algorithm~\ref{alg:acc_cubic}. \\
$\hDex{\hx}{k}$, $\hDeg{\hx}{k}$, $\hDeH{\hx}{k}$ & see~\eqref{eq:cons_bounds_x} and~\eqref{eq:cons_bounds_derivatives} & Consensus errors for points, gradients, and Hessians at $\hx_i^k$. \\
$\hDex{\hv}{k}$, $\hDeg{\hv}{k}$, $\hDeH{\hv}{k}$, $\hDeg{x}{k+1}$ & see~\eqref{eq:cons_bounds_derivatives} and the accelerated appendix bounds & Consensus errors for points $\hv_i^k$ and for gradients at $x_i^{k+1}$. \\
$\Done{\hx}{k}$, $\Dtwo{\hx}{k}$, $\Done{\hv}{k}$, $\Dtwo{\hv}{k}$, $\Dtwoo{\hv}$ & $\Done{z}{k} \eqdef \hDeg{z}{k} + 2\Lavg_1\hDex{z}{k}$, $\Dtwo{z}{k} \eqdef \hDeH{z}{k} + 2\Lavg_2\hDex{z}{k}$, $\Dtwoo{\hv} \eqdef \max_{0 \le k \le \numIter}\Dtwo{\hv}{k}$ & Aggregated first- and second-order Taylor-model errors and their maximal accelerated version. \\
$\Delta_{x,k}^{\mathrm{raw}}$, $\econs_k$, $\econs$ & $\Delta_{x,k}^{\mathrm{raw}} \eqdef \max_{i,j\in[\numMach]}\|x_i^k-x_j^k\|$ & Raw disagreement measure and the consensus-induced additive error terms used in the non-accelerated descent bounds. \\
$\delta_{1,k}$, $\delta_{2,k}$, $L$, $\gamma$ & --- & Smoothing and regularization parameters in the local cubic models. \\
$\alpha_k$, $A_k$, $\bk_{2,k}$, $\bk_{3,k}$, $\numIter$ & $A_0 = 1$, $A_k = \prod_{j=1}^{k}(1-\alpha_j)$ & Accelerated weights, estimating-sequence coefficients, and the total number of outer iterations. \\
\bottomrule
\end{tabular}
\end{table*}

\section{\break Proofs}
\label{app:proofs}

\subsection{Proofs for the Decentralized Cubic Newton}
\label{app:decentralized_cubic_newton_proofs}

\begin{customlemma}{\ref{lem:inexact_taylor}}
Let Assumptions~\ref{as:L1},~\ref{as:L2} hold. Let $\hg_i(\hx_i^k)$ and $\hH_i(\hx_i^k)$ be the local gradient and Hessian approximations obtained after the consensus procedure at node $i \in [\numMach]$, with consensus errors defined in~\eqref{eq:cons_bounds_x},~\eqref{eq:cons_bounds_derivatives} applied at $z_i^k = \hx_i^k$. Define the aggregated consensus errors 
\begin{equation}
    \label{eq:agg_deltas_app}
    \begin{aligned}
       \Done{\hx}{k} &\eqdef \hDeg{\hx}{k} + 2\Lavg_1\hDex{\hx}{k}, \\
       \Dtwo{\hx}{k} &\eqdef \hDeH{\hx}{k} + 2\Lavg_2\hDex{\hx}{k}.
    \end{aligned}
\end{equation}
Then, for any $x\in\R^\dimd$,
\begin{equation}
    \label{eq:inexact_taylor_value_app}
    \begin{aligned}
        \bigl| f(x) - \hphi_i^k(x; \hx_i^k) \bigr| \le~ &\Done{\hx}{k}\|x-\hx_i^k\| + \tfrac{\Dtwo{\hx}{k}}{2} \|x-\hx_i^k\|^2\\
        & + \tfrac{\Lavg_2}{6}\|x-\hx_i^k\|^3  .
    \end{aligned}
\end{equation}
Moreover,
\begin{equation}
\label{eq:inexact_taylor_grad_app}
\begin{aligned}
    \|\nabla f(x) -  \nabla \hphi_i^k(x; \hx_i^k)\|
    \le~&\Done{\hx}{k} +\Dtwo{\hx}{k} \|x-\hx_i^k\| \\
    &+ \tfrac{\Lavg_2}{2}\|x-\hx_i^k\|^2
\end{aligned}
\end{equation}
\end{customlemma}

\begin{proof}
    Denote $h_i^k \eqdef x-\hx_i^k$.
    By~\eqref{eq:taylor_bound_global},
    \begin{equation}
    \label{eq:global_taylor_at_hx}
        \bigl|f(x) - \phi(x;\hx_i^k)\bigr| \le \tfrac{\Lavg_2}{6}\|h_i^k\|^3.
    \end{equation}
    Let us denote 
    \begin{equation*}
        \bar g_x^k \eqdef \tfrac{1}{\numMach}\tsum_{j=1}^{\numMach}\nabla f_j(\hx_j^k),
        \qquad
        \bar H_x^k \eqdef \tfrac{1}{\numMach}\tsum_{j=1}^{\numMach}\nabla^2 f_j(\hx_j^k).
    \end{equation*}
    Next, add and subtract the average derivatives:
    \begin{equation*}
        \begin{aligned}
            f(x)  - & \hphi_i^k(x; \hx_i^k)  \\
            \stackrel{\eqref{eq:inexact_taylor_2ord}}{=} &~ f(x) - \ls f(\hx_i^k) 
            + \la \hg_i(\hx_i^k), h_i^k\ra 
            + \tfrac12\la \hH_i(\hx_i^k) h_i^k, h_i^k\ra\rs \\
            =&~f(x)-\phi(x;\hx_i^k)
            + \la \nabla f(\hx_i^k)-\hg_i(\hx_i^k), h_i^k\ra \\
            & + \tfrac12\la \ls\nabla^2 f(\hx_i^k)
            - \hH_i(\hx_i^k)\rs h_i^k, h_i^k\ra \\
            =&~ f(x)-\phi(x;\hx_i^k) \\
            & + \underbrace{\la \nabla f(\hx_i^k)-\bar g_x^k, h_i^k\ra
            + \tfrac12\la \ls\nabla^2 f(\hx_i^k)-\bar H_x^k\rs h_i^k, h_i^k\ra}_{\text{dissimilarity due to disagreement in }\hx_i^k}\\
            & + \underbrace{\la \bar g_x^k-\hg_i(\hx_i^k), h_i^k\ra 
            + \tfrac12\la \ls\bar H_x^k-\hH_i(\hx_i^k)\rs h_i^k, h_i^k\ra}_{\text{consensus errors in derivatives}}.
        \end{aligned}
    \end{equation*}
    By the consensus error bounds~\eqref{eq:cons_bounds_derivatives},
    \begin{equation}
        \label{eq:cons_bounds_in_lemma}
        \begin{gathered}
            \bigl|\la \bar g_x^k-\hg_i(\hx_i^k), h_i^k\ra\bigr| \le \hDeg{\hx}{k}\|h_i^k\| \\
            ~
            \tfrac12\bigl|\la \ls\bar H_x^k-\hH_i(\hx_i^k)\rs h_i^k, h_i^k\ra\bigr|
            \le \tfrac{\hDeH{\hx}{k}}{2}\|h_i^k\|^2.
        \end{gathered}
    \end{equation}
    For the disagreement terms, note that
    \begin{equation*}
        \begin{gathered}
            \nabla f(\hx_i^k) - \bar g_x^k
            = \tfrac{1}{\numMach}\tsum_{j=1}^\numMach \ls\nabla f_j(\hx_i^k)-\nabla f_j(\hx_j^k)\rs \\
            \nabla^2 f(\hx_i^k) - \bar H_x^k
            = \tfrac{1}{\numMach} \tsum_{j=1}^\numMach \ls\nabla^2 f_j(\hx_i^k)-\nabla^2 f_j(\hx_j^k)\rs.
        \end{gathered}
    \end{equation*}
    Using Assumption~\ref{as:L1} and $\|\hx_i^k-\hx_j^k\|\le 2\hDex{\hx}{k}$ by~\eqref{eq:cons_bounds_x}, we obtain
    \begin{equation}
    \label{eq:grad_disagreement}
    \begin{aligned}
        \|\nabla f(\hx_i^k)-\bar g_x^k\|
        &\le \tfrac{1}{\numMach}\tsum_{j=1}^\numMach L_{1,j}\|\hx_i^k-\hx_j^k\| \\
        &\le \tfrac{1}{\numMach}\tsum_{j=1}^\numMach 2L_{1,j}\hDex{\hx}{k}
        \le 2\Lavg_1\hDex{\hx}{k}.
    \end{aligned}
    \end{equation}
    Similarly, by Assumption~\ref{as:L2},
    \begin{equation}
    \label{eq:hess_disagreement}
    \begin{aligned}
        \opnorm{\nabla^2 f(\hx_i^k)-\bar H_x^k}
        &\le \tfrac{1}{\numMach}\tsum_{j=1}^\numMach L_{2,j}\|\hx_i^k-\hx_j^k\| \\
        & \le \tfrac{1}{\numMach}\tsum_{j=1}^\numMach 2L_{2,j}\hDex{\hx}{k}
        = 2\Lavg_2\hDex{\hx}{k}.
    \end{aligned}
    \end{equation}
    Combining~\eqref{eq:global_taylor_at_hx}, \eqref{eq:cons_bounds_in_lemma}, \eqref{eq:grad_disagreement}, and~\eqref{eq:hess_disagreement} yields~\eqref{eq:inexact_taylor_value_app}.
    For~\eqref{eq:inexact_taylor_grad_app}, add and subtract exact derivatives at $\hx_i^k$:
    \begin{equation*}
    \begin{aligned}
        \|  \nabla f(x) - \nabla \hphi_i^k  ( x;  \hx_i^k)  \| 
        & =    \|\nabla f(x) - \ls\hg_i(\hx_i^k) + \hH_i(\hx_i^k) h_i^k\rs\| \\
        \le~ & \|\nabla f(x) - \ls\nabla f(\hx_i^k) + \nabla^2 f(\hx_i^k)h_i^k\rs\| \\
        & + \|\nabla f(\hx_i^k)-\hg_i(\hx_i^k)\| \\
        & + \opnorm{\nabla^2 f(\hx_i^k)-\hH_i(\hx_i^k)}\,\|h_i^k\|.
    \end{aligned}
    \end{equation*}
    By the global $\Lavg_2$-Lipschitzness of the Hessian~\eqref{eq:taylor_grad_bound_global},
    \begin{equation*}
        \|\nabla f(x) - \ls\nabla f(\hx_i^k) + \nabla^2 f(\hx_i^k)h_i^k\rs\| \le \tfrac{\Lavg_2}{2}\|h_i^k\|^2.
    \end{equation*}
    Also,
    \begin{equation*}
        \begin{aligned}
            \|\nabla f(\hx_i^k)-\hg_i(\hx_i^k)\|
            & \le \|\nabla f(\hx_i^k)-\bar g_x^k\| + \|\bar g_x^k-\hg_i(\hx_i^k)\| \\
            & \le 2\Lavg_1\hDex{\hx}{k} + \hDeg{\hx}{k}.
        \end{aligned}
    \end{equation*}
    Finally,
    \begin{equation*}
    \begin{aligned}
        \opnorm{\nabla^2 f(\hx_i^k)-\hH_i(\hx_i^k)}
        \le~ & \opnorm{\nabla^2 f(\hx_i^k)-\bar H_x^k} \\
        & + \opnorm{\bar H_x^k-\hH_i(\hx_i^k)} \\
        \le~ & 2\Lavg_2\hDex{\hx}{k} + \hDeH{\hx}{k}.
    \end{aligned}
\end{equation*}
    Combining the three bounds gives~\eqref{eq:inexact_taylor_grad_app}.
\end{proof}

\begin{customlemma}{\ref{lem:node_descent}}
Let Assumptions~\ref{as:L1},\ref{as:L2} hold. 
Let $x_i^{k+1}$ be defined by~\eqref{eq:step} with parameters $\delta_{1,k},\delta_{2,k}, L$ satisfying
\begin{equation}
    \label{eq:params}
    \delta_{1,k} \ge \Done{\hx}{k},\quad \delta_{2,k} \ge \Dtwo{\hx}{k}, \quad L \ge \Lavg_2,
\end{equation}
where $\Done{\hx}{k}$ and $\Dtwo{\hx}{k}$ are defined in~\eqref{eq:agg_deltas}.
Then, for every node $i$,
\begin{equation}
\label{eq:node_descent_final_app}
    \begin{aligned}
        f(x_i^{k+1})
        \le
        \min_{x\in\R^\dimd}
    \left\{ \right.& \left.
            f(x)
            + \tfrac{L+\Lavg_2}{6}\|x-\hx_i^k\|^3
    \right.\\ 
        &\left. 
            + \ls\gamma\delta_{1,k} + \delta_{2,k}\rs\|x-\hx_i^k\|^2 
    \right\} + \tfrac{\delta_{1,k}}{\gamma}.
    \end{aligned}
\end{equation}
\end{customlemma}

\begin{proof}
    Let us denote $s_i^k \eqdef x_i^{k+1} - \hx_i^k$. By Lemma~\ref{lem:inexact_taylor} with $y=x_i^{k+1}$,
\begin{equation}
\label{eq:majorization_step_app}
\begin{aligned}
    f(x_i^{k+1}) 
    \stackrel{\eqref{eq:inexact_taylor_value}}{\le}~
    &\hphi_i^k(x_i^{k+1}; \hx_i^k) 
    + \Done{\hx}{k}\|s_i^k\|  + \tfrac{\Dtwo{\hx}{k}}{2}\|s_i^k\|^2 \\
    &+ \tfrac{\Lavg_2}{6}\|s_i^k\|^3 \\
    \stackrel{\eqref{eq:params}}{\le}~
    &f(\hx_i^k)
    + \la \hg_i(\hx_i^k), s_i^k\ra
    + \tfrac12\la \hH_i(\hx_i^k) s_i^k, s_i^k\ra \\
    &+ \delta_{1,k}\|s_i^k\|
    + \tfrac{\delta_{2,k}}{2}\|s_i^k\|^2
    + \tfrac{L}{6}\|s_i^k\|^3 \\
    \le~ &
    f(\hx_i^k)
    + \la \hg_i(\hx_i^k), s_i^k\ra
    + \tfrac12\la \hH_i(\hx_i^k) s_i^k, s_i^k\ra \\
    &+ \tfrac{\delta_{1,k}}{2\gamma} + \tfrac{\gamma\delta_{1,k}}{2}\|s_i^k\|^2
    + \tfrac{\delta_{2,k}}{2}\|s_i^k\|^2
    + \tfrac{L}{6}\|s_i^k\|^3 \\
    \stackrel{\eqref{eq:model_2ord}}{=}~& f(\hx_i^k) + \home_i^k\ls x_i^{k+1}; \hx_i^k\rs \\
    \stackrel{\eqref{eq:step}}{=}~& \min_{x\in\R^\dimd} \lb f(\hx_i^k) + \home_i^k\ls x; \hx_i^k\rs\rb,
\end{aligned}
\end{equation}
where the third inequality follows from Young's inequality $\delta_{1,k}\|s_i^k\| \le \frac{\delta_{1,k}}{2\gamma} + \frac{\gamma\delta_{1,k}}{2}\|s_i^k\|^2$.

Now take any $x\in\R^\dimd$ and denote $h_i^k \eqdef x-\hx_i^k$.
By Lemma~\ref{lem:inexact_taylor}, we have
\begin{equation*}
    \begin{aligned}
        \hphi_i^k(x; \hx_i^k) \stackrel{\eqref{eq:inexact_taylor_2ord}}{=} &  ~f(\hx_i^k) + \la \hg_i(\hx_i^k), h_i^k\ra + \tfrac12\la \hH_i(\hx_i^k) h_i^k, h_i^k\ra \\
        \stackrel{\eqref{eq:inexact_taylor_value}}{\le} & ~f(x) 
        + \Done{\hx}{k}\|h_i^k\| 
        + \tfrac{\Dtwo{\hx}{k}}{2}\|h_i^k\|^2 \\
        & + \tfrac{\Lavg_2}{6}\|h_i^k\|^3  \\
        \stackrel{\eqref{eq:params}}{\le} &
        ~f(x) 
        + \delta_{1,k}\|h_i^k\|
        + \tfrac{\delta_{2,k}}{2}\|h_i^k\|^2 
        + \tfrac{\Lavg_2}{6}\|h_i^k\|^3.
     \end{aligned}
\end{equation*}
Applying Young's inequality $\delta_{1,k}\|h_i^k\| \le \tfrac{\delta_{1,k}}{2\gamma} + \tfrac{\gamma\delta_{1,k}}{2}\|h_i^k\|^2$ and adding $\frac{\delta_{1,k}}{2\gamma} + \frac{\gamma\delta_{1,k} + \delta_{2,k}}{2}\|h_i^k\|^2 + \tfrac{L}{6}\|h_i^k\|^3$ to both sides gives
\begin{equation*}
    \begin{aligned}
        f(\hx_i^k) + \home_i^k\ls x;\hx_i^k\rs
        \le~ &
        f(x) + \left(\gamma\delta_{1,k} + \delta_{2,k}\right)\|h_i^k\|^2 
        \\
        &+ \tfrac{L+\Lavg_2}{6}\|h_i^k\|^3 + \tfrac{\delta_{1,k}}{\gamma}.
    \end{aligned}
\end{equation*}
Since the inequality holds for all $x$, taking the minimum over $x\in\R^\dimd$ and combining with~\eqref{eq:majorization_step_app} yields~\eqref{eq:node_descent_final_app}.
\end{proof}
\textbf{Proof of~\eqref{eq:avg_descent_form}.}
\begin{proof}
Since $\|x - \hx_i^k\| \le \|x - \bar{x}^i_k\| + \|\hx_i^k - \bar{x}^i_k\|$, we have:
\begin{equation}
\label{eq:avgg_descent_form}
\begin{aligned}
    \tfrac{1}{\numMach}\tsum_{i=1}^\numMach f(x_i^{k+1})  &\le \min_{x\in\R^\dimd} \biggl\{ f(x) + \tfrac{4(L+\Lavg_2)}{6}\|x-\bx^k\|^3 \\
    & + 2(\gamma\delta_{1,k} + \delta_{2,k})\|x-\bx^k\|^2 \biggr\} + \tfrac{\delta_{1,k}}{\gamma}\\
    & + \tfrac{2(L+\Lavg_2)}{3}\hDex{\hx}{k}^3 + 2(\gamma\delta_{1,k} + \delta_{2,k})\hDex{\hx}{k}^2,
\end{aligned}
\end{equation}
where we used:
\begin{equation*}
    \begin{gathered}
        \|a+b\|^3 \le 4\|a\|^3 + 4\|b\|^3,\\
        \|a+b\|^2 \le 2\|a\|^2 + 2\|b\|^2,\\
        \|a+b\| \le \|a\|+\|b\|,
    \end{gathered}
\end{equation*}

By convexity of $f$, Jensen's inequality gives
\begin{equation}
\label{eq:jensen_step}
    f(\bar x^{k+1}) \le \tfrac{1}{\numMach}\textstyle{\sum}_{i=1}^\numMach f(x_i^{k+1}),
\end{equation}
and combining~\eqref{eq:avgg_descent_form} with~\eqref{eq:jensen_step} yields ~\eqref{eq:avg_descent_form}.
\end{proof}


\begin{theorem}
\label{thm:convex_case_theorem}

    Let Assumptions~\ref{as:cnvxty},~\ref{as:L1},~\ref{as:L2} hold. Let $\numIter$ be the total number of iterations, 
    \begin{equation*}
        \delta_{1, k} = \delta_1 \ge \max_{0 \leq j \leq \numIter} \Done{\hx}{j},~~\delta_{2, k} = \delta_2 \ge \max_{0 \leq j \leq \numIter} \Dtwo{\hx}{j},
    \end{equation*}
    and $L \ge \Lavg_2$ in Algorithm~\ref{alg:basic}. Also define 
    \begin{equation*}
        \hDe_{\hx} = \max_{0 \leq j \leq \numIter} \hDex{\hx}{j}
    \end{equation*}
    Then, for any $\gamma > 0$, after $N + 1$ iterations of Algorithm~\ref{alg:basic}, we have the following:
    \begin{align*}
        \notag &f(\bx^{\numIter+1})  - f(x^*) 
        \le \delta_1 (\numIter+4) \ls \tfrac{9 D^2}{(\numIter+1)(\numIter+2)} \gamma + \tfrac{1}{4\gamma} \rs \\
        \notag&\qquad \qquad + 9 \tfrac{\numIter+4}{(\numIter+1)(\numIter+2)} \delta_2 D^2+ \tfrac{18}{(\numIter+2)(\numIter+3)} (L+\Lavg_2) D^3 \\
        &\qquad \qquad + \tfrac{\numIter+4}{4} \left( \tfrac{2(L+\Lavg_2)}{3}\hDe_{\hx}^3  + 2(\gamma\delta_1 + \delta_2)\hDe_{\hx}^2 \right).
    \end{align*}
    By introducing $\cO$-notation and fixing $\gamma = \cO(\numIter/D)$, we get:
    \begin{equation*}
        \begin{aligned}
            f(\bx^{N+1}) &- f(x^*) 
            \leq 
            \cO \left( \delta_1 D + \tfrac{\delta_2 D^2}{N} + \tfrac{(L+\Lavg_2) D^3}{N^2}\right. \\
            & + 
            \left. (L + \Lavg_2)N \hDe_{\hx}^3 + \tfrac{N^2\delta_1\hDe_{\hx}^2}{D} + N\delta_2\hDe_{\hx}^2  \right)        
        \end{aligned}
    \end{equation*}
\end{theorem}

\begin{proof}

    From ~\eqref{eq:avg_descent_form}, since $f$ is convex and $x^*$ is a minimizer, it is enough to consider points on the segment between $x^k$ and $x^*$. For any $k\geq0$ we obtain:
    \begin{align*}
        f(\bx^{k+1})
        &\le \min_{\alpha_k \in [0,1]}
            \biggl\{
                f\bigl(\bx^k + \alpha_k(x^* - \bx^k)\bigr) \notag \\
            & + 2\bigl(\gamma\delta_1 + \delta_2\bigr)\|\alpha_k(x^* - \bx^k)\|^2 \notag \\
            & + \tfrac{2(L+\Lavg_2)}{3}\|\alpha_k(x^* - \bx^k)\|^3
            \biggr\} + \tfrac{\delta_1}{\gamma} \notag \\
            & +  \tfrac{2(L+\Lavg_2)}{3}\hDex{\hx}{k}^3  + 2(\gamma\delta_1 + \delta_2)\hDex{\hx}{k}^2  \notag \\ 
        &\le \min_{\alpha_k \in [0,1]}
            \biggl\{
                (1 - \alpha_k)f(\bx^k) + \alpha_k f(x^*) \notag \\
            &+ 2\bigl(\gamma\delta_1 + \delta_2\bigr)\alpha_k^2 D^2 + \tfrac{2(L+\Lavg_2)}{3}\alpha_k^3 D^3
            \biggr\} \notag \\
            &+ \tfrac{\delta_1}{\gamma} +  \tfrac{2(L+\Lavg_2)}{3}\hDex{\hx}{k}^3 + 2(\gamma\delta_1 + \delta_2)\hDex{\hx}{k}^2 ,
    \end{align*}
    Subtracting $f(x^*)$ from both sides of inequality above, we obtain the following one-step recursion valid for any $\alpha_k \in [0,1]$:
    \begin{equation}
        \label{eq:one_step_recursion}
        \begin{aligned}
            f(\bx^{k+1}) & - f(x^*) 
            \le (1 - \alpha_k)\bigl(f(\bx^k) - f(x^*)\bigr) \\
            & + 2\ls\gamma\delta_1 + \delta_2\rs D^2 \alpha_k^2 + \tfrac{2(L+\Lavg_2)}{3} D^3 \alpha_k^3 + \tfrac{\delta_1}{\gamma} \\
            & +  \tfrac{2(L+\Lavg_2)}{3}\hDex{\hx}{k}^3 + 2(\gamma\delta_1 + \delta_2)\hDex{\hx}{k}^2
        \end{aligned}
    \end{equation}
    To unroll this recursion, we choose the sequence $\alpha_k = \tfrac{3}{k+3} \in (0, 1]$ and multiply both sides of (31) by $A_k \stackrel{\mathrm{def}}{=} (k+1)(k+2)(k+3)$. Note that $(1 - \alpha_k)A_k = \tfrac{k}{k+3}A_k = k(k+1)(k+2) = A_{k-1}$. Also define $\alpha_0 = 1$ and $A_{-1} \stackrel{\mathrm{def}}{=} A_0(1 - \alpha_0) = 0$.
    Using the upper bounds $\tfrac{k+1}{k+3} \leq 1$, $\tfrac{k+2}{k+3} \leq 1$ and $(k+2) < (k+3)$, we get:
    \begin{align*}
        A_k \alpha_k^2 &= \tfrac{9(k+1)(k+2)(k+3)}{(k+3)^2} = \tfrac{9(k+1)(k+2)}{k+3} \leq 9(k+3), \\
        A_k \alpha_k^3 &= \tfrac{27(k+1)(k+2)(k+3)}{(k+3)^3} = \tfrac{27(k+1)(k+2)}{(k+3)^2} \leq 27.
    \end{align*}  

    Hence, we obtain:
    \begin{align}
    \label{eq:A_N_A_0}
        \notag &A_{\numIter} \ls f(\bx^{\numIter+1}) - f(x^*) \rs
        \le A_0(1-\alpha_0)\ls f(\bx^0) - f(x^*) \rs \\
        \notag& + 18\ls\gamma\delta_1 + \delta_2\rs D^2 \tsum_{k=0}^\numIter (k+3) + 18(L+\Lavg_2) D^3 (\numIter+1) \\
        & + \tfrac{\delta_1}{\gamma} \tsum_{k=0}^\numIter A_k + \tsum_{k=0}^\numIter A_k \biggl\{ \tfrac{2(L+\Lavg_2)}{3}\hDex{\hx}{k}^3 + 2(\gamma\delta_1 + \delta_2)\hDex{\hx}{k}^2 \biggr\}. 
    \end{align}
    Note that $\tsum_{k=0}^N A_k = \tsum_{k=0}^N (k+1)(k+2)(k+3) = \tfrac{(N+1)(N+2)(N+3)(N+4)}{4}$ and $\tsum_{k=0}^{N}(k+3) = \tfrac{(N+1)(N+6)}{2}$. Then, dividing by $A_{\numIter}$ both sides of ~\eqref{eq:A_N_A_0} and using $\hDex{\hx}{k} \le \hDe_{\hx}$:

    \begin{align*}
        \notag &f(\bx^{\numIter+1}) - f(x^*) 
        \le \delta_1 (\numIter+4) \ls \tfrac{9 D^2}{(\numIter+1)(\numIter+2)} \gamma + \tfrac{1}{4\gamma} \rs \\
        &+ 9 \tfrac{\numIter+4}{(\numIter+1)(\numIter+2)} \delta_2 D^2 + \tfrac{18}{(\numIter+2)(\numIter+3)} (L+\Lavg_2) D^3 \notag\\ 
         &+ \tfrac{\numIter+4}{4} \biggl\{ \tfrac{2(L+\Lavg_2)}{3}\hDe_{\hx}^3 + 2(\gamma\delta_1 + \delta_2)\hDe_{\hx}^2 \biggr\}.
    \end{align*}
    Fixing $\gamma = \tfrac{\sqrt{(\numIter+1)(\numIter+2)}}{6D}$, we arrive at the corresponding convergence rate:
    \begin{align}
        \notag &f(\bx^{\numIter+1}) - f(x^*) 
        \le 3 \delta_1 D \tfrac{\numIter+4}{\sqrt{(\numIter+1)(\numIter+2)}}  \\  
        &+ \tfrac{9(\numIter+4)}{(\numIter+1)(\numIter+2)} \delta_2 D^2 + \tfrac{18(L+\Lavg_2) D^3}{(\numIter+2)(\numIter+3)} \label{exact_upper_bound_after_gamma}
        \\
        \notag &+ \tfrac{\numIter+4}{4} \biggl\{ \tfrac{2(L+\Lavg_2)}{3}\hDe_{\hx}^3 
         + \left(\tfrac{\sqrt{(\numIter+1)(\numIter+2)}}{3D}\delta_1 + 2\delta_2\right) \hDe_{\hx}^2 \biggr\} \\
        \notag&= \cO(\delta_1 D) + \cO\left(\tfrac{\delta_2 D^2}{\numIter}\right) + \cO\left(\tfrac{(L+\Lavg_2) D^3}{\numIter^2}\right) + \cO\left((L + \Lavg_2) \numIter \hDe_{\hx}^3 \right) \\
        \notag&\quad + \cO\left(\tfrac{\numIter^2\delta_1\hDe_{\hx}^2}{D} \right) + \cO\left( \numIter\delta_2\hDe_{\hx}^2 \right).
    \end{align}
\end{proof}


\begin{customthm}{\ref{thm:basic_delta_choice}}
    Let Assumptions~\ref{as:cnvxty},~\ref{as:L1},~\ref{as:L2} hold. Let $\numIter$ be the total number of iterations, 
    \begin{equation*}
        \delta_{1, k} = \delta_1 \ge \max_{0 \leq j \leq \numIter} \Done{\hx}{j},~~\delta_{2, k} = \delta_2 \ge \max_{0 \leq j \leq \numIter} \Dtwo{\hx}{j},
    \end{equation*}
    and $L \ge \Lavg_2$ in Algorithm~\ref{alg:basic}.
    Let $\e>0$ be the desired accuracy and let the gradient and Hessian inexactness levels be chosen for every $k = 0, \ldots, N$ as
    \begin{equation}
        \label{eq:non_acc_cons_bound_g_H}
        \hDeg{\hx}{k} \eqdef \hDe_{g \vert \hx } \le \tfrac{\sqrt{2}}{144}\tfrac{\e}{D}, \quad
        \hDeH{\hx}{k} \eqdef \hDe_{H \vert \hx } \le \tfrac{\sqrt{3}}{72}\sqrt{\tfrac{\e(L+\Lavg_2)}{D}}.
    \end{equation}

    Depending on the target accuracy $\e$, we choose the number of iterations $\numIter$ and the state inexactness level $\hDe_{\hx \vert k} \eqdef \hDe_{\hx}$ as follows:

    \begin{enumerate}[noitemsep,topsep=0pt,leftmargin=12pt]
        \item If $\e$ is sufficiently small, namely $\e \le 12(L+\Lavg_2)D^3$, we set:
    \begin{equation}
    \label{eq:num_iter_small}
        \numIter = \left\lceil \sqrt{\tfrac{108(L+\Lavg_2) D^3}{\e}} \right\rceil - 2,
    \end{equation}
    \begin{equation}
        \label{eq:non_acc_cons_bound_x_small}
        \hDe_{\hx} \le \min \lb
            \tfrac{\sqrt{2}\e}{288 \Lavg_1 D}, 
            \tfrac{\sqrt{3\e(L+\Lavg_2)}}{144 \Lavg_2\sqrt{D}} , 
            \tfrac{\sqrt\e}{3\sqrt{(L+\Lavg_2)D}}
        \rb.
    \end{equation}
        \item Otherwise, if $\e > 12(L+\Lavg_2)D^3$, we set $N = 1$ and:
    \begin{equation}
        \label{eq:non_acc_cons_bound_x_large}
        \hDe_{\hx} \le \min \lb
            \tfrac{\sqrt{2}\e}{288 \Lavg_1 D}, 
            \tfrac{\sqrt{3\e(L+\Lavg_2)}}{144 \Lavg_2\sqrt{D}}, 
            \tfrac{\e^{1/3}}{(6(L+\Lavg_2))^{1/3}}, 
            D
        \rb.
    \end{equation}
    \end{enumerate}

    We also set $\gamma = \tfrac{\sqrt{(\numIter+1)(\numIter+2)}}{6D}$. Then after $\numIter+1$ iterations Algorithm~\ref{alg:basic} outputs an $\e$-solution, i.e.,
    \begin{equation*}
        f(\bx^{\numIter+1})-f(x^*) \le \e.
    \end{equation*}
\end{customthm}


\begin{proof}
    By the conditions of Theorem~\ref{thm:convex_case_theorem}, the regularization parameters must satisfy $\delta_{1, k} = \delta_1 \ge \max\limits_{j \in [\numIter]} \Done{\hx}{j}$ and $\delta_{2, k} = \delta_2 \ge \max\limits_{j \in [\numIter]} \Dtwo{\hx}{j}$.
    
    In both cases, $\hDe_{\hx}$ is bounded by the first two terms in the minimum. Therefore, by~\eqref{eq:agg_deltas_app}, for every $j = 0, \ldots, \numIter$ we have:
    \begin{equation}
    \begin{aligned}
    \label{eq:deltas_choice}
        \Done{\hx}{j} &\le \hDe_{g \vert \hx} + 2\Lavg_1\hDe_{\hx} \\
        &\le \tfrac{\sqrt{2}}{144}\tfrac{\e}{D} + 2\Lavg_1 \left( \tfrac{\sqrt{2}}{288 \Lavg_1} \tfrac{\e}{D} \right) = \tfrac{\sqrt{2}}{72}\tfrac{\e}{D} \eqdef \delta_1, \\[2ex]
        \Dtwo{\hx}{j} &\le \hDe_{ H \vert \hx} + 2\Lavg_2\hDe_{\hx} \\
        &\le \tfrac{\sqrt{3}}{72}\sqrt{\tfrac{\e(L+\Lavg_2)}{D}} + 2\Lavg_2 \left( \tfrac{\sqrt{3}}{144 \Lavg_2} \sqrt{\tfrac{\e(L+\Lavg_2)}{D}} \right) \\
        &= \tfrac{\sqrt{3}}{36}\sqrt{\tfrac{\e(L+\Lavg_2)}{D}} \eqdef \delta_2.
    \end{aligned}
    \end{equation}
    
    Using the general inequalities $\tfrac{\numIter+4}{4} \leq \tfrac{\numIter+2}{2}$ and $\sqrt{(\numIter+1)(\numIter+2)} \leq \numIter+2$, the consensus error term from~\eqref{exact_upper_bound_after_gamma} can be bounded as:
    \begin{align}
    \label{eq:consensus_term_bound}
        & \tfrac{\numIter+4}{4} \left\{ \tfrac{2(L+\Lavg_2)}{3}\hDe_{\hx}^3 + \left(\tfrac{\sqrt{(\numIter+1)(\numIter+2)}}{3D}\delta_1 + 2\delta_2\right) \hDe_{\hx}^2 \right\} \notag \\
        &\leq \tfrac{L+\Lavg_2}{3}(\numIter+2) \hDe_{\hx}^3 + \tfrac{(\numIter+2)^2}{6D}\delta_1\hDe_{\hx}^2 + (\numIter+2)\delta_2\hDe_{\hx}^2.
    \end{align}
    
    \begin{enumerate}[noitemsep,topsep=0pt,leftmargin=12pt]
        \item $\e \le 12(L+\Lavg_2)D^3$ regime:
        
        Let us bound every term in~\eqref{exact_upper_bound_after_gamma}. 
    
    Since $(\numIter+2)^2 \ge \tfrac{108(L+\Lavg_2) D^3}{\e}$, we have:
    \begin{align}
    \label{D^3_term}
        \tfrac{18(L+\Lavg_2) D^3}{(\numIter+2)(\numIter+3)} < \tfrac{18(L+\Lavg_2) D^3}{(\numIter+2)^2} \le \tfrac{18}{108} \e = \tfrac{\e}{6}.
    \end{align}

    Then, using the fact that $\tfrac{\numIter+4}{\sqrt{(\numIter+1)(\numIter+2)}} \le 2\sqrt{2}$ for $\numIter \ge 1$:
    \begin{align}
    \label{eq:gradient_term_bound}
        \notag &3 \delta_1 D \tfrac{\numIter+4}{\sqrt{(\numIter+1)(\numIter+2)}} \le 6\sqrt{2} \delta_1 D \\
        &= 6\sqrt{2} D \left( \tfrac{\sqrt{2}}{72}\tfrac{\e}{D} \right) = \tfrac{12}{72}\e = \tfrac{\e}{6}.
    \end{align}
    
    Then, using $\tfrac{\numIter+4}{(\numIter+1)(\numIter+2)} \le \tfrac{4}{\numIter+2} \le 4\sqrt{\tfrac{\e}{108(L+\Lavg_2)D^3}}$ and the definition of $\delta_2$ from ~\eqref{eq:deltas_choice}:
    \begin{align}
    \label{eq:hessian_term_bound}
        9 \delta_2 D^2 \tfrac{\numIter+4}{(\numIter+1)(\numIter+2)} &\le \tfrac{\e}{6}.
    \end{align}
    According to ~\eqref{eq:num_iter_small}:

    \begin{equation*}
    \label{eq:number_of_iterations}
        \numIter + 2 = \max\left\{ 3, \left\lceil \sqrt{\tfrac{108(L+\Lavg_2) D^3}{\e}}  \right\rceil\right\}.
    \end{equation*}

    As $\e \leq 12(L+\Lavg_2)D^3$, $\sqrt{\tfrac{108(L+\Lavg_2) D^3}{\e}} \geq 3$ and $N+2 = \left\lceil\sqrt{\tfrac{108(L+\Lavg_2) D^3}{\e}} \right\rceil \leq \sqrt{\tfrac{108(L+\Lavg_2) D^3}{\e}} + 1 \leq \tfrac{4}{3}\sqrt{\tfrac{108(L+\Lavg_2) D^3}{\e}}$
    \begin{align*}
        & \tfrac{L+\Lavg_2}{3} (\numIter+2) \hDe_{\hx}^3 \\
        &\le \tfrac{L+\Lavg_2}{3} \tfrac{4}{3} \sqrt{\tfrac{108(L+\Lavg_2)D^3}{\e}} \left[ \tfrac{1}{27} \tfrac{\e^{3/2}}{(L+\Lavg_2)^{3/2}D^{3/2}} \right] = \tfrac{8\sqrt{3}}{81} \e, \\[1.5ex]
        & \tfrac{(\numIter+2)^2}{6D} \delta_1 \hDe_{\hx}^2 \\
        &\le (\tfrac{4}{3})^2\tfrac{108(L+\Lavg_2)D^2}{6\e} \left( \tfrac{\sqrt{2}}{72} \tfrac{\e}{D} \right) \left[ \tfrac{1}{9} \tfrac{\e}{(L+\Lavg_2)D} \right] = \tfrac{4\sqrt{2}}{81} \e, \\[1.5ex]
        & (\numIter+2) \delta_2 \hDe_{\hx}^2 \\
        &\le \tfrac{4}{3}\sqrt{\tfrac{108(L+\Lavg_2)D^3}{\e}} \left( \tfrac{\sqrt{3}}{36} \sqrt{\tfrac{\e(L+\Lavg_2)}{D}} \right) \left[ \tfrac{1}{9} \tfrac{\e}{(L+\Lavg_2)D} \right] \\
        &= \tfrac{2}{27} \e.
    \end{align*}
        Summing up these three components and combining with ~\eqref{D^3_term}, ~\eqref{eq:gradient_term_bound} and ~\eqref{eq:hessian_term_bound}, we have:
    \begin{equation*}
    \begin{aligned}
        f(\bx^{\numIter+1}) - f(x^*) &\leq \tfrac{\e}{6} + \tfrac{\e}{6} + \tfrac{\e}{6} + \tfrac{\e}{2} = \e,
    \end{aligned}
    \end{equation*}
        \item  $\e > 12(L+\Lavg_2)D^3$: 
        
        $\numIter = 1$, thus $\numIter+2 = 3$. Then:
    \begin{align*}
        \tfrac{18(L+\Lavg_2) D^3}{3 \cdot 4} &= \tfrac{3}{2}(L+\Lavg_2)D^3 < \tfrac{3}{2} (\tfrac{\e}{12}) = \tfrac{\e}{8} \le \tfrac{\e}{6}, \\
        3 \delta_1 D \tfrac{5}{\sqrt{6}} &= \tfrac{15}{\sqrt{6}} D \left( \tfrac{\sqrt{2}}{72}\tfrac{\e}{D} \right) = \tfrac{5}{24\sqrt{3}}\e \le \tfrac{\e}{6}, \\
        9 \delta_2 D^2 \tfrac{5}{6} &= \tfrac{15}{2} D^2 \left( \tfrac{\sqrt{3}}{36}\sqrt{\tfrac{\e(L+\Lavg_2)}{D}} \right) \le \tfrac{15}{2 \cdot72}\e \le \tfrac{\e}{6}.
    \end{align*}
    Then, using the third and the fourth bounds from~\eqref{eq:non_acc_cons_bound_x_large}, for ~\eqref{eq:consensus_term_bound} we obtain:
    \begin{align*}
        & \tfrac{L+\Lavg_2}{3} (3) \hDe_{\hx}^3 \le (L+\Lavg_2) \left[ \tfrac{\e}{6(L+\Lavg_2)} \right] = \tfrac{\e}{6}, \\[1.5ex]
        & \tfrac{9}{6D} \delta_1 \hDe_{\hx}^2 \le \tfrac{3}{2D} \left( \tfrac{\sqrt{2}}{72} \tfrac{\e}{D} \right) D^2 = \tfrac{\sqrt{2}}{48} \e \le \tfrac{\e}{12}, \\[1.5ex]
        & 3 \delta_2 \hDe_{\hx}^2 \le 3 \left( \tfrac{\sqrt{3}}{36} \sqrt{\tfrac{\e(L+\Lavg_2)}{D}} \right) D^2 \leq \tfrac{\e}{24}.
    \end{align*}
    Summing up these bounds, we again obtain $f(\bx^{\numIter+1}) - f(x^*) \le \e$, which concludes the proof.
    \end{enumerate}
\end{proof}

\textbf{Proof of Remark~\ref{rem:lebesgue_diameter_convex}}
\begin{proof}
\label{prf:remark_convex_proof}
    We need to show that $\numIter \ls \tfrac{\delta_1}{\gamma} + \hDe_0 \rs \leq \e$.  Choosing regularization parameter $\delta_1$, number of iterations $N$ and consensus errors as in Theorem \ref{thm:basic_delta_choice}, we have:
    \begin{align*}
        & N \tfrac{\delta_1}{\gamma} = \tfrac{N}{\sqrt{(N+1)(N+2)}} \cdot 6\delta_1 D \le \delta_1 D  = \tfrac{6\sqrt{2}}{72} \tfrac{\varepsilon}{D} \cdot D = \tfrac{\sqrt{2}}{12} \varepsilon \\[10pt]
        & N \tfrac{2 (L + \Lavg_2)}{3} (\Delta_x^{\text{max}})^3 \le (N+2) \cdot \tfrac{2}{3} (L + \Lavg_2) (\Delta_x^{\text{max}})^3 \\
        & \quad \le \tfrac{4}{3}\sqrt{\tfrac{108 (L + \Lavg_2) D^3}{\varepsilon}} \cdot \tfrac{2}{3} (L + \Lavg_2) \cdot \tfrac{\varepsilon^{3/2}}{27 ((L + \Lavg_2) D)^{3/2}} \\
        & \quad = \tfrac{8\sqrt{108}}{243} \varepsilon = \tfrac{8\sqrt{3}}{81}  \varepsilon \\[10pt]
        & N \tfrac{\sqrt{(N+1)(N+2)}}{3} \delta_1 (\Delta_x^{\text{max}})^2 \le \tfrac{(N+2)^2}{3} \delta_1 (\Delta_x^{\text{max}})^2 \\
        & \quad \le \tfrac{16}{9}\tfrac{108 (L + \Lavg_2) D^3}{\varepsilon \cdot 3D} \cdot \tfrac{\sqrt{2}}{72} \tfrac{\varepsilon}{D} \cdot \tfrac{\varepsilon}{9 (L + \Lavg_2) D} = \tfrac{8\sqrt{2}}{81}  \varepsilon \\[10pt]
        & N \cdot 2 \delta_2 (\Delta_x^{\text{max}})^2 \le 2(N+2) \delta_2 (\Delta_x^{\text{max}})^2 \\
        & \quad \le \tfrac{2 \cdot 4}{3}\tfrac{\sqrt{108 (L + \Lavg_2) D^3}}{\sqrt{\varepsilon}} \cdot \tfrac{\sqrt{3}}{36} \sqrt{\tfrac{\varepsilon (L + \Lavg_2)}{D}} \cdot \tfrac{\varepsilon}{9 (L + \Lavg_2) D} = \tfrac{4}{27} \varepsilon 
    \end{align*}
    Summing these four terms up, we have:
    \begin{align*}
        & \left( \tfrac{\sqrt{2}}{12} + \tfrac{8\sqrt{3}}{81} + \tfrac{8\sqrt{2}}{81} + \tfrac{4}{27} \right) \varepsilon \leq \e
    \end{align*}
    Hence:
    \begin{align*}
        \numIter \ls \tfrac{\delta_1}{\gamma} + \econs \rs \leq \e
    \end{align*}
    Then, our method is monotone up to a small inaccuracy $\e$.
\end{proof}


\begin{customlemma}{\ref{lem:comm_complexity}}
    To satisfy the consensus bounds \eqref{eq:non_acc_cons_bound_x_small} and ~\eqref{eq:non_acc_cons_bound_g_H} at each iteration $k = 1, \ldots, \numIter$, it is sufficient to perform the following number of consensus steps:
    \begin{subequations}
    \label{eq:T_complexities}
    \begin{align}
        \Tx{k} 
        & \geq 
        \tfrac{\tau}{\lambda} \log 
            \left( 
                2D\sqrt{\numMach} \cdot \max 
                    \left\{ 
                        \tfrac{288 \Lavg_1 D}{\sqrt{2}\e}, 
                    \right. 
            \right. 
        \notag \\
        & \qquad 
            \left. 
                \left. 
                    \tfrac{144 \Lavg_2\sqrt{D}}{\sqrt{3\e(L+\Lavg_2)}}, 
                    \tfrac{3\sqrt{(L+\Lavg_2)D}}{\sqrt{\e}} 
                \right\} 
            \right), 
        \label{eq:Tx_final} 
        \\[2ex]
        \Tg{k} 
        & \geq 
        \tfrac{\tau}{\lambda} \log 
            \left( 
                \tfrac{144 D \sqrt{\numMach} \big( \zeta_g + 2\Lmax_1 D \big)}{\sqrt{2}\e} 
            \right), 
        \label{eq:Tg_final} 
        \\[2ex]
        \THh{k} 
        & \geq 
        \tfrac{\tau}{\lambda} \log 
            \left( 
                \tfrac{72\sqrt{\numMach D} \big( \zeta_H + 2\Lmax_2\sqrt{\dimd} D \big)}{\sqrt{3\e(L+\Lavg_2)}} 
            \right), 
        \label{eq:TH_final}
    \end{align}
    \end{subequations}
    where for brevity we define $\zeta_g \eqdef \sqrt{\tfrac{1}{\numMach} \sum_{i=1}^\numMach \|\nabla f_i(x^*)\|^2}$ and $\zeta_H \eqdef \sqrt{\tfrac{1}{\numMach} \sum_{i=1}^\numMach \|\nabla^2 f_i(x^*) - \nabla^2 f(x^*)\|_F^2}$.
\end{customlemma}

\begin{proof}
    The consensus bounds for points, gradients, and Hessians share a unified structure. Let $u_i^k$ be a generic local vector, and $\bar{u}^k$ be its exact network average. By Assumption~\ref{assum:mixing_matrix_sequence}, every $\tau$ consensus steps contract the distance to the average by a factor of $(1 - \lambda)$. Therefore, the consensus error after $T$ steps is bounded as:
    \begin{equation}
    \label{eq:general_contraction}
    \begin{split}
        \max_{i\in[\numMach]} \|\hat{u}_i^k - \bar{u}^k\| &\le \|\hat{\mU}^k - \bar{\mU}^k\|_F \\
        &\le (1 - \lambda)^{T / \tau} \|\mU^k - \bar{\mU}^k\|_F \\
        &\le e^{-\lambda T / \tau} \|\mU^k - \bar{\mU}^k\|_F.
    \end{split}
    \end{equation}

    To guarantee an accuracy $\Delta$ it is necessary to make $T \ge \tfrac{\tau}{\lambda} \log \big( \tfrac{\|\mU^k - \bar{\mU}^k\|_F}{\Delta} \big)$ consensus iterations. It remains to bound $\|\mU^k - \bar{\mU}^k\|_F$ for each sequence.
    
    Consider Lemma~\ref{lem:node_descent} and consider $x = \hx_i^{k-1}$. Then, from ~\eqref{eq:node_descent_final_app}, we have:
    \begin{equation*}
        f(x_i^k) \le f(\hx_i^{k-1}) + \tfrac{\delta_{1, k-1}}{\gamma}.
    \end{equation*}
    Let $w_{ij}^{k-1} = [\mW^{k-1}]_{ij}$. Then, $\hx_i^{k-1} = \tsum_{j=1}^\numMach w_{ij} x_j^{k-1}$ and, due to the convexity of $f$, we have:
    \begin{equation*}
        f(\hx_i^{k-1}) = f\Big(\tsum_{j=1}^\numMach w_{ij} x_j^{k-1}\Big) \le \max_{j\in[\numMach]} f(x_j^{k-1}).
    \end{equation*}
    Combining these two inequalities, we have:
    \begin{equation*}
        \max_{i\in[\numMach]} f(x_i^k) \le \max_{j\in[\numMach]} f(x_j^{k-1}) + \tfrac{\delta_{1, k-1}}{\gamma}.
    \end{equation*}
    Unrolling this recursion, we obtain:
    \begin{equation*}
        \max_{i\in[\numMach]} f(x_i^k) \le \max_{i\in[\numMach]} f(x_i^0) + \tsum_{t=0}^{k-1} \tfrac{\delta_{1, t}}{\gamma}.
    \end{equation*}
    Assuming all nodes start at the same point $x_i^0 = \bx^0$, the right-hand side is strictly bounded by $f(\bx^0) + \numIter \ls \tfrac{\delta_1}{\gamma} + \econs \rs$. Thus, by definition, $x_i^k \in \mathcal{L}'(\bx^0)$ for all $i$ and $k \le \numIter$, and then we have $\|x_i^k - x^*\| \le D$. 
    
    Furthermore, since $\bx^k$ is a convex combination of points within $\mathcal{L}'(\bx^0)$, it also belongs to this set, implying $\|\bx^k - x^*\| \le D$. 
    Now we have the bound:
    \begin{equation*}
        \|x_i^k - \bx^k\| \le \|x_i^k - x^*\| + \|x^* - \bx^k\| \le 2D.
    \end{equation*}
    Summing over all nodes, we obtain the bound in the Frobenius norm:
    \begin{equation*}
        \|\mX^k - \bar{\mX}^k\|_F = \sqrt{ \tsum_{i=1}^\numMach \|x_i^k - \bx^k\|^2 } \le 2D\sqrt{\numMach}.
    \end{equation*}
    Using~\eqref{eq:general_contraction} with $\Delta = \hDe_{\hx}$ from~\eqref{eq:non_acc_cons_bound_x_small}, we obtain~\eqref{eq:Tx_final}.
    Thus:
    \begin{align*}
        \|\mX^k - \bar{\mX}^k\|_F \le 2D\sqrt{\numMach}
    \end{align*}

    Next, we know that $\tfrac{1}{\numMach} \tsum_{j=1}^\numMach \nabla f_j(x^*) = 0$. As $\bgg^k = \tfrac{1}{\numMach} \tsum_{j=1}^\numMach \nabla f_j(\hx_j^k)$, we obtain:
    \begin{equation*}
        \begin{split}
        \|\nabla f_i(\hat{x}_i^k) - \bar{g}^k\| 
        &= \bigg\| \nabla f_i(\hat{x}_i^k) - \nabla f_i(x^*) + \nabla f_i(x^*) \\
        &\quad - \frac{1}{m} \sum_{j=1}^m \big( \nabla f_j(\hat{x}_j^k) - \nabla f_j(x^*) \big) \bigg\| \\
        &\le \|\nabla f_i(\hat{x}_i^k) - \nabla f_i(x^*)\| + \|\nabla f_i(x^*)\| \\
        &\quad + \frac{1}{m} \sum_{j=1}^m \|\nabla f_j(\hat{x}_j^k) - \nabla f_j(x^*)\|.
    \end{split}
    \end{equation*}
    Using the bound $\|\hx_i^k - x^*\| \le D$ and~\eqref{eq:L1} for the first and third terms, and the bound $\tfrac{1}{\numMach} \tsum_{j=1}^\numMach L_{1, j} D \le \Lmax_1 D$, we have:
    \begin{equation*}
    \begin{split}
        \big\| \nabla f_i(\hx_i^k) - \bgg^k \big\| &\le \Lmax_1 D + \|\nabla f_i(x^*)\| + \Lmax_1 D \\
        &= \|\nabla f_i(x^*)\| + 2\Lmax_1 D.
    \end{split}
    \end{equation*}

    Now, let $\mG^k \in \R^{\numMach \times \dimd}$ be the matrix whose $i$-th row is the local gradient $\nabla f_i(\hx_i^k)^\top$, and $\bar{\mG}^k$ be the matrix where each row is the average gradient $(\bgg^k)^\top$. We can write:
    \begin{equation*}
    \begin{split}
        \|\mG^k - \bar{\mG}^k\|_F &= \sqrt{ \tsum_{i=1}^\numMach \|\nabla f_i(\hx_i^k) - \bgg^k\|^2 } \\
        &\le \sqrt{ \tsum_{i=1}^\numMach \big( \|\nabla f_i(x^*)\| + 2\Lmax_1 D \big)^2 }.
    \end{split}
    \end{equation*}
    Applying Minkowski's inequality, we obtain:
    \begin{equation*}
    \begin{split}
        \|\mG^k - \bar{\mG}^k\|_F &\le \sqrt{ \tsum_{i=1}^\numMach \|\nabla f_i(x^*)\|^2 } + \sqrt{ \tsum_{i=1}^\numMach (2\Lmax_1 D)^2 } \\
        &= \sqrt{\numMach} \zeta_g + 2\sqrt{\numMach} \Lmax_1 D. \label{eq:Rg_bound}
    \end{split}
    \end{equation*}

    By the contraction property in Assumption~\ref{assum:mixing_matrix_sequence}, using~\eqref{eq:general_contraction}, the distance after $\Tg{k}$ steps of consensus to the exact average $\bgg^k$ is:
    \begin{equation*}
        \max_{i\in[\numMach]} \left\| \hg_i(\hx_i^k) - \bgg^k \right\| \le e^{-\lambda \Tg{k} / \tau} \sqrt{\numMach} \big( \zeta_g + 2\Lmax_1 D \big).
    \end{equation*}
    To ensure this condition, we set:

    \begin{equation}
        \Tg{k} \ge \tfrac{\tau}{\lambda} \log \left( \tfrac{\sqrt{\numMach} \big( \zeta_g + 2\Lmax_1 D \big)}{\hDe_{g \vert \hx}} \right).
    \end{equation}
    Using $\hDe_{g \vert \hx}$ from~\eqref{eq:non_acc_cons_bound_g_H}, we obtain \eqref{eq:Tg_final}.

    As $\bH^k \eqdef \tfrac{1}{\numMach} \tsum_{j=1}^\numMach \nabla^2 f_j(\hx_j^k)$, we can write:
    \begin{equation*}
    \begin{split}
        &\big\| \nabla^2 f_i(\hx_i^k) - \bH^k \big\|_F = \bigg\| \nabla^2 f_i(\hx_i^k) - \nabla^2 f_i(x^*) + \nabla^2 f_i(x^*) \\
        &\quad - \nabla^2 f(x^*) + \nabla^2 f(x^*) - \bH^k \bigg\|_F \\
        &\le \left\| \nabla^2 f_i(\hx_i^k) - \nabla^2 f_i(x^*) \right\|_F + \left\| \nabla^2 f_i(x^*) - \nabla^2 f(x^*) \right\|_F  \\
        &+ \tfrac{1}{\numMach} \tsum_{j=1}^\numMach \left\| \nabla^2 f_j(x^*) - \nabla^2 f_j(\hx_j^k) \right\|_F.
    \end{split}
    \end{equation*}
    Using ~\eqref{eq:L2}, the standard norm inequality $\|A\|_F \le \sqrt{\dimd} \opnorm{A}$ (hold for any symmetric matrix A), and the bound $\|\hx_i^k - x^*\| \le D$, we have  $\|\nabla^2 f_i(\hx_i^k) - \nabla^2 f_i(x^*)\|_F \le \sqrt{\dimd} L_{2, i} D \le \sqrt{\dimd} \Lmax_2 D$. The third term is bounded by $\tfrac{1}{\numMach} \tsum_{j=1}^\numMach \sqrt{\dimd} L_{2, j} D \leq \sqrt{\dimd} \Lmax_2 D$. Thus:
    \begin{equation*}
    \begin{split}
        \big\| \nabla^2 &f_i(\hx_i^k) - \bH^k \big\|_F \\
        &\le \sqrt{\dimd}\Lmax_2 D + \left\| \nabla^2 f_i(x^*) - \nabla^2 f(x^*) \right\|_F + \sqrt{\dimd}\Lmax_2 D \\
        &= \left\| \nabla^2 f_i(x^*) - \nabla^2 f(x^*) \right\|_F + 2\Lmax_2\sqrt{\dimd} D.
    \end{split}
    \end{equation*}

    Let $\mathcal{H}^k \in \R^{\numMach \times \dimd^2}$ be the matrix where each row is the flattened local Hessian $\nabla^2 f_i(\hx_i^k)$, and $\bar{\mathcal{H}}^k$ be the matrix of the flattened average Hessians $\bH^k$. The initial consensus error over the network in the Frobenius norm is:
    \begin{equation*}
    \begin{split}
        \|\mathcal{H}^k &- \bar{\mathcal{H}}^k\|_F = \sqrt{ \tsum_{i=1}^\numMach \|\nabla^2 f_i(\hx_i^k) - \bH^k\|_F^2 } \\
        &\le \sqrt{ \tsum_{i=1}^\numMach \big( \|\nabla^2 f_i(x^*) - \nabla^2 f(x^*)\|_F + 2\Lmax_2\sqrt{\dimd} D \big)^2 }.
    \end{split}
    \end{equation*}
    Applying Minkowski's inequality, we obtain:
    \begin{equation*}
    \begin{split}
        \|\mathcal{H}^k - \bar{\mathcal{H}}^k\|_F &\le \sqrt{ \tsum_{i=1}^\numMach \|\nabla^2 f_i(x^*) - \nabla^2 f(x^*)\|_F^2 } \\
        &\quad + \sqrt{ \tsum_{i=1}^\numMach \big(2\Lmax_2\sqrt{\dimd} D\big)^2 } \\
        &= \sqrt{\numMach} \zeta_H + \sqrt{\numMach} 2\Lmax_2\sqrt{\dimd} D.
    \end{split}
    \end{equation*}

    By the contraction property in Assumption~\ref{assum:mixing_matrix_sequence}, and since the operator norm of any matrix is upper-bounded by its Frobenius norm ($\opnorm{\mA} \le \|\mA\|_F$ for any symmetric matrix $\mA$), using~\eqref{eq:general_contraction}, the distance after $\THh{k}$ steps of consensus is:
    \begin{equation*}
    \begin{split}
        \max_{i\in[\numMach]} \opnorm{\hH_i^k - \bH^k} &\le \max_{i\in[\numMach]} \big\| \hH_i^k - \bH^k \big\|_F \\
        &\le \big\| \hat{\mathcal{H}}^k - \bar{\mathcal{H}}^k \big\|_F \\
        &\le e^{-\lambda \THh{k} / \tau} \|\mathcal{H}^k - \bar{\mathcal{H}}^k\|_F \\
        &\le e^{-\lambda \THh{k} / \tau} \sqrt{\numMach} \big( \zeta_H + 2\Lmax_2\sqrt{\dimd} D \big).
    \end{split}
    \end{equation*}
    To ensure this error is bounded by the target accuracy $\hDe_{H \vert \hx}$, we set:
    \begin{equation}
        \THh{k} \ge \tfrac{\tau}{\lambda} \log \left( \tfrac{\sqrt{\numMach} \big( \zeta_H + 2\Lmax_2\sqrt{\dimd} D \big)}{\hDe_{H \vert \hx}} \right).
    \end{equation}
    Using $\hDe_{H \vert \hx}$ from~\eqref{eq:non_acc_cons_bound_g_H}, this directly yields~\eqref{eq:TH_final}, which concludes the proof.
\end{proof}

\begin{theorem}
\label{thm:strongly_convex_theorem}
    Let Assumptions~\ref{as:L1}~\ref{as:L2} and Assumption~\ref{as:strcnvxty} with $\mu_i >0$ hold. For a given number of iterations $\numIter \ge 1$, let us define for all $0 \le j \le \numIter$ as:
    \begin{align*}
        \hDe_{g\vert \hx} &\eqdef \max_{0 \le j \le \numIter} \hDeg{\hx}{j}, \\
        \hDe_{H\vert \hx} &\eqdef \max_{0 \le j \le \numIter} \hDeH{\hx}{j}, \\
        \hDe_{\hx} &\eqdef \max_{0 \le j \le \numIter} \hDex{\hx}{j}.
    \end{align*}
    Assume that the parameters of Algorithm~\ref{alg:basic} satisfy $\delta_{1, k} = \delta_1 \geq \hDe_{g\vert \hx} + 2\Lavg_1\hDe_{\hx}$, $\delta_{2, k} = \delta_2 \geq \hDe_{H\vert \hx} + 2\Lavg_2\hDe_{\hx}$, and $L \ge \Lavg_2$. 
    
    Then, choosing $\alpha_k = \alpha$ as:
    \begin{equation}
        \label{eq:str_convex_alpha_thm}
        \alpha = \min \left\{ \tfrac{1}{2}, \tfrac{\muavg}{16\bigl(\gamma\hDe_{g \vert \hx} + \hDe_{H\vert \hx} + 2(\gamma\Lavg_1 + \Lavg_2)\hDe_{\hx}\bigr)}, \sqrt{\tfrac{3\muavg}{16(L+\Lavg_2)D}} \right\},
    \end{equation}
    after $N+1$ iterations of Algorithm~\ref{alg:basic}, we have:
    \begin{equation*}
    \begin{aligned}
        f(\bx^{N+1}) &- f(x^{*}) \leq (1 - \alpha)^{N+1} \bigl(f(\bx^{0}) - f(x^{*})\bigr) \\
        &+ \tfrac{1}{\alpha} \Biggl\{\tfrac{\delta_1}{\gamma} +\tfrac{2(L+\Lavg_2)}{3}\hDe_{\hx}^3 + 2(\gamma\delta_1 + \delta_2)\hDe_{\hx}^2 \Biggr\}.
    \end{aligned}
    \end{equation*}
\end{theorem}

\begin{proof}
    From ~\eqref{eq:avg_descent_form}, since $f$ is convex and $x^*$ is a minimizer, it is enough to consider points on the segment between $x^k$ and $x^*$, using the definitions of $D$ and $\hDe_{\hx}$, we obtain for any $k\geq0$ we obtain:
    \begin{align*}
        f(\bx^{k+1}) &\le f\bigl(\bx^k + \alpha_k(x^* - \bx^k)\bigr) \\
        &\quad + 2\alpha_k^2(\gamma\delta_1 + \delta_2)\|x^* - \bx^k\|^2 \\
        &\quad + \tfrac{2(L+\Lavg_2)}{3}\alpha_k^3\|x^* - \bx^k\|^3 \\
        &\quad + \biggl\{\tfrac{\delta_1}{\gamma} +  \tfrac{2(L+\Lavg_2)}{3}\hDe_{\hx}^3 + 2(\gamma\delta_1 + \delta_2)\hDe_{\hx}^2 \biggr\}.
    \end{align*}
    By the strong convexity of $f$, we have:
    \begin{align*}
        f(\bx^{k+1}) &\le (1 - \alpha_k)f(\bx^k) + \alpha_k f(x^*) \\
        &\quad - \tfrac{\muavg\alpha_k (1-\alpha_k)}{2}\|\bx^k-x^{*}\|^2 \\
        &\quad + 2\alpha_k^2(\gamma\delta_1 + \delta_2)\|\bx^k-x^{*}\|^2 \\
        &\quad + \alpha_k^3\tfrac{2(L+\Lavg_2)}{3}\|\bx^k-x^{*}\|^3 \\
        &\quad + \biggl\{\tfrac{\delta_1}{\gamma} +  \tfrac{2(L+\Lavg_2)}{3}\hDe_{\hx}^3 + 2(\gamma\delta_1 + \delta_2)\hDe_{\hx}^2 \biggr\}.
    \end{align*}
    
    Subtracting $f(x^{*})$ from both sides, we get:
    \begin{align}
        \label{eq:strong_convex_bracket}
        f(\bx^{k+1}) &- f(x^{*}) \le (1-\alpha_{k})\bigl( f(\bx^k) - f(x^{*}) \bigr) \notag \\
        &- \tfrac{\alpha_{k}}{2}\|\bx^k-x^{*}\|^2 \biggl[ \muavg(1-\alpha_{k}) \notag \\
        &- 4\alpha_k(\gamma\delta_1 + \delta_2) - \alpha_{k}^2 \tfrac{4(L+\Lavg_2)}{3}\|\bx^k-x^{*}\| \biggr] \notag \\
        &+ \tfrac{\delta_1}{\gamma} + \tfrac{2(L+\Lavg_2)}{3}\hDe_{\hx}^3 + 2(\gamma\delta_1 + \delta_2)\hDe_{\hx}^2.
    \end{align}
    
    We want to choose a constant $\alpha \in (0, 1/2]$ such that the bracketed term is non-negative. Since $\alpha \le 1/2$, we have $\muavg(1-\alpha) \ge \muavg/2$. Using the bound $\|\bx^k - x^*\| \le D$, a sufficient condition for the bracket to be non-negative is:
    \begin{equation*}
        \tfrac{\muavg}{2} - 4\alpha(\gamma\delta_1 + \delta_2) - \alpha^2\tfrac{4(L+\Lavg_2)}{3}D \ge 0.
    \end{equation*}
    This inequality holds if we ensure that $4\alpha(\gamma\delta_1 + \delta_2) \le \muavg/4$ and $\alpha^2\tfrac{4(L+\Lavg_2)}{3}D \le \muavg/4$. Substituting the definitions of $\delta_1$ and $\delta_2$, the first condition becomes $4\alpha\bigl(\gamma\hDe_{g\vert \hx} + \hDe_{H\vert \hx} + 2(\gamma\Lavg_1 + \Lavg_2)\hDe_{\hx}\bigr) \le \muavg/4$. These conditions are satisfied by our choice of $\alpha$ in~\eqref{eq:str_convex_alpha_thm}.
    
    Therefore, the bracket is non-negative, and we have:
    \begin{align*}
        f(\bx^{k+1}) - f(x^{*}) &\leq (1 - \alpha) \bigl(f(\bx^k) - f(x^{*})\bigr) \\
        &\quad+ \tfrac{\delta_1}{\gamma} +\tfrac{2(L+\Lavg_2)}{3}\hDe_{\hx}^3 + 2(\gamma\delta_1 + \delta_2)\hDe_{\hx}^2.
    \end{align*}
    Unrolling this recursion over $k = 0, \dots, N$, we obtain:
    \begin{align*}
        &f(\bx^{N+1}) - f(x^{*}) \leq (1 - \alpha)^{N+1} \bigl(f(\bx^{0}) - f(x^{*})\bigr) \\
        & + \biggl[ \tfrac{\delta_1}{\gamma} + \tfrac{2(L+\Lavg_2)}{3}\hDe_{\hx}^3 + 2(\gamma\delta_1 + \delta_2)\hDe_{\hx}^2 \biggr] \tsum_{i=0}^N (1 - \alpha)^i \\
        &\leq (1 - \alpha)^{N+1} \bigl(f(\bx^{0}) - f(x^{*})\bigr) \\
        &\quad + \tfrac{1}{\alpha} \biggl\{ \tfrac{\delta_1}{\gamma} +  \tfrac{2(L+\Lavg_2)}{3}\hDe_{\hx}^3 + 2(\gamma\delta_1 + \delta_2)\hDe_{\hx}^2 \biggr\},
    \end{align*}
    where in the last step we used the geometric series bound $\sum_{i=0}^{\infty} (1-\alpha)^i \leq \tfrac{1}{\alpha}$.
\end{proof}


\begin{customthm}{\ref{thm:strongly_convex_complexity}}
    Let the regularization parameter be chosen as $\gamma = 1/D$. Let $\e > 0$ be the desired accuracy. Also let Assumptions~\ref{as:L1},~\ref{as:L2} and Assumption~\ref{as:strcnvxty} with $\mu_i >0$ hold.

    Let inexactness levels be chosen for every $k=0, \ldots, \numIter$ as:
    \begin{equation}
    \label{eq:sc_cons_bound}
    \begin{aligned}
        \hDex{\hx}{k} &\eqdef \hDe_{\hx} 
         \le 
        \min 
            \left\{ 
                \tfrac{\alpha \e}{24 \Lavg_1 D}, 
                \sqrt[3]{\tfrac{\alpha\e}{4(L+\Lavg_2)}}, 
            \right. 
        \\
        & \qquad
            \left. 
                2D \sqrt{ \tfrac{\alpha\e \Lavg_1}{3\muavg D^2\Lavg_1 + 4\alpha\e(2\Lavg_1 + D\Lavg_2)} }, 
                \tfrac{\muavg}{64(\Lavg_1/D + \Lavg_2)} 
            \right\}, 
        \\[3ex]
        \hDeg{\hx}{k} &\eqdef \hDe_{g \vert \hx} 
         \le 
        \min 
            \left\{ 
                \tfrac{\alpha \e}{12 D}, \; 
                \tfrac{\muavg D}{32} 
            \right\}, 
        \\[1.5ex]
        \hDeH{\hx}{k} &\eqdef \hDe_{H \vert \hx} 
         \le 
        \tfrac{\muavg}{16},
    \end{aligned}
    \end{equation}
    where we define $\alpha$ as:
    \begin{equation}
        \label{eq:sc_alpha_const}
        \alpha = \min \left\{ \tfrac{1}{2}, \sqrt{\tfrac{3\muavg}{16(L+\Lavg_2)D}} \right\}.
    \end{equation}
    Assume that the parameters of Algorithm~\ref{alg:basic} satisfy $\delta_{1, k} = \delta_1 \geq \hDe_{g\vert \hx} + 2\Lavg_1\hDe_{\hx}$, $\delta_{2, k} = \delta_2 \geq \hDe_{H\vert \hx} + 2\Lavg_2\hDe_{\hx}$, and $L \ge \Lavg_2$.
    
    Then after $\numIter + 1$ iterations with 
    \begin{equation}
        \label{eq:sc_iterations_bound}
        \numIter = \left\lceil \tfrac{1}{\alpha} \log \left( \tfrac{2\bigl(f(\bx^0) - f(x^*)\bigr)}{\e} \right) \right\rceil - 1,
    \end{equation}
    the method outputs an $\e$-solution, i.e.:
    \begin{equation*}
        f(\bx^{N+1}) - f(x^*) \le \e.
    \end{equation*}
\end{customthm}

\begin{proof}
    Firstly:
    \begin{align*}
        \gamma\delta_1 + \delta_2 
        &= \gamma \bigl(\hDe_{g\vert \hx} + 2\Lavg_1\hDe_{\hx}\bigr) + \bigl(\hDe_{H\vert \hx} + 2\Lavg_2\hDe_{\hx}\bigr) \\
        &= \gamma\hDe_{g\vert \hx} + \hDe_{H\vert \hx} + 2(\gamma\Lavg_1 + \Lavg_2)\hDe_{\hx} \\
        &= \tfrac{\hDe_{g\vert \hx}}{D} + \hDe_{H\vert \hx} + 2\hDe_{\hx} \left( \tfrac{\Lavg_1}{D} + \Lavg_2 \right) \\
        &\le \tfrac{\muavg}{32} + \tfrac{\muavg}{16} + 2 \left( \tfrac{\muavg}{64(\Lavg_1/D + \Lavg_2)} \right) \left( \tfrac{\Lavg_1}{D} + \Lavg_2 \right) \\
        &= \tfrac{\muavg}{32} + \tfrac{\muavg}{16} + \tfrac{\muavg}{32} = \tfrac{\muavg}{8} \le \tfrac{\muavg}{16\alpha},
    \end{align*}
    where the last inequality holds as $\alpha \le 1/2$. Hence, we obtain:
    \begin{equation*}
        \gamma\hDe_{g\vert \hx} + \hDe_{H\vert \hx} + 2(\gamma\Lavg_1 + \Lavg_2)\hDe_{\hx} \le \tfrac{\muavg}{16\alpha},
    \end{equation*}
    and then:
    \begin{equation*}
        \alpha \le \tfrac{\muavg}{16\bigl(\gamma\hDe_{g\vert \hx} + \hDe_{H\vert \hx} + 2(\gamma\Lavg_1 + \Lavg_2)\hDe_{\hx}\bigr)}.
    \end{equation*}
    Therefore, Theorem~\ref{thm:strongly_convex_theorem} applies with $\alpha_k \equiv \alpha$.

    Next, using $\hDe_{\hx} \le \tfrac{\alpha \e}{24 \Lavg_1 D}$ and $\hDe_{g\vert \hx} \le \tfrac{\alpha \e}{12 D}$, we have:
    \begin{equation*}
        \delta_1 = \hDe_{g\vert \hx} + 2\Lavg_1\hDe_{\hx} \le \tfrac{\alpha \e}{12 D} + 2\Lavg_1 \left( \tfrac{\alpha \e}{24 \Lavg_1 D} \right) = \tfrac{\alpha \e}{6 D},
    \end{equation*}
    and using $\hDe_{H\vert \hx} \le \tfrac{\muavg}{16}$, we get:
    \begin{align*}
        \delta_2 = \hDe_{H\vert \hx} + 2\Lavg_2\hDe_{\hx} \le \tfrac{\muavg}{16} + 2\Lavg_2 \left( \tfrac{\alpha \e}{24 \Lavg_1 D} \right) = \tfrac{\muavg}{16} + \tfrac{\alpha \e \Lavg_2}{12 \Lavg_1 D}.
    \end{align*}

    Now, by the result of Theorem~\ref{thm:strongly_convex_theorem}:
    \begin{equation*}
        f(\bx^{N+1}) - f(x^*) \le (1 - \alpha)^{N+1} \bigl(f(\bx^0) - f(x^*)\bigr) + \tfrac{1}{\alpha} \left( \delta_1 D + \econs \right),
    \end{equation*}
    where $\econs = \tfrac{2(L+\Lavg_2)}{3}\hDe_{\hx}^3 + 2(\tfrac{\delta_1}{D} + \delta_2)\hDe_{\hx}^2$.

    By or choice of $N$ in~\eqref{eq:sc_iterations_bound} and using the inequality $(1-\alpha)^{N+1} \le e^{-\alpha(N+1)}$, the geometric term is bounded by
    \begin{align}
    \label{eq:iter_str_cvx}
        \notag&(1 - \alpha)^{N+1} \bigl(f(\bx^0) - f(x^*)\bigr) \\
        \notag&\quad \le \exp\biggl(-\alpha \cdot \tfrac{1}{\alpha} \log \Bigl( \tfrac{2\bigl(f(\bx^0) - f(x^*)\bigr)}{\e} \Bigr)\biggr) \\
        &\quad \qquad \times \bigl(f(\bx^0) - f(x^*)\bigr)  = \tfrac{\e}{2}.
    \end{align}

    Let us bound the consensus error term $\econs$. According to \eqref{eq:sc_cons_bound}, for the first part we have:
    \begin{equation*}
       \tfrac{2(L+\Lavg_2)}{3} \left( \sqrt[3]{\tfrac{\alpha\e}{4(L+\Lavg_2)}} \right)^3 = \tfrac{\alpha\e}{6}.
    \end{equation*}
    
    For the second part, we obtain:
    \begin{align*}
        \tfrac{\delta_1}{D} + \delta_2 &\le \tfrac{1}{D} \left( \tfrac{\alpha \e}{6 D} \right) + \left( \tfrac{\muavg}{16} + \tfrac{\alpha \e \Lavg_2}{12 \Lavg_1 D} \right) \\
        &= \tfrac{\muavg}{16} + \tfrac{\alpha \e}{12 D} \left( \tfrac{2}{D} + \tfrac{\Lavg_2}{\Lavg_1} \right).
    \end{align*}
    Using this and the third bound for $\hDe_{\hx}$ from \eqref{eq:sc_cons_bound}, we have:
    \begin{align*}
        2 \left( \tfrac{\delta_1}{D} + \delta_2 \right) \hDe_{\hx}^2 
        & \le 
        2 \left( \tfrac{\muavg}{16} + \tfrac{\alpha \e}{12 D} \left( \tfrac{2}{D} + \tfrac{\Lavg_2}{\Lavg_1} \right) \right) \hDe_{\hx}^2 
        \\[1.5ex]
        & = 
        \tfrac{3\muavg D^2\Lavg_1 + 4\alpha\e(2\Lavg_1 + D\Lavg_2)}{24 D^2 \Lavg_1} 
        \\
        & \qquad \times 
        \left( 2D \sqrt{ \tfrac{\alpha\e \Lavg_1}{3\muavg D^2\Lavg_1 + 4\alpha\e(2\Lavg_1 + D\Lavg_2)} } \right)^2 
        \\[1.5ex]
        & = 
        \tfrac{3\muavg D^2\Lavg_1 + 4\alpha\e(2\Lavg_1 + D\Lavg_2)}{24 D^2 \Lavg_1} 
        \\
        & \qquad \times 
        \tfrac{4 D^2 \alpha\e \Lavg_1}{3\muavg D^2\Lavg_1 + 4\alpha\e(2\Lavg_1 + D\Lavg_2)} 
        \\[1.5ex]
        & = 
        \tfrac{4 D^2 \alpha\e \Lavg_1}{24 D^2 \Lavg_1} 
        = 
        \tfrac{\alpha\e}{6}.
    \end{align*}
    Summing these components, the total consensus error is bounded by $\econs \le \tfrac{\alpha\e}{6} + \tfrac{\alpha\e}{6} = \tfrac{\alpha\e}{3}$.

    Therefore, the total error term satisfies:
    \begin{equation}
    \label{eq:cons_err_str_cvx}
        \tfrac{1}{\alpha} \left( \delta_1 D + \econs \right) \le \tfrac{1}{\alpha} \left( \tfrac{\alpha\e}{6} + \tfrac{\alpha\e}{3} \right) = \tfrac{\e}{2}.
    \end{equation}

    Summing ~\eqref{eq:iter_str_cvx} and ~\eqref{eq:cons_err_str_cvx}, we have $f(\bx^{N+1}) - f(x^*) \le \tfrac{\e}{2} + \tfrac{\e}{2} = \e$, which concludes the proof.
\end{proof}

\textbf{Proof of Remark~\ref{rem:sc_lebesgue_set_bound}.}
\begin{proof}
\label{prf:sc_remark_proof}
    In the general analysis~\eqref{eq:error_accumulation}, the additive error accumulates linearly, requiring $\numIter \cdot \text{error} \le \e$ to keep iterates bounded. However, in the strongly convex regime, the error accumulation is bounded uniformly: according to Theorem~\ref{thm:strongly_convex_theorem}, unrolling the recurrence for any $k \ge 0$ (in particular, for $k = N$), we obtain:
    \begin{equation*}
        f(\bx^N) - f(x^*) \le (1 - \alpha)^N \bigl(f(\bx^0) - f(x^*)\bigr) + \tfrac{1}{\alpha} \left( \delta_1 D + \econs \right).
    \end{equation*}
    Since $(1-\alpha)^N \le 1$ and Theorem~\ref{thm:strongly_convex_complexity} ensures $\tfrac{1}{\alpha} \left( \delta_1 D + \econs \right) \le \tfrac{\e}{2}$, we obtain:
    \begin{equation*}
        f(\bx^N) - f(x^*) \le f(\bx^0) - f(x^*) + \tfrac{\e}{2} \implies f(\bx^N) \le f(\bx^0) + \e.
    \end{equation*}
    Hence, our method is monotone up to a small inaccuracy $\e$.
\end{proof}

\textbf{Proof of Lemma~\ref{lem:strcvx_comm_complexity}.}
\begin{proof}
    The proof is similar to the proof of Lemma~\ref{lem:comm_complexity}. 
    
    Plugging consensus inaccuracies from ~\eqref{eq:sc_cons_bound} into the universal contraction bound~\eqref{eq:general_contraction} guarantees the communication complexities ~\eqref{eq:sc_T_complexities}.
\end{proof}

\subsection{Proofs for the Accelerated Decentralized Cubic Newton}

\label{app:accelerated_proofs}

 The proof proceeds as follows:
\begin{itemize}
    \item Lemma~\ref{lem:acc_psi_upper} establishes an upper bound on the local estimating sequence $\psi_i^k(x)$.
    \item Lemma~\ref{lem:acc_step_progress} quantifies the progress of one local cubic step~\eqref{eq:acc_step}.
    \item Lemma~\ref{lem:acc_psi_lower} proves a lower bound on $\psi_i^{k+1}(y_i^{k+1})$ with explicit error terms.
    \item Lemma~\ref{lem:estimating_sequence_final} combines these results into a two-sided estimating-sequence bound and specifies admissible choices of the parameters $\alpha_k$ and $A_k$.
    \item Lemmas~\ref{lem:err_v},~\ref{lem:err_g}, and~\ref{lem:raw_dissimilarity} control the remaining consensus and disagreement terms that appear in the final estimate.
    \item Theorem~\ref{thm:acc_convergence} combines all previous results and yields the final convergence guarantee for the decentralized accelerated method.
    \item Theorem~\ref{thm:acc_delta_choice} translates this result into explicit per-iteration requirements on the consensus errors needed to obtain an $\e$-solution.
    \item Finally, these bounds are combined with the contraction properties of the consensus routine, established in Lemma~\ref{lem:acc_comm_complexity}, to get total communication complexity.
\end{itemize}

\begin{lemma}[Upper bound on the estimating sequence]
\label{lem:acc_psi_upper}
Let Assumptions~\ref{as:L1},~\ref{as:L2} and Assumption~\ref{as:strcnvxty} hold with $\mu_i \ge 0,~\forall i \in [\numMach]$ hold.
Following~\eqref{eq:agg_deltas}, define
\begin{equation}
    \label{eq:acc_Delta1_Delta2}
    \Done{\hv}{k} \eqdef \hDeg{\hv}{k} + 2\Lavg_1\hDex{\hv}{k},
    ~
    \Dtwo{\hv}{k} \eqdef \hDeH{\hv}{k} + 2\Lavg_2\hDex{\hv}{k}.
\end{equation}
Assume that
\begin{equation}
    \label{eq:acc_psi_upper_params}
    L \ge \Lavg_2,
    \qquad
    \delta_{2,0} \ge \Dtwo{\hv}{0}.
\end{equation}
Then, for any node $i\in[\numMach]$, any $k\ge 1$, and any $x\in\R^\dimd$, the estimating sequence defined by~\eqref{eq:acc_psi1} and~\eqref{eq:acc_psi_kp} satisfies
\begin{equation}
    \label{eq:acc_psi_upper_bound}
    \begin{aligned}
        \psi_i^k(x)
        \le~
        & \tfrac{f(x)}{A_{k-1}}
        + \Done{\hv}{0}\|x_i^1-\hv_i^0\|
        + \Done{\hv}{0}\|x-\hv_i^0\| \\
        & + \tfrac{\Dtwo{\hv}{0}+\delta_{2,0}+\bk_{2,k-1}}{2}\|x-\hv_i^0\|^2 \\
        & + \tfrac{\Lavg_2+L+\bk_{3,k-1}}{6}\|x-\hv_i^0\|^3 \\
        & + \tsum_{j=1}^{k-1}\tfrac{\alpha_j}{A_j}
        \la \hg_i(x_i^{j+1})-\nabla f(x_i^{j+1}),\, x-x_i^{j+1}\ra .
    \end{aligned}
\end{equation}
\end{lemma}

\begin{proof}
By the definition of the accelerated cubic model~\eqref{eq:model_2ord_acc},
\begin{equation*}
    \hnu_i^0(x;\hv_i^0)
    =
    \hphi_i^0(x;\hv_i^0) - f(\hv_i^0)
    + \tfrac{\delta_{2,0}}{2}\|x-\hv_i^0\|^2
    + \tfrac{L}{6}\|x-\hv_i^0\|^3.
\end{equation*}
Since $x_i^1$ minimizes $\hnu_i^0(\cdot;\hv_i^0)$, for any $x\in\R^\dimd$ we have
\begin{equation}
    \label{eq:acc_x1_opt}
    \begin{aligned}
        \hphi_i^0(x_i^1;\hv_i^0) 
        & +  \tfrac{\delta_{2,0}}{2}\|x_i^1-\hv_i^0\|^2
        + \tfrac{L}{6}\|x_i^1-\hv_i^0\|^3
        \\
        \le~ & \hphi_i^0(x;\hv_i^0) 
        + \tfrac{\delta_{2,0}}{2}\|x-\hv_i^0\|^2
        + \tfrac{L}{6}\|x-\hv_i^0\|^3.
    \end{aligned}
\end{equation}

By Lemma~\ref{lem:inexact_taylor} applied at the initial consensus point $\hv_i^0$, with $x=x_i^1$, we have
\begin{equation}
    \label{eq:acc_fx1_bound_1}
    \begin{aligned}
        f(x_i^1)
        \le~
        & \hphi_i^0(x_i^1;\hv_i^0)
        + (\hDeg{\hv}{0} + 2\Lavg_1\hDex{\hv}{0})\|x_i^1-\hv_i^0\| \\
        & + \tfrac{\hDeH{\hv}{0} + 2\Lavg_2\hDex{\hv}{0}}{2}\|x_i^1-\hv_i^0\|^2
        + \tfrac{\Lavg_2}{6}\|x_i^1-\hv_i^0\|^3\\
        \le~
        & \hphi_i^0(x_i^1;\hv_i^0)
        + \Done{\hv}{0}\|x_i^1-\hv_i^0\| \\
        & + \tfrac{\Dtwo{\hv}{0}}{2}\|x_i^1-\hv_i^0\|^2
        + \tfrac{\Lavg_2}{6}\|x_i^1-\hv_i^0\|^3,
    \end{aligned}
\end{equation}
where the second inequality follows from the definition~\eqref{eq:acc_Delta1_Delta2}.

Combining~\eqref{eq:acc_fx1_bound_1} with~\eqref{eq:acc_x1_opt} and using~\eqref{eq:acc_psi_upper_params}, we get
\begin{equation}
    \label{eq:acc_fx1_bound_2}
    \begin{aligned}
        f(x_i^1)
        \le~
        & \hphi_i^0(x;\hv_i^0)
        + \Done{\hv}{0}\|x_i^1-\hv_i^0\| \\
        & + \tfrac{\delta_{2,0}}{2}\|x-\hv_i^0\|^2
        + \tfrac{L}{6}\|x-\hv_i^0\|^3.
    \end{aligned}
\end{equation}

Applying Lemma~\ref{lem:inexact_taylor} once again, now to $\hphi_i^0(x;\hv_i^0)$:
\begin{equation*}
    \begin{aligned}
        \hphi_i^0(x;\hv_i^0)
        \le
        f(x)
        + \Done{\hv}{0}\|x-\hv_i^0\|
        +& \tfrac{\Dtwo{\hv}{0}}{2}\|x-\hv_i^0\|^2 \\
        &+ \tfrac{\Lavg_2}{6}\|x-\hv_i^0\|^3,
    \end{aligned}
\end{equation*}
and putting it into~\eqref{eq:acc_fx1_bound_2} yields
\begin{equation*}
    \begin{aligned}
        f(x_i^1)
        \le~
        & f(x)
        + \Done{\hv}{0}\|x_i^1-\hv_i^0\|
        + \Done{\hv}{0}\|x-\hv_i^0\| \\
        & + \tfrac{\Dtwo{\hv}{0}+\delta_{2,0}}{2}\|x-\hv_i^0\|^2
        + \tfrac{\Lavg_2+L}{6}\|x-\hv_i^0\|^3.
    \end{aligned}
\end{equation*}

Therefore, by the initialization~\eqref{eq:acc_psi1},
\begin{equation}
    \label{eq:acc_psi1_upper}
    \begin{aligned}
        \psi_i^1(x)
        \le~
        & f(x)
        + \Done{\hv}{0}\|x_i^1-\hv_i^0\|
        + \Done{\hv}{0}\|x-\hv_i^0\| \\
        & + \tfrac{\Dtwo{\hv}{0}+\delta_{2,0}+\bk_{2,0}}{2}\|x-\hv_i^0\|^2 \\
        & + \tfrac{\Lavg_2+L+\bk_{3,0}}{6}\|x-\hv_i^0\|^3.
    \end{aligned}
\end{equation}

Now let $k\ge 1$. Expanding the recursion~\eqref{eq:acc_psi_kp}, we obtain
\begin{equation}
    \label{eq:acc_psi_expand}
    \begin{aligned}
        \psi_i^k(x)
        =~
        & \psi_i^1(x)
        + \tfrac{\bk_{2,k-1}-\bk_{2,0}}{2}\|x-\hv_i^0\|^2 \\
        & + \tfrac{\bk_{3,k-1}-\bk_{3,0}}{6}\|x-\hv_i^0\|^3 \\
        & + \tsum_{j=1}^{k-1}\tfrac{\alpha_j}{A_j}
        \left(
            f(x_i^{j+1})
            + \la \hg_i(x_i^{j+1}), x-x_i^{j+1}\ra 
        \right.\\
        & \quad \quad \quad  \quad \quad + \left.
            \tfrac{\muavg}{2}\|x-x_i^{j+1}\|^2
        \right)\\
        =~
        & \psi_i^1(x)
        + \tfrac{\bk_{2,k-1}-\bk_{2,0}}{2}\|x-\hv_i^0\|^2 \\
        & + \tfrac{\bk_{3,k-1}-\bk_{3,0}}{6}\|x-\hv_i^0\|^3 \\
        & + \tsum_{j=1}^{k-1}\frac{\alpha_j}{A_j}
       \left(
            f(x_i^{j+1})
            + \la \nabla f(x_i^{j+1}), x-x_i^{j+1}\ra
       \right. \\
        & \quad \quad \quad  \quad \quad + \left.
            \tfrac{\muavg}{2}\|x-x_i^{j+1}\|^2
        \right)\\
        & + \tsum_{j=1}^{k-1}\frac{\alpha_j}{A_j}
        \la \hg_i(x_i^{j+1})-\nabla f(x_i^{j+1}), x-x_i^{j+1}\ra .
    \end{aligned}
\end{equation}

Since $f$ is $\muavg$-strongly convex, for all $j\ge 1$,
\begin{equation*}
    f(x_i^{j+1}) + \la \nabla f(x_i^{j+1}), x-x_i^{j+1} \ra + \tfrac{\muavg}{2}\|x-x_i^{j+1}\|^2 \le f(x).
\end{equation*}
Hence,
\begin{equation}
    \label{eq:acc_sum_convex}
    \begin{aligned}
        \tsum_{j=1}^{k-1}\tfrac{\alpha_j}{A_j}
        \left(
            f(x_i^{j+1})
        \right .
        +
        &
        \left.
            \la \nabla f(x_i^{j+1}), x-x_i^{j+1}\ra
            + \tfrac{\muavg}{2}\|x-x_i^{j+1}\|^2 
        \right)\\
        &\le
        f(x)\tsum_{j=1}^{k-1}\tfrac{\alpha_j}{A_j}.
    \end{aligned}
\end{equation}
Using the relation $A_j=(1-\alpha_j)A_{j-1}$, we have
\begin{equation*}
    \tfrac{\alpha_j}{A_j}
    =
    \tfrac{1}{A_j} - \tfrac{1}{A_{j-1}},
\end{equation*}
and therefore
\begin{equation}
    \label{eq:acc_alpha_A_sum}
    \tsum_{j=1}^{k-1}\tfrac{\alpha_j}{A_j}
    =
    \tsum_{j=1}^{k-1}\left(\tfrac{1}{A_j}-\tfrac{1}{A_{j-1}}\right)
    =
    \frac{1}{A_{k-1}}-1.
\end{equation}

Substituting~\eqref{eq:acc_psi1_upper},~\eqref{eq:acc_sum_convex}, and~\eqref{eq:acc_alpha_A_sum} into~\eqref{eq:acc_psi_expand}, we get
\begin{equation*}
    \begin{aligned}
        \psi_i^k(x)
        \le~
        & \tfrac{f(x)}{A_{k-1}}
        + \Done{\hv}{0}\|x_i^1-\hv_i^0\|
        + \Done{\hv}{0}\|x-\hv_i^0\| \\
        & + \tfrac{\Dtwo{\hv}{0}+\delta_{2,0}+\bk_{2,k-1}}{2}\|x-\hv_i^0\|^2 \\
        & + \tfrac{\Lavg_2+L+\bk_{3,k-1}}{6}\|x-\hv_i^0\|^3 \\
        & + \tsum_{j=1}^{k-1}\tfrac{\alpha_j}{A_j}
        \la \hg_i(x_i^{j+1})-\nabla f(x_i^{j+1}), x-x_i^{j+1}\ra .
    \end{aligned}
\end{equation*}
\end{proof}

\begin{lemma}[Progress of the accelerated cubic step]
\label{lem:acc_step_progress}
Fix an iteration $k\ge 1$ and a node $i\in[\numMach]$. Let Assumptions~\ref{as:L1},~\ref{as:L2}, and Assumption~\ref{as:strcnvxty} hold with $\mu_i \ge 0$.  
Then the following holds.
\begin{itemize}
    \item If 
        \begin{equation}
            \label{eq:acc_case1_cond}
            \Done{\hv}{k}
            \ge
            \Dtwo{\hv}{k}\|\hv_i^k-x_i^{k+1}\|
            + \tfrac{\Lavg_2}{2}\|\hv_i^k-x_i^{k+1}\|^2,
        \end{equation}
        then
        \begin{equation}
            \label{eq:acc_case1_bound}
            \begin{aligned}
                \|\nabla f(x_i^{k+1})\|
                \le
                \left(
                    \max\left\{
                        \tfrac{\delta_{2,k}}{\Dtwo{\hv}{k}},
                        \tfrac{L}{\Lavg_2}
                    \right\}
                    + 2
                \right) \Done{\hv}{k}
            \end{aligned}
        \end{equation}
    \item otherwise, if 
        \begin{equation}
            \label{eq:acc_case2_params}
            \delta_{2,k}\ge 3\Dtwo{\hv}{k},
            \qquad
            L\ge 3\Lavg_2,
        \end{equation}
        then
        \begin{equation}
        \label{eq:acc_case2_bound}
            \begin{aligned}
                \left \langle \nabla \right . & \left. f(x_i^{k+1}),  \hv_i^k-x_i^{k+1}\right \rangle
                \ge
                 \tfrac{\sqrt{5}}{3} \\
                &~\times \min
                    \left\{\tfrac{\|\nabla f(x_i^{k+1})\|^2}{2(2\Dtwo{\hv}{k}+\delta_{2,k})},
                    \tfrac{\|\nabla f(x_i^{k+1})\|^{3/2}}{(2\Lavg_2+L)^{1/2}}
                \right\}.
            \end{aligned}
        \end{equation}
\end{itemize}
\end{lemma}

\begin{proof}
By the optimality condition for the step~\eqref{eq:acc_step},
\begin{equation}
    \label{eq:acc_opt_cond}
    0
    =
    \nabla \hphi_i^k(x_i^{k+1};\hv_i^k)
    + \delta_{2,k}s_i^k
    + \tfrac{L}{2}\|s_i^k\|s_i^k,
\end{equation}
where $s_i^k \eqdef x_i^{k+1}-\hv_i^k$.

Hence,
\begin{equation*}
    \begin{aligned}
        \|\nabla f&(x_i^{k+1})\|
        \\
        =~&
        \left\|
            \nabla f(x_i^{k+1})
            - \nabla \hphi_i^k(x_i^{k+1};\hv_i^k)
            + \delta_{2,k}s_i^k
            + \tfrac{L}{2}\|s_i^k\|s_i^k
        \right\| \\
        \le~
        &\|\nabla f(x_i^{k+1})-\nabla \hphi_i^k(x_i^{k+1};\hv_i^k)\|
        + \delta_{2,k}\|s_i^k\|
        + \tfrac{L}{2}\|s_i^k\|^2.
    \end{aligned}
\end{equation*}
Applying~\eqref{eq:inexact_taylor_grad} at $x=x_i^{k+1}$, we obtain
\begin{equation}
    \label{eq:acc_grad_bound_main}
    \|\nabla f(x_i^{k+1})\|
    \le
    \Done{\hv}{k}
    + (\Dtwo{\hv}{k}+\delta_{2,k})\|s_i^k\|
    + \tfrac{\Lavg_2+L}{2}\|s_i^k\|^2.
\end{equation}

Assume first that~\eqref{eq:acc_case1_cond} holds. Then
\begin{equation}
    \label{eq:acc_case1_grad}
    \begin{aligned}
        \|\nabla f(x_i^{k+1})\|
        \le~
        & 2\Done{\hv}{k}
        + \delta_{2,k}\|s_i^k\|
        + \tfrac{L}{2}\|s_i^k\|^2 \\
        \le~
        & 2\Done{\hv}{k}
        + \max\left\{
            \tfrac{\delta_{2,k}}{\Dtwo{\hv}{k}},
            \tfrac{L}{\Lavg_2}
        \right\} \\
        & \times 
        \left(
            \Dtwo{\hv}{k}\|s_i^k\|
            + \tfrac{\Lavg_2}{2}\|s_i^k\|^2
        \right) \\
        \le~
        & \left(
            \max\left\{
                \tfrac{\delta_{2,k}}{\Dtwo{\hv}{k}},
                \tfrac{L}{\Lavg_2}
            \right\}
            + 2
        \right)\Done{\hv}{k}.
    \end{aligned}
\end{equation}
This proves~\eqref{eq:acc_case1_bound}.

Now assume that
\begin{equation}
    \label{eq:acc_case2_cond}
    \Done{\hv}{k}
    \le
    \Dtwo{\hv}{k}\|\hv_i^k-x_i^{k+1}\|
    + \tfrac{\Lavg_2}{2}\|\hv_i^k-x_i^{k+1}\|^2,
\end{equation}
holds. Then, by~\eqref{eq:inexact_taylor_grad},
\begin{equation}
    \label{eq:acc_model_grad_case2}
    \begin{aligned}
        \|\nabla \hphi_i^k(x_i^{k+1};\hv_i^k)-&\nabla f(x_i^{k+1})\|
        \\
        \le~ & \Done{\hv}{k}
        + \Dtwo{\hv}{k}\|s_i^k\|
        + \tfrac{\Lavg_2}{2}\|s_i^k\|^2 \\
        \le~
        & 2\Dtwo{\hv}{k}\|s_i^k\|
        + \Lavg_2\|s_i^k\|^2\\
        \stackrel{\eqref{eq:acc_case2_params}}{\le} 
        & \tfrac{2}{3}\delta_{2,k}\|s_i^k\|
        + \tfrac{L}{3}\|s_i^k\|^2 \\
        =~
        & \tfrac{2}{3}
        \left(
            \delta_{2,k}
            + \tfrac{L}{2}\|s_i^k\|
        \right)\|s_i^k\|.
    \end{aligned}
\end{equation}
Denote $\zeta_{i, k} \eqdef \delta_{2,k} + \tfrac {L}{2}\|s_i^k\|$. Then~\eqref{eq:acc_opt_cond} can be rewritten as
\begin{equation}
    \label{eq:acc_opt_cond_zeta}
    \nabla \hphi_i^k(x_i^{k+1};\hv_i^k) = -\zeta_{i, k} s_i^k.
\end{equation}
Therefore, by~\eqref{eq:acc_model_grad_case2}, 
\begin{equation*}
    \begin{aligned}
    \tfrac{4}{9}(\zeta_{i, k})^2\|s_i^k\|^2
    \ge&~
    \|\nabla \hphi_i^k(x_i^{k+1};\hv_i^k)-\nabla f(x_i^{k+1})\|^2 \\
    \stackrel{\eqref{eq:acc_opt_cond_zeta}}{=}&~
    \|-\zeta_{i, k} s_i^k-\nabla f(x_i^{k+1})\|^2 \\
    =&~
    \|\nabla f(x_i^{k+1})\|^2
    - 2\zeta_{i, k} \la \nabla f(x_i^{k+1}), \hv_i^k-x_i^{k+1}\ra \\
    &+ (\zeta_{i, k})^2\|s_i^k\|^2.
    \end{aligned}
\end{equation*}
Rearranging terms, we obtain
\begin{equation*}
    2\zeta_{i, k} \la \nabla f(x_i^{k+1}), \hv_i^k-x_i^{k+1}\ra
    \ge
    \|\nabla f(x_i^{k+1})\|^2
    + \tfrac{5}{9}(\zeta_{i, k})^2\|s_i^k\|^2.
\end{equation*}
Using $a^2+b^2\ge 2ab$ with $a = \|\nabla f(x_i^{k+1})\|$, $b = \tfrac{\sqrt{5}}{3}\zeta_{i, k}\|s_i^k\|$, we get
\begin{equation}
    \label{eq:acc_inner_basic}
    \la \nabla f(x_i^{k+1}), \hv_i^k-x_i^{k+1}\ra
    \ge
    \tfrac{\sqrt{5}}{3}\|s_i^k\|\,\|\nabla f(x_i^{k+1})\|.
\end{equation}

It remains to lower bound $\|s_i^k\|$.
Combining~\eqref{eq:acc_grad_bound_main} with~\eqref{eq:acc_case2_cond}, we obtain
\begin{equation}
    \label{eq:acc_grad_bound_case2}
    \|\nabla f(x_i^{k+1})\|
    \le
    (2\Dtwo{\hv}{k}+\delta_{2,k})\|s_i^k\|
    + \tfrac{2\Lavg_2+L}{2}\|s_i^k\|^2.
\end{equation}
Now there are two possibilities.

If
\begin{equation*}
    (2\Dtwo{\hv}{k}+\delta_{2,k})\|s_i^k\|
    \ge
    \tfrac{2\Lavg_2+L}{2}\|s_i^k\|^2,
\end{equation*}
then from~\eqref{eq:acc_grad_bound_case2} we have
\begin{equation}
    \label{eq:acc_s_lower_1}
    \|s_i^k\|
    \ge
    \tfrac{\|\nabla f(x_i^{k+1})\|}{2(2\Dtwo{\hv}{k}+\delta_{2,k})}.
\end{equation}
Otherwise,
\begin{equation*}
    \tfrac{2\Lavg_2+L}{2}\|s_i^k\|^2
    \ge
    (2\Dtwo{\hv}{k}+\delta_{2,k})\|s_i^k\|,
\end{equation*}
and therefore
\begin{equation}
    \label{eq:acc_s_lower_2}
    \|s_i^k\|
    \ge
    \left(
        \tfrac{\|\nabla f(x_i^{k+1})\|}{2\Lavg_2+L}
    \right)^{1/2}.
\end{equation}
Combining~\eqref{eq:acc_inner_basic}, \eqref{eq:acc_s_lower_1}, and~\eqref{eq:acc_s_lower_2}, we obtain~\eqref{eq:acc_case2_bound}.
\end{proof}

\begin{lemma}[(\cite{ghadimi2017second},~Lemma~7),~(\cite{agafonov2023inexact},~Lemma~7)]
    \label{lm:argmin}
    Let $h(x)$ be a $\lambda$-strongly convex function ($\lambda \ge 0$), $x^0 \in \mathbb{R}^\dimd$, $p \ge 2$, $\theta_i \ge 0$~~for $i = 2,\ldots, p + 1$ and 
    $$y = \arg \min \limits_{x \in \mathbb{R}^\dimd} \{\widetilde{h}(x) = h(x) + \tsum_{i = 2}^{p + 1}  \tfrac{\theta_i}{i} \|x -x^0\|^i\}.$$
    Then, for all $x\in \mathbb{R}^\dimd$,
    $$\widetilde{h}(x) \ge \widetilde{h}(y) + \tfrac{\lambda}{2}\|x - y\|^2 + \tsum_{i = 2}^{p + 1} \left(\frac{1}{2}\right)^{i - 2} \tfrac{\theta_i}{i} \|x - y\|^i .$$
\end{lemma}

\begin{lemma}[Lower bound of the estimating sequence]
\label{lem:acc_psi_lower}
Let $f$ be $\muavg$-strongly ($\muavg \geq 0$) convex, let Assumptions~\ref{as:L1},~\ref{as:L2} hold, let $\{\bk_{2,k}\}_{k \geq 0}$,~$\{\bk_{3,k}\}_{k \geq 0}$ be non-negative non-decreasing sequences, and define
\begin{equation}
    \label{eq:err}
    \begin{aligned}
        \err_i^k
        =~ & \tsum_{j=1}^{k-1} \tfrac{\alpha_j}{A_j} \la \hg_i(x_i^{j+1}) - \nabla f(x_i^{j+1}), y_i^{j+1} - x_{i}^{j+1} \ra \\
        & + \tsum_{j=1}^{k-1} \tfrac{1}{A_j} \la \nabla f(x_i^{j+1}), v_i^{j} - \hv_{i}^{j} \ra.
    \end{aligned}
\end{equation}
Assume also that for $k \geq 1, ~i \in [\numMach]$, $\|y_i^{k} - x^*\| \leq \bR$ and 
\begin{equation}
    \label{eq:psi_assum}
    \psi_i^k(y_i^k) \ge \tfrac{f(x_i^k)}{A_{k-1}} + \err_i^k - r_i^k,
\end{equation}
for some $r_i^k \geq 0$.
Then, for any node $i\in[\numMach]$ and any $k\ge 1$, we have
\begin{equation*}
    \begin{aligned}
        \psi_i^{k+1}(y_i^{k+1})
        \ge~
        & \tfrac{f(x_i^{k+1})}{A_{k}} + \tfrac{1}{A_{k}} \la \nabla f(x_i^{k+1}), \hv_i^k - x_i^{k+1} \ra \\
        & + \tfrac{2\bk_{2,k-1} + \lambda_k}{4}\|y_i^{k+1} - y_i^k\|^2 \\
        & + \tfrac{2\bk_{3,k-1} + 3\lambda_k \bR^{-1}}{24}\|y_i^{k+1} - y_i^k\|^3 \\
        & + \tfrac{\alpha_k}{A_k} \la \nabla f(x_i^{k+1}), y_i^{k+1} - y_i^k \ra \\
        & + \err_i^{k+1}  - r_i^k,
    \end{aligned}
\end{equation*}
where $\lambda_k = \ls \tfrac{1}{A_{k-1}} - 1 \rs \muavg$.
\end{lemma}

\begin{proof}
    By~\eqref{eq:acc_psi1},~\eqref{eq:acc_psi_kp}
    \begin{equation*}
        \begin{aligned}
            \psi_{i}^{k}(x)
            =&~ f(x_i^1)
            + \tfrac{\bk_{2,{k-1}}}{2}\|x - \hv_i^0\|^2 + \tfrac{\bk_{3,{k-1}}}{6}\|x - \hv_i^0\|^3 \\
            &+ \tsum_{j=1}^{k-1}\tfrac{\alpha_j}{A_j}
            \left(
                f(x_i^{j+1}) + \la \hg_i(x_i^{j+1}), x - x_i^{j+1} \ra
            \right. \\
            & + \left. \tfrac{\muavg}{2}\|x - x_i^{j+1}\|^2 \right).
        \end{aligned}
    \end{equation*}
    Next, we apply Lemma~\ref{lm:argmin} to the function
    \begin{equation*}
        \begin{aligned}
            h(x) 
            =~ f(x_i^1) 
            + \tsum_{j=1}^{k-1}\tfrac{\alpha_j}{A_j} &
            \left(
                f(x_i^{j+1}) + \la \hg_i(x_i^{j+1}), x - x_i^{j+1} \ra
            \right. \\
            &~+ \left. \tfrac{\muavg}{2}\|x - x_i^{j+1}\|^2 \right).
        \end{aligned}
    \end{equation*}
    with $x^0 = \hv_i^0$, $p=2$, $\theta_2 = \bk_{2,k-1}$, $\theta_3 = \tfrac{\bk_{3,k-1}}{2}$, and $y = y_i^k$. Note, that $h(x)$ is strongly convex with parameter 
    \begin{equation*}
        \lambda_k = \tsum_{j=1}^{k-1} \tfrac{\alpha_j}{A_j} \muavg = \ls\tfrac{1}{A_{k-1}} - \tfrac{1}{A_0}\rs \muavg = \ls\tfrac{1}{A_{k-1}} -1\rs \muavg.
    \end{equation*}
    This yields
    \begin{equation*}
        \begin{aligned}
            \psi_i^k(x) 
            \ge &~ \psi_i^k(y_i^k) 
            + \tfrac{\bk_{2,k-1} + \lambda_k}{2}\|x - y_i^k\|^2 
            + \tfrac{\bk_{3,k-1}}{12}\|x - y_i^k\|^3 \\
            \stackrel{\eqref{eq:psi_assum}}{\ge} &~\tfrac{f(x_i^k)}{A_{k-1}} 
            + \tfrac{\bk_{2,k-1} + \lambda_k}{2}\|x - y_i^k\|^2 
            + \tfrac{\bk_{3,k-1}}{12}\|x - y_i^k\|^3 \\
            & + \err_i^k - r_i^k.
        \end{aligned}
    \end{equation*}
    Next, by definition of $\psi_i^{k+1}(x)$, the above inequality, and $\muavg$-strong convexity of $f$, we obtain 
    \begin{equation*}
        \begin{aligned}
            \psi_{i}^{k+1}(x)
            = &~ \psi_{i}^{k}(x)
            + \tfrac{\bk_{2,k} - \bk_{2,k-1}}{2}\|x - \hv_i^0\|^2 \\
            &+ \tfrac{\bk_{3,k} - \bk_{3, k-1}}{6}\|x - \hv_i^0\|^3 \\
            &+ \tfrac{\alpha_k}{A_k}
           \left(
                f(x_i^{k+1}) + \la \hg_i(x_i^{k+1}), x - x_i^{k+1} \ra
            \right. \\
            & \quad \quad \quad + \left. \tfrac{\muavg}{2}\|x - x_i^{k+1}\|^2 \right) \\
            \ge &~ \tfrac{f(x_i^k)}{A_{k-1}}  
            + \tfrac{\bk_{2,k-1} + \lambda_k}{2}\|x - y_i^k\|^2
            + \tfrac{\bk_{3,k-1}}{12}\|x - y_i^k\|^3 \\
            &  + \tfrac{\alpha_k}{A_k} \left(
                f(x_i^{k+1}) + \la \hg_i(x_i^{k+1}), x - x_i^{k+1} \ra
            \right. \\
            & \quad \quad \quad + \left. \tfrac{\muavg}{2}\|x - x_i^{k+1}\|^2 \right) + \err_i^k -  r_i^k\\
            \ge &~ \tfrac{1}{A_{k-1}} 
             \left( 
                f(x_i^{k+1}) 
                + 
                \la \nabla f(x_i^{k+1}), x_i^k - x_i^{k+1}\ra 
            \right. \\
            & \left. \quad \quad \quad + \tfrac{\muavg}{2}\|x_i^k - x_i^{k+1}\|^2 \right)\\
            & + \tfrac{\bk_{2,k-1} + \lambda_k}{2}\|x - y_i^k\|^2
            + \tfrac{\bk_{3,k-1}}{12}\|x - y_i^k\|^3 \\ 
            &+ \tfrac{\alpha_k}{A_k}
            \ls
                f(x_i^{k+1}) + \la \nabla f(x_i^{k+1}), x - x_i^{k+1} \ra
            \rs \\ 
            &+ \tfrac{\alpha_k}{A_k}
            \la 
                \hg_i(x_i^{k+1}) - \nabla f(x_i^{k+1}), x - x_i^{k+1} 
            \ra\\
            & + \tfrac{\alpha_k}{A_k} \tfrac{\muavg}{2}\|x - x_i^{k+1}\|^2 + \err_i^k  - r_i^k.
        \end{aligned}
    \end{equation*}  
    Next, we consider the sum of two linear models from the last inequality:
    \begin{equation*}
        \begin{aligned}
        \tfrac{1}{A_{k-1}}
        \left( 
            f(x_i^{k+1}) 
        \right. &   
        +\left.
            \la 
                \nabla f(x_i^{k+1}), x_i^k - x_i^{k+1}
            \ra 
        \rs \\
        &+ \tfrac{\alpha_k}{A_k}
        \ls
            f(x_i^{k+1}) + \la \nabla f(x_i^{k+1}), x - x_i^{k+1} \ra
        \rs \\
        \stackrel{\eqref{eq:As}}{=}~ & \tfrac{1 - \alpha_k}{A_{k}} f(x_i^{k+1}) 
        + \tfrac{1 - \alpha_k}{A_{k}}
        \la 
            \nabla f(x_i^{k+1}), x_i^k - x_i^{k+1}
        \ra \\
        & + \tfrac{\alpha_k}{A_{k}} f(x_i^{k+1}) 
        + \tfrac{\alpha_k}{A_{k}}
        \la 
            \nabla f(x_i^{k+1}), x - x_i^{k+1}
        \ra \\
        \stackrel{\eqref{eq:acc_vk}}{=}~ & \tfrac{f(x_i^{k+1})}{A_{k}} 
        + \tfrac{1 - \alpha_k}{A_{k}} 
        \la 
            \nabla f(x_i^{k+1}), 
            \tfrac{v_i^k - \alpha_k y_i^k}{1 - \alpha_k} - x_i^{k+1} 
        \ra \\
        & + \tfrac{\alpha_k}{A_{k}} 
        \la 
            \nabla f(x_i^{k +1}), x - x_i^{k+1} 
        \ra \\
        = & \tfrac{f(x_i^{k+1})}{A_{k}} 
        + \tfrac{1}{A_{k}}
        \la 
            \nabla f(x_i^{k+1}), v_i^k - x_i^{k+1} 
        \ra \\
        & + \tfrac{\alpha_k}{A_{k}}
        \la 
            \nabla f(x_i^{k+1}), x - y_i^k
        \ra  \\
        = & \tfrac{f(x_i^{k+1})}{A_{k}} 
        + \tfrac{1}{A_{k}}
        \la 
            \nabla f(x_i^{k+1}), \hv_i^k - x_i^{k+1} 
        \ra \\
        & + \tfrac{1}{A_{k}}
        \la 
            \nabla f(x_i^{k+1}), v_i^k - \hv_i^{k} 
        \ra \\
        & + \tfrac{\alpha_k}{A_{k}}
        \la 
            \nabla f(x_i^{k+1}), x - y_i^k
        \ra.
        \end{aligned}
    \end{equation*}
    Therefore, 
    \begin{equation*}
        \begin{aligned}
            \psi_{i}^{k+1}(x)
            \ge &~ \tfrac{f(x_i^{k+1})}{A_{k}} 
            + \tfrac{1}{A_{k}}
            \la 
                \nabla f(x_i^{k+1}), \hv_i^k - x_i^{k+1} 
            \ra \\
            &+ \tfrac{1}{A_{k}}
            \la 
                \nabla f(x_i^{k+1}), v_i^k - \hv_i^{k} 
            \ra \\
            &+ \tfrac{\alpha_k}{A_{k}}
            \la 
                \nabla f(x_i^{k+1}), x - y_i^k
            \ra  \\
            &+  \tfrac{\bk_{2,k-1} + \lambda_k}{2}\|x - y_i^k\|^2
            + \tfrac{\bk_{3,k-1}}{12}\|x - y_i^k\|^3  \\
            &+ \tfrac{\alpha_k}{A_k}
            \la 
                \hg_i(x_i^{k+1}) - \nabla f(x_i^{k+1}), x - x_i^{k+1} 
            \ra \\
            &+ \err_i^k  - r_i^k.
        \end{aligned}
    \end{equation*}
    Finally, by definition of $\err_i^{k+1}$~\eqref{eq:err}, we have
    \begin{equation*}
        \begin{aligned}
            \psi_i^{k+1}(y_i^{k+1}) \ge &~ \tfrac{f(x_i^{k+1})}{A_{k}}
            + \tfrac{1}{A_{k}}
            \la 
                \nabla f(x_i^{k+1}), \hv_i^k - x_i^{k+1} 
            \ra \\
            &+ \tfrac{\alpha_k}{A_{k}}
            \la 
                \nabla f(x_i^{k+1}), y_i^{k+1} - y_i^k
            \ra  \\
            &+  \tfrac{\bk_{2,k-1} + \lambda_k}{2}\|y_i^{k+1} - y_i^k\|^2 \\
            & + \tfrac{\bk_{3,k-1}}{12}\|y_i^{k+1} - y_i^k\|^3 \\
            & + \err_i^{k+1} - r_i^k\\
            \ge &~ \tfrac{f(x_i^{k+1})}{A_{k}}
            + \tfrac{1}{A_{k}}
            \la 
                \nabla f(x_i^{k+1}), \hv_i^k - x_i^{k+1} 
            \ra \\
            &+ \tfrac{\alpha_k}{A_{k}}
            \la 
                \nabla f(x_i^{k+1}), y_i^{k+1} - y_i^k
            \ra  \\
            &+  \tfrac{2\bk_{2,k-1} + \lambda_k}{4}\|y_i^{k+1} - y_i^k\|^2 \\
            &+  \tfrac{\lambda_k}{4}\tfrac{\|y_i^{k+1} - y_i^k\|}{2\bR}\|y_i^{k+1} - y_i^k\|^2  \\
            & + \tfrac{\bk_{3,k-1}}{12}\|y_i^{k+1} - y_i^k\|^3  + \err_i^{k+1} - r_i^k,
        \end{aligned}
    \end{equation*}
    where the last inequality holds due to triangle inequality $\|y_i^{k+1} - y_i^k\| \leq \|y_i^{k+1} - x^*\| + \|y_i^k - x^*\| \leq 2\bR$.
\end{proof}

\begin{lemma}[\cite{nesterov2008accelerating},~Lemma~2]
    \label{lem:dual}
    Let $g(z)=\tfrac{\theta}{p}\|z\|^{p}$ for $p \ge 2$ and $g^{*}$ be its conjugate function i.e., $g^{*}(v)=\sup _{z}\{\langle v, z\rangle-$ $g(z)\} .$ Then, we have
    $$
        g^{*}(v)=\tfrac{p-1}{p}\left(\tfrac{\|v\|^{p}}{\theta}\right)^{\frac{1}{p-1}}
    $$
    Moreover, for any $v, z \in \mathbb{R}^{n}$, we have $g(z)+g^{*}(v)-\langle z, v\rangle \ge 0 .$
\end{lemma}

\begin{lemma}[Two-sided estimating-sequence bound]
    \label{lem:estimating_sequence_final}
    Let Assumptions~\ref{as:L1},~\ref{as:L2}, and Assumption~\ref{as:strcnvxty} hold with $\mu_i \ge 0$. Let~\eqref{eq:acc_case2_params} hold and $\{\bk_{2,k}\}_{k \geq 1}$,~$\{\bk_{3,k}\}_{k \geq 1}$ be non-negative non-decreasing sequences such that for all $k \ge 1$ 
    \begin{equation}
        \label{eq:kappas}
        \begin{gathered}
             2\bk_{2, k-1} + \lambda_k \ge \tfrac{6}{\sqrt{5}} \tfrac{\alpha_k^2}{A_k}(2\Dtwo{\hv}{k}+\delta_{2,k}),\\
            2\bk_{3, k-1} + 3 \lambda_k \bR^{-1} \ge \tfrac{32}{5} \tfrac{\alpha_k^3}{A_k}(2\Lavg_2+L),
        \end{gathered}
    \end{equation}
    where $\|y_i^{k} - x^*\| \leq \bR~ \forall k \ge 1, ~ i \in [\numMach]$ and $\lambda_k = \ls\tfrac{1}{A_{k-1}} - 1 \rs \muavg$.
    Then, the following holds for each machine $i \in [\numMach]$ and for all iterations $k \ge 1$
    \begin{equation*}
        \begin{aligned}
            \tfrac{f(x_i^{k+1})}{A_{k}} \le~ & \psi_i^{k+1}(y_i^{k+1}) - \err_i^{k+1} + \berr_i^{k+1} \\
            \le~& \psi_i^{k+1}(x^*) - \err_i^{k+1} + \berr_i^{k+1} \\
            \le~ &   \tfrac{f(x^*)}{A_{k}}
            + \Done{\hv}{0}\|x_i^1-\hv_i^0\| + \Done{\hv}{0}\|x^*-\hv_i^0\| \\
            & + \tfrac{\Dtwo{\hv}{0}+\delta_{2,0}+\bk_{2,k}}{2}\|x^*-\hv_i^0\|^2 \\
            & + \tfrac{\Lavg_2+L+\bk_{3,k}}{6}\|x^*-\hv_i^0\|^3 \\
            & + \tsum_{j=1}^{k}\tfrac{\alpha_j}{A_j}
            \la \hg_i(x_i^{j+1})-\nabla f(x_i^{j+1}), x^*-y_i^{j+1}\ra \\
            & + \tsum_{j=1}^{k}\tfrac{1}{A_j}
            \la \nabla f(x_i^{j+1}), \hv_i^j -  v_i^j \ra + \berr_i^{k+1},
        \end{aligned}  
    \end{equation*}
    where
    \begin{equation}
        \label{eq:berr}
        \begin{aligned}
            \berr_i^{k+1} = \tsum_{j=1}^k \left( \tfrac{1}{A_{j}}   \right. &  \left.\|\hv_i^j - x_i^{j+1}\| + \tfrac{\alpha_j}{A_j} \|y_i^{j+1} - y_i^j\|\right)\\ 
            &\times
                \ls
                    \max\left\{
                        \tfrac{\delta_{2,j}}{\Dtwo{\hv}{j}},
                        \tfrac{L}{\Lavg_2}
                    \right\}
                    + 2
                \rs \Done{\hv}{j},
        \end{aligned}
    \end{equation}
    for $k \ge 1$ and $\berr_i^{1} = 0$.
\end{lemma}

\begin{proof}
    We aim to show by induction that the following holds for all machines $i \in [\numMach]$ and $k \ge 1$
    \begin{equation*}
        \psi_i^k(y_i^k) \ge \tfrac{f(x_i^k)}{A_{k-1}} + \err_i^k -\berr_i^{k},
    \end{equation*}
    where $\err_i^k$ is defined in~\eqref{eq:err}. From~\eqref{eq:acc_psi1}, since $A_0=1$, we have that $\frac{f(x_i^1)}{A_0} \le \psi_i^1(y_i^1)$.
    
    Let us assume that
    \begin{equation*}
        \psi_i^k(y_i^k) \ge \tfrac{f(x_i^k)}{A_{k-1}} + \err_i^k -\berr_i^{k},
    \end{equation*}
    and show that~$\psi_i^{k+1}(y_i^{k+1}) \ge \tfrac{f(x_i^{k+1})}{A_k} + \err_i^{k+1} -\berr_i^{k+1}$. By applying Lemma~\ref{lem:acc_psi_lower} with $r_i^k = \berr_i^k$,
    \begin{equation}
    \label{eq:show_pos}
        \begin{aligned}
            \psi_i^{k+1}(y_i^{k+1})
            \ge~
            & \tfrac{f(x_i^{k+1})}{A_{k}} + \tfrac{1}{A_{k}} \la \nabla f(x_i^{k+1}), \hv_i^k - x_i^{k+1} \ra \\
            & + \tfrac{2\bk_{2,k-1} + \lambda_k}{4}\|y_i^{k+1} - y_i^k\|^2 \\
            & + \tfrac{2\bk_{3,k-1} + 3\lambda_k \bR^{-1}}{24}\|y_i^{k+1} - y_i^k\|^3 \\
            & + \tfrac{\alpha_k}{A_k} \la \nabla f(x_i^{k+1}), y_i^{k+1} - y_i^k \ra \\
            & + \err_i^{k+1} - \berr_i^k
        \end{aligned}
    \end{equation}
    Thus, to complete the induction step, it remains to show that the sum of all terms on the right-hand side, except for $\tfrac{f(x_i^{k+1})}{A_k}$ and the error terms, is non-negative.

    We begin with the first case~\eqref{eq:acc_case1_cond} in Lemma~\ref{lem:acc_step_progress}. By~\eqref{eq:acc_case1_bound} we have
    \begin{equation*}
        \begin{aligned}
            \tfrac{1}{A_{k}} &\la \nabla f(x_i^{k+1}), \hv_i^k - x_i^{k+1} \ra + \tfrac{\alpha_k}{A_k} \la \nabla f(x_i^{k+1}), y_i^{k+1} - y_i^k \ra \\
            &\ge -  \ls \tfrac{1}{A_{k}}  \|\hv_i^k - x_i^{k+1}\| + \tfrac{\alpha_k}{A_k} \|y_i^{k+1} - y_i^k\|\rs \|\nabla f(x_i^{k+1})\| \\
            &\ge -  \ls \tfrac{1}{A_{k}}  \|\hv_i^k - x_i^{k+1}\| + \tfrac{\alpha_k}{A_k} \|y_i^{k+1} - y_i^k\|\rs \\
            & \qquad \times
             \ls
                \max\left\{
                    \tfrac{\delta_{2,k}}{\Dtwo{\hv}{k}},
                    \tfrac{L}{\Lavg_2}
                \right\}
                + 2
            \rs \Done{\hv}{k} 
        \end{aligned}
    \end{equation*}
    Therefore, by~\eqref{eq:berr} and~\eqref{eq:show_pos},
    \begin{equation*}
        \begin{aligned}
            \psi_i^{k+1}(y_i^{k+1})
            \ge~
            & \tfrac{f(x_i^{k+1})}{A_{k}} + \err_i^{k+1} - \berr_i^{k+1}
        \end{aligned}
    \end{equation*}
    Next, let us move to the second case in Lemma~\ref{lem:acc_step_progress}. When~\eqref{eq:acc_case2_cond} is satisfied Lemma~\ref{lem:acc_step_progress} provides a lower bound for 
    $\langle \nabla f(x_i^{k+1}), \hv_i^k - x_i^{k+1}\rangle$.
    Consider the case when the minimum on the right-hand side of~\eqref{eq:acc_case2_bound}
    is achieved by the first term.
    Then, applying Lemma~\ref{lem:dual} with the following choice of parameters,
    \begin{equation*}
        z = y_i^k - y_i^{k+1}, ~~ v = \tfrac{\alpha_k}{A_k}\nabla f(x_i^{k+1}), ~~ \theta = \tfrac{2\bk_{2, k-1} + \lambda_k}{2},
    \end{equation*}
    we have
    \begin{equation*}
        \begin{aligned}
            \tfrac{2\bk_{2,k-1} + \lambda_k}{4}  \|y_i^{k+1} - y_i^k\|^2 + &  \tfrac{\alpha_k}{A_k} \la \nabla f(x_i^{k+1}), y_i^{k+1} - y_i^k \ra \\
            & \ge - \tfrac{1     }{2\bk_{2,k-1} + \lambda_k}\|\tfrac{\alpha_k}{A_k} \nabla f(x_i^{k+1})\|^2.
        \end{aligned}
    \end{equation*}
    Hence,
    \begin{equation*}
        \begin{aligned}
            \psi_i^{k+1}(y_i^{k+1})
            \ge~
            & \tfrac{f(x_i^{k+1})}{A_{k}} + \tfrac{1}{A_{k}} \la \nabla f(x_i^{k+1}), \hv_i^k - x_i^{k+1} \ra \\
            & + \tfrac{2\bk_{2,k-1} + \lambda_k}{4}\|y_i^{k+1} - y_i^k\|^2 \\
            & + \tfrac{2\bk_{3,k-1} + 3\lambda_k \bR^{-1}}{24}\|y_i^{k+1} - y_i^k\|^3 \\
            & + \tfrac{\alpha_k}{A_k} \la \nabla f(x_i^{k+1}), y_i^{k+1} - y_i^k \ra + \err_i^{k+1} \\
            \ge~ 
            & \tfrac{f(x_i^{k+1})}{A_{k}} + \tfrac{\sqrt{5}}{3 A_k} \tfrac{\|\nabla f(x_i^{k+1})\|^2}{2(2\Dtwo{\hv}{k}+\delta_{2,k})} \\
            & - \tfrac{1}{2\bk_{2,k-1} + \lambda_k}\|\tfrac{\alpha_k}{A_k} \nabla f(x_i^{k+1})\|^2 + \err_i^{k+1} \\
            \ge~ 
            & \tfrac{f(x_i^{k+1})}{A_{k}} + \err_i^{k+1} - \berr_i^k. 
        \end{aligned}
    \end{equation*}
    where the last inequality holds by our choice of the parameter~$\bk_{2, k-1}$~\eqref{eq:kappas}
    \begin{equation*}
        2\bk_{2, k-1} + \lambda_k \ge \tfrac{6}{\sqrt{5}} \tfrac{\alpha_k^2}{A_k}(2\Dtwo{\hv}{k}+\delta_{2,k}).
    \end{equation*}

     Next, we consider the case when the minimum in the RHS of \eqref{eq:acc_case2_bound} is achieved on the second term. Again, by Lemma \ref{lem:dual} with the same choice of $z, v$ and with $\theta = \tfrac{2\bk_{3,k-1} + 3\lambda_k \bR^{-1}}{8}$, we have
    \begin{equation*}
        \begin{aligned}
            \tfrac{2\bk_{3,k-1} + 3\lambda_k \bR^{-1}}{24} & \|y_i^{k+1} - y_i^k\|^3 +   \tfrac{\alpha_k}{A_k} \la \nabla f(x_i^{k+1}), y_i^{k+1} - y_i^k \ra \\
            & \ge - \tfrac{4\sqrt{2}}{3\sqrt{2\bk_{3,k-1} + 3\lambda_k \bR^{-1}}}\|\tfrac{\alpha_k}{A_k} \nabla f(x_i^{k+1})\|^{3/2}.
        \end{aligned}
    \end{equation*}
    Hence, we get
        \begin{equation*}
        \begin{aligned}
            \psi_i^{k+1}(y_i^{k+1})
            \ge~
            & \tfrac{f(x_i^{k+1})}{A_{k}} + \tfrac{1}{A_{k}} \la \nabla f(x_i^{k+1}), \hv_i^k - x_i^{k+1} \ra \\
            & + \tfrac{2\bk_{2,k-1} + \lambda_k}{4}\|y_i^{k+1} - y_i^k\|^2 \\
            & + \tfrac{2\bk_{3,k-1} + 3 \lambda_k \bR^{-1}}{24}\|y_i^{k+1} - y_i^k\|^3 \\
            & + \tfrac{\alpha_k}{A_k} \la \nabla f(x_i^{k+1}), y_i^{k+1} - y_i^k \ra + \err_i^{k+1} \\
            \ge~ 
            & \tfrac{f(x_i^{k+1})}{A_{k}} 
            + \tfrac{\sqrt{5}}{3 A_k} \tfrac{\|\nabla f(x_i^{k+1})\|^{3/2}}{(2\Lavg_2+L)^{1/2}} \\
            & - \tfrac{4\sqrt{2}}{3\sqrt{2\bk_{3,k-1} + 3 \lambda_k \bR^{-1}}}\|\tfrac{\alpha_k}{A_k} \nabla f(x_i^{k+1})\|^{3/2} \\
            \ge~ 
            & \tfrac{f(x_i^{k+1})}{A_{k}} + \err_i^{k+1} - \berr_i^k. 
        \end{aligned}
    \end{equation*}
    where the last inequality holds by our choice of the parameter $\bk_{3, k-1}$~\eqref{eq:kappas}
    \begin{equation*}
        2\bk_{3, k-1} + 3 \lambda_k \bR^{-1} \ge \tfrac{32}{5} \tfrac{\alpha_k^3}{A_k}(2\Lavg_2+L).
    \end{equation*}

    To sum up, by our choice of the parameters $\bk_{i, k}$, $i=2, 3$ we obtain from~\eqref{eq:show_pos} that
    \begin{equation*}
        \begin{aligned}
        \psi_i^{k+1}(y_i^{k+1}) & \ge \tfrac{f(x_i^{k+1})}{A_k} + \err_i^{k+1} - \berr_i^{k}\\
        & \ge \tfrac{f(x_i^{k+1})}{A_k} + \err_i^{k+1} - \berr_i^{k+1}.
        \end{aligned}
    \end{equation*}

    Therefore, in both cases we obtained the bound
    \begin{equation*}
        \psi_i^{k+1}(y_i^{k+1}) \ge  \tfrac{f(x_i^{k+1})}{A_k} + \err_i^{k+1} - \berr_i^{k+1},
    \end{equation*}
    and, thus, by induction, we obtain that, for all $k \ge 1$,
    \begin{equation*}
        \begin{aligned}
            \tfrac{f(x_i^{k+1})}{A_{k}}&~ + \err_i^{k+1} -  \berr_i^{k+1} \le \psi_i^{k+1}(y_i^{k+1}) \le \psi_i^{k+1}(x^*) \\
            \stackrel{\eqref{eq:acc_psi_upper_bound}}{\le} &   \tfrac{f(x^*)}{A_{k}}
            + \Done{\hv}{0}\|x_i^1-\hv_i^0\| + \Done{\hv}{0}\|x^*-\hv_i^0\| \\
            & + \tfrac{\Dtwo{\hv}{0}+\delta_{2,0}+\bk_{2,k}}{2}\|x^*-\hv_i^0\|^2 \\
            & + \tfrac{\Lavg_2+L+\bk_{3,k}}{6}\|x^*-\hv_i^0\|^3 \\
            & + \tsum_{j=1}^{k}\tfrac{\alpha_j}{A_j}
            \la \hg_i(x_i^{j+1})-\nabla f(x_i^{j+1}),\, x^*-x_i^{j+1}\ra \\
        \end{aligned}  
    \end{equation*}
    Plugging in~\eqref{eq:err} and rearranging terms finishes the proof.  
\end{proof}

In the next two lemmas we provide bounds for terms $\la \nabla f(x_i^{j+1}), \hv_i^j -  v_i^j \ra$ and $\la \hg_i(x_i^{j+1})-\nabla f(x_i^{j+1}), x^*-y_i^{j+1}\ra$. First, let us introduce the following notation for the dissimilarity of $x_i^k$ on machines:
\begin{equation}
    \label{eq:Delta_raw}
    \Delta_{x, k}^{\mathrm{raw}} = \max_{i, j \in [\numMach]} \|x_i^k - x_j^k\|.
\end{equation}

\begin{lemma}[Bound for the consensus-on-$v$ error term]
\label{lem:err_v}
    Let Assumption~\ref{as:L1} hold.
    Then, for any $j \ge 1$
    \begin{equation*}
        \tfrac{1}{\numMach}\tsum_{i=1}^{\numMach}
        \la \nabla f(x_i^{j+1}), \hv_i^j - v_i^j \ra
        \le
        \tfrac{\Lavg_1 \Delta_{x,j+1}^{\mathrm{raw}}}{\numMach}
        \tsum_{i=1}^{\numMach}\|\hv_i^j - v_i^j\|.
    \end{equation*}  
\end{lemma}
\begin{proof}
    Since the consensus routine preserves averages by Definition~\ref{assum:mixing_matrix_sequence}, we have $\bv^k = \tfrac{1}{\numMach} \tsum_{i=1}^{\numMach} \hv_i^k = \tfrac{1}{\numMach} \tsum_{i=1}^{\numMach} v_i^k$. Therefore,
    \begin{equation*}
        \begin{aligned}
            \tfrac{1}{\numMach} \textstyle{\tsum_{i=1}^{\numMach}}& \la \nabla f(x_i^{j+1}), \hv_i^j -  v_i^j \ra \\
            &= \tfrac{1}{\numMach} \textstyle{\tsum_{i=1}^{\numMach}} \la \nabla f(x_i^{j+1}) - \nabla f(\bx^{j+1}), \hv_i^j -  v_i^j \ra \\
            & \le \tfrac{1}{\numMach} \textstyle{\tsum_{i=1}^{\numMach}} \|\nabla f(x_i^{j+1}) - \nabla f(\bx^{j+1})\| \|\hv_i^j -  v_i^j\| \\
            & \le \tfrac{1}{\numMach} \tsum_{i=1}^{\numMach} \Lavg_1 \|x_i^{j+1} - \bx^{j+1}\| \|\hv_i^j -  v_i^j\| \\
        \end{aligned}
    \end{equation*}
    By definition of $\bx^{j+1}$
    \begin{equation*}
        \begin{aligned}
            \|x_i^{j+1} - \bx^{j+1}\| 
            = &~ \|x_i^{j+1} - \tfrac{1}{\numMach} \tsum_{\ell=1}^{\numMach} x_\ell^{j+1}\| \\
            \le &~ \tfrac{1}{\numMach} \tsum_{\ell=1}^{\numMach} \|x_i^{j+1} - x_\ell^{j+1}\| \le \Delta_{x, j+1}^{\mathrm{raw}}.
        \end{aligned}
    \end{equation*}
    Hence, we have
    \begin{equation*}
        \begin{aligned}
            \tfrac{1}{\numMach} \textstyle{\tsum_{i=1}^{\numMach}}& \la \nabla f(x_i^{j+1}), \hv_i^j -  v_i^j \ra \\
            & \le \tfrac{1}{\numMach}  \Lavg_1 \Delta_{x, j+1}^{\mathrm{raw}} \tsum_{i=1}^{\numMach} \|\hv_i^j -  v_i^j\|.
        \end{aligned}
    \end{equation*}
\end{proof}

\begin{lemma}[Bound for the gradient-consensus error term]
\label{lem:err_g}
Let Assumption~\ref{as:L1} hold. For any $j \ge 1$
\begin{equation*}
    \begin{aligned}
        \tfrac{1}{\numMach}\tsum_{i=1}^{\numMach}
        &\la \hg_i(x_i^{j+1})-\nabla f(x_i^{j+1}), x^*-y_i^{j+1}\ra \\
        & \le~
        \left(\hDeg{x}{j+1}+\Lavg_1\Delta_{x,j+1}^{\mathrm{raw}}\right)
        \tfrac{1}{\numMach}\tsum_{i=1}^{\numMach}\|x^*-y_i^{j+1}\|.
    \end{aligned}
\end{equation*}
\end{lemma}
\begin{proof}
    \begin{equation*}
        \begin{aligned}
            \tfrac{1}{\numMach} \textstyle{\tsum_{i=1}^{\numMach}} & \la \hg_i(x_i^{j+1})-\nabla f(x_i^{j+1}), x^*-y_i^{j+1}\ra \\
            & \le \tfrac{1}{\numMach} \textstyle{\tsum_{i=1}^{\numMach}} \|\hg_i(x_i^{j+1})-\nabla f(x_i^{j+1})\| \|x^*-y_i^{j+1}\|
        \end{aligned}
    \end{equation*}
    Next, 
    \begin{equation*}
        \begin{aligned}
            \|\hg_i(x_i^{j+1})-~& \nabla f(x_i^{j+1})\| \le \|\hg_i(x_i^{j+1})- \tfrac{1}{\numMach} \tsum_{\ell=1}^{\numMach} \nabla f_\ell(x_\ell^{j+1})\| \\
            & +  \|\tfrac{1}{\numMach} \tsum_{\ell=1}^{\numMach} \nabla f_\ell(x_\ell^{j+1}) - \tfrac{1}{\numMach} \tsum_{\ell=1}^{\numMach} \nabla f_\ell(x_i^{j+1})\| \\
            & \le \hDeg{x}{j+1} + \Lavg_1 \Delta_{x, j+1}^{\mathrm{raw}} \\
        \end{aligned}
    \end{equation*}
    where the last inequality holds by definition of $\hDeg{x}{j+1}$~\eqref{eq:cons_bounds_derivatives}, Assumption~\ref{as:L1}, and~\eqref{eq:Delta_raw}. Substituting this bound into the first inequality proves Lemma.
\end{proof}

The last technical lemma of this section provides a bound for the raw disagreement $\Delta_{x, k+1}^{\mathrm{raw}}$~\eqref{eq:Delta_raw}.
\begin{lemma}[Progress of the cubic step and raw disagreement]
\label{lem:raw_dissimilarity}
Let Assumption~\ref{as:strcnvxty} hold with $\mu_i > 0$ for all $i \in [\numMach]$.
Then, 
\begin{align}
        \notag
        \Delta_{x,k+1}^{\mathrm{raw}}
        \le~ &
        \tfrac{1}{\mumin+\delta_{2,k}}
        (
            2\hDeH{\hv}{k}
            \max_{\ell\in[\numMach]}\|x_\ell^{k+1}-\hv_\ell^k\|
            + 2\hDeg{\hv}{k}
        ) \\
        \label{eq:raw_dissimilarity_bound}
        &+ 2\hDex{\hv}{k}.
\end{align}
\end{lemma}

\begin{proof}
By the optimality condition for the cubic subproblem~\eqref{eq:acc_step}, for any $i\in[\numMach]$,
\begin{equation*}
    0
    =
    \nabla \hphi_i^k(x_i^{k+1};\hv_i^k)
    + \delta_{2,k}(x_i^{k+1}-\hv_i^k)
    + \tfrac{L}{2}\|x_i^{k+1}-\hv_i^k\|(x_i^{k+1}-\hv_i^k).
\end{equation*}
Let $ M_i^k \eqdef \hH_i(\hv_i^k) + \delta_{2,k} I$ and $s_i^k \eqdef x_i^{k+1}-\hv_i^k$, then this can be rewritten as
\begin{equation}
    \label{eq:opt_cond_raw_sk}
    -\hg_i(\hv_i^k)
    =
    M_i^k s_i^k
    + \tfrac{L}{2}\|s_i^k\|s_i^k.
\end{equation}
Similarly, for any $\ell\in[\numMach]$,
\begin{equation}
    \label{eq:opt_cond_raw_sl}
    -\hg_\ell(\hv_\ell^k)
    =
    M_\ell^k s_\ell^k
    + \tfrac{L}{2}\|s_\ell^k\|s_\ell^k.
\end{equation}
Subtracting~\eqref{eq:opt_cond_raw_sl} from~\eqref{eq:opt_cond_raw_sk}, we obtain
\begin{equation*}
    \begin{aligned}
        - & \ls \hg_i(\hv_i^k) -\hg_\ell(\hv_\ell^k) \rs
        \\ 
        &=
        M_i^k(s_i^k-s_\ell^k)
        + (M_i^k-M_\ell^k)s_\ell^k
        + \tfrac{L}{2}\bigl(\|s_i^k\|s_i^k-\|s_\ell^k\|s_\ell^k\bigr).
    \end{aligned}
\end{equation*}
Taking the inner product with $s_i^k-s_\ell^k$ gives
\begin{equation}
    \label{eq:inner_product_raw}
    \begin{aligned}
        \left \langle M_i^k(s_i^k-s_\ell^k),  s_i^k - \right. & \left. s_\ell^k  \right \rangle
         + \la (M_i^k-M_\ell^k)s_\ell^k, s_i^k-s_\ell^k\ra \\
        &
        + \tfrac{L}{2}\la \|s_i^k\|s_i^k-\|s_\ell^k\|s_\ell^k, s_i^k-s_\ell^k\ra \\
        & =
        -\la \hg_i(\hv_i^k)-\hg_\ell(\hv_\ell^k), s_i^k-s_\ell^k\ra.
    \end{aligned}
\end{equation}

Since each $f_i$ is $\mu_i$-strongly convex, by Assumption~\eqref{assum:mixing_matrix_sequence} we have
\begin{equation*}
    \label{eq:Hk_pd}
    \hH_i (\hv_i^k) \succeq \min \limits_{i \in [\numMach]} \mu_i I \eqdef \mumin I,~ \forall i \in [\numMach],~k \ge 0.
\end{equation*}
Therefore,
\begin{equation}
    \label{eq:strong_monotone_A}
    \la M_i^k(s_i^k-s_\ell^k), s_i^k-s_\ell^k\ra
    \ge
    (\mumin+\delta_{2,k})\|s_i^k-s_\ell^k\|^2.
\end{equation}
Moreover, the function $s \mapsto \tfrac{1}{3}\|s\|^3$ is convex, and therefore(see e.g.~\cite[Theorem 2.1.3]{nesterov2018lectures})
\begin{equation}
    \label{eq:cubic_monotonicity}
    \la \|s_i^k\|s_i^k-\|s_\ell^k\|s_\ell^k, s_i^k-s_\ell^k\ra \ge 0.
\end{equation}
Combining~\eqref{eq:inner_product_raw}, \eqref{eq:strong_monotone_A}, and~\eqref{eq:cubic_monotonicity}, we get
\begin{equation}
    \label{eq:basic_raw_bound}
    \begin{aligned}
        (\mumin+\delta_{2,k})\|s_i^k-s_\ell^k\|^2
        \le~ &
        \|(M_i^k-M_\ell^k)s_\ell^k\|\,\|s_i^k-s_\ell^k\| \\
        & +
        \|\hg_i(\hv_i^k)-\hg_\ell(\hv_\ell^k)\|\,\|s_i^k-s_\ell^k\|.
    \end{aligned}
\end{equation}
 Assume $s_i^k \ne s_\ell^k$, dividing both sides of~\eqref{eq:basic_raw_bound} by $\|s_i^k-s_\ell^k\|$, we obtain
\begin{equation}
    \label{eq:divide_raw_bound}
    (\mumin+\delta_{2,k})\|s_i^k-s_\ell^k\|
    \le
    \|M_i^k-M_\ell^k\|\,\|s_\ell^k\|
    +
    \|\hg_i(\hv_i^k)-\hg_\ell(\hv_\ell^k)\|.
\end{equation}
Let $\bH_v^k \eqdef \tfrac{1}{\numMach}\tsum_{j=1}^{\numMach} \nabla^2 f_j(\hv_j^k)$. Since $M_i^k-M_\ell^k=\hH_i(\hv_i^k)-\hH_\ell(\hv_\ell^k)$, by~\eqref{eq:cons_bounds_derivatives} we have
\begin{equation}
    \label{eq:hessian_diff_bound}
    \begin{aligned}
        \|M_i^k-M_\ell^k\|
        =~
        & \|\hH_i(\hv_i^k)-\hH_\ell(\hv_\ell^k)\| \\
        \le~
        & \|\hH_i(\hv_i^k)-\bH_v^k\| + \|\hH_\ell(\hv_\ell^k)-\bH_v^k\| \\
        \le~ &
        2\hDeH{\hv}{k}.
    \end{aligned}
\end{equation}
Similarly,
\begin{equation}
    \label{eq:gradient_diff_bound}
    \|\hg_i(\hv_i^k)-\hg_\ell(\hv_\ell^k)\|
    \le
    2\hDeg{\hv}{k}.
\end{equation}
Substituting~\eqref{eq:hessian_diff_bound} and~\eqref{eq:gradient_diff_bound} into~\eqref{eq:divide_raw_bound} yields
\begin{equation}
    \label{eq:step_progress_si}
    \|s_i^k - s_\ell^k\|
    \le
    \tfrac{1}{\mumin+\delta_{2,k}}
    \left(
        2\hDeH{\hv}{k} \|s_\ell^k\|
        + 2\hDeg{\hv}{k}
    \right).
\end{equation}
If $s_i^k=s_\ell^k$, then~\eqref{eq:step_progress_si} is trivial.

Finally, observe that
\begin{equation*}
    x_i^{k+1}-x_\ell^{k+1}
    =
    (s_i^k-s_\ell^k)
    + (\hv_i^k-\hv_\ell^k).
\end{equation*}
Hence,
\begin{equation*}
    \|x_i^{k+1}-x_\ell^{k+1}\|
    \le
    \|s_i^k-s_\ell^k\|
    +
    \|\hv_i^k-\hv_\ell^k\|.
\end{equation*}
Using~\eqref{eq:step_progress_si} and~\eqref{eq:cons_bounds_x}, we get
\begin{equation*}
    \begin{aligned}
        \|x_i^{k+1}-x_\ell^{k+1}\|
        \le~
        & \tfrac{1}{\mumin+\delta_{2,k}}
        \left(
            2\hDeH{\hv}{k} \|s_\ell^k\|
            + 2\hDeg{\hv}{k}
        \right) \\
        & + \|\hv_i^k-\bv^k\|
        + \|\hv_\ell^k-\bv^k\| \\
        \le~
        & \tfrac{1}{\mumin+\delta_{2,k}}
        \left(
            2\hDeH{\hv}{k} \|x_\ell^{k+1}-\hv_\ell^k\|
            + 2\hDeg{\hv}{k}
        \right) \\
        & + 2\hDex{\hv}{k}.
    \end{aligned}
\end{equation*}
Taking the maximum over $i,\ell\in[\numMach]$ proves~\eqref{eq:raw_dissimilarity_bound}.
\end{proof}

Next, we finally combine all the results above to provide convergence guarantees.
\begin{theorem}
    \label{thm:acc_convergence}
    Let Assumption~\ref{as:strcnvxty},~\ref{as:L1},~\ref{as:L2}, and ~\ref{as:boundnes} hold. Define $\Dtwoo{\hv} \eqdef \max \limits_{0 \leq k \leq \numIter} \Dtwo{\hv}{k} $, where $\numIter$ is the total number of iterations, and let us choose $\forall k \geq 0$
    \begin{equation}
        \label{eq:acc_params}
        L = 3\Lavg_2, \qquad \delta_{2,k} = 3 \Dtwo{\hv}{k},
    \end{equation} 
    where $\Done{\hv}{k}$ and $\Dtwo{\hv}{k}$ are defined in~\eqref{eq:acc_Delta1_Delta2},
    \begin{equation}
        \label{eq:str_convex_alphas}
        \alpha_k = \alpha = \min \lb \tfrac{4}{5}, \ls \tfrac{\muavg}{30\sqrt{5}\Dtwoo{\hv}}\rs^{1/2}, \ls \tfrac{3\muavg\bR^{-1}}{160 \Lavg_2} \rs^{1/3}\rb,
    \end{equation}
    and
    \begin{equation}
        \label{eq:acc_kappas}
        \bk_{2, k} = \bk_2 = \tfrac{\muavg}{2} ,
        ~~~
        \bk_{3, k} = \bk_3 = \tfrac{3}{2}\muavg \bR^{-1}.
    \end{equation} 
    Then for all $k \geq 1$ we have
    \begin{equation}
        \label{eq:acc_thm_rate1}
        \begin{aligned}
            &f(\bx^{k+1})-f(x^*)
            \\
            &\le (1 - \alpha)^k 
            \left(\tfrac{4 \Done{\hv}{0}}{\muavg\|\bx^0 - x^*\|} + \tfrac{4\Done{\hv}{0}\bR}{\muavg\|\bx^0 - x^*\|^2}  \right.  \\
            & \quad  \quad  \quad  \quad  \quad  \quad + \tfrac{4\Dtwo{\hv}{0}}{\muavg} + \tfrac{1}{2}\\
            & \quad  \quad  \quad  \quad  \quad  \quad \left. +  \tfrac{8\Lavg_2 + 3\muavg \bR^{-1}}{6\muavg} \|\bx^0 - x^*\|  \right) \ls f(\bx^{0})-f(x^*) \rs \\
            & \quad + \tsum_{j=1}^{k}\alpha (1 - \alpha)^{k - j}
            \left(
                \hDeg{x}{j+1}
                + \tfrac{2\Lavg_1}{\mumin}
                \ls
                    2\hDeH{\hv}{j}\bR
                    + \hDeg{\hv}{j}
                \rs 
            \right.\\
            & \quad  \quad \quad  \quad  \quad \quad  \quad  \quad  \quad \quad \left.
                + 2\Lavg_1 \hDex{\hv}{j}
            \right) \bR \\
            & \quad + 2\tsum_{j=1}^{k} (1 - \alpha)^{k - j} \left( \tfrac{2\Lavg_1}{\mumin}
                \ls
                    2\hDeH{\hv}{j}  \bR
                    + \hDeg{\hv}{j}
                \rs 
            \right. \\
            & \quad \quad \quad \quad \quad \quad \quad \quad \quad  \quad + \left. 2\Lavg_1 \hDex{\hv}{j} \right) \bR. \\
            & \quad + 10\bR \tsum_{j=1}^{k} (1 + \alpha)(1 - \alpha)^{k - j}\Done{\hv}{j}.
        \end{aligned}
    \end{equation}
\end{theorem}
\begin{proof}
    By convexity of $f$,
    \begin{equation}
        \label{eq:acc_thm_jensen}
        f(\bx^{k+1})\le \tfrac{1}{\numMach}\tsum_{i=1}^{\numMach} f(x_i^{k+1}).
    \end{equation}
    First, let us average Lemma~\ref{lem:estimating_sequence_final} over $i\in[\numMach]$ and multiply the resulting inequality by $A_k$. We obtain
    \begin{equation*}
        \begin{aligned}
            &f(\bx^{k+1})-f(x^*)
            \\
            &\le A_k\tfrac{1}{\numMach}\tsum_{i=1}^{\numMach}
            \left[\Done{\hv}{0}\|x_i^1-\hv_i^0\|+\Done{\hv}{0}\|x^*-\hv_i^0\|\right]
            \\
            &\quad +A_k\tfrac{\Dtwo{\hv}{0}+\delta_{2,0}+\bk_{2,k}}{2\numMach}
            \tsum_{i=1}^{\numMach}\|x^*-\hv_i^0\|^2
            \\
            &\quad +A_k\tfrac{\Lavg_2+L+\bk_{3,k}}{6\numMach}
            \tsum_{i=1}^{\numMach}\|x^*-\hv_i^0\|^3 \\
            &\quad +
            A_k\tsum_{j=1}^{k}\tfrac{\alpha_j}{A_j}
            \tfrac{1}{\numMach}\tsum_{i=1}^{\numMach}
            \la \hg_i(x_i^{j+1})-\nabla f(x_i^{j+1}),x^*-y_i^{j+1} \ra \\
            &\quad + A_k\tsum_{j=1}^{k}\tfrac{1}{A_j}\tfrac{1}{\numMach}\tsum_{i=1}^{\numMach}
            \la \nabla f(x_i^{j+1}),\hv_i^j-v_i^j \ra + \tfrac{A_k}{\numMach} \tsum_{i=1}^{\numMach}\berr_i^{k+1}
        \end{aligned}
    \end{equation*}

    Now, lets bound each of the terms in above inequality. 
    By initialization $v_i^0=x_i^0$, hence $\hv_i^0=\bv^0=\bx^0$.
    First, by the definition of $\Done{\hv}{0}$ 
    \begin{equation*}
        \begin{aligned}
            A_k\tfrac{1}{\numMach}\tsum_{i=1}^   {\numMach}
            \left[\Done{\hv}{0}\right.&\left.\|x_i^1-\bx^0\| + \Done{\hv}{0}\|x^*-\bx^0\|\right]  \\
            &~\leq A_k\Done{\hv}{0} \ls 2\|x^*-\bx^0\| + \bR \rs. 
        \end{aligned}
    \end{equation*}
    Next, by our choice of $\delta_{2,0}$\eqref{eq:acc_params}
    \begin{equation*}
        \begin{aligned}
            A_k\tfrac{\Dtwo{\hv}{0}+\delta_{2,0}}{2\numMach}
            \tsum_{i=1}^{\numMach}\|x^*-\hv_i^0\|^2 
            &\leq 2A_k\Dtwo{\hv}{0}\|x^*-\bx^0\|^2. 
        \end{aligned}
    \end{equation*}
    Further, by our choice of $\delta_{2,k}$\eqref{eq:acc_params} and $\bk_{2, k}$\eqref{eq:acc_kappas},
    \begin{equation*}
        \begin{aligned}
            \tfrac{A_k \bk_{2,k}}{2\numMach} \tsum_{i=1}^{\numMach}&\|x^*-\hv_i^0\|^2 = \tfrac{A_k\bk_{2,k}}{2}\|x^*-\bx^0\|^2 =
            \tfrac{\muavg}{4}A_k\|x^*-\bx^0\|^2 
        \end{aligned}
    \end{equation*}
    For the cubic term, again using~\eqref{eq:acc_params},~\eqref{eq:acc_kappas},
    \begin{equation*}
        \begin{aligned}
            \tfrac{A_k}{6}(\Lavg_2+L&+\bk_{3,k})\|x^*-\bx^0\|^3
            \\
            &=
            A_k
            \ls
                \tfrac{2}{3}\Lavg_2
                + \tfrac{\muavg \bR^{-1}}{4}
            \rs \|x^*-\bx^0\|^3.
        \end{aligned}
    \end{equation*}
    It remains to bound the two accumulated error sums. First, by Lemma~\ref{lem:raw_dissimilarity}
    \begin{equation*}
        \begin{aligned}
            \Delta_{x,j+1}^{\mathrm{raw}}
            \le
            \tfrac{2}{\mumin+\delta_{2,j}}
            \ls
                2\hDeH{\hv}{j}\bR
                + \hDeg{\hv}{j}
            \rs
            + 2\hDex{\hv}{j}.
        \end{aligned}
    \end{equation*}
    By Lemma~\ref{lem:err_g}, for every $j\in\{1,\dots,k\}$,
    \begin{equation*}
        \begin{aligned}
            \tfrac{1}{\numMach} & \tsum_{i=1}^{\numMach}
            \la \hg_i(x_i^{j+1})-\nabla f(x_i^{j+1}), x^*-y_i^{j+1}\ra \\
            & \le
            \ls\hDeg{x}{j+1}+\Lavg_1\Delta_{x,j+1}^{\mathrm{raw}}\rs\bR\\
            & \le 
            \ls
                \hDeg{x}{j+1}
                + \tfrac{2\Lavg_1}{\mumin}
                \ls
                    2\hDeH{\hv}{j}\bR
                    + \hDeg{\hv}{j}
                \rs 
            + 2\Lavg_1 \hDex{\hv}{j}
            \rs \bR.
        \end{aligned}
    \end{equation*}
    Similarly, Lemma~\ref{lem:err_v} gives
    \begin{equation*}
        \begin{aligned}
            \tfrac{1}{\numMach}\tsum_{i=1}^{\numMach}&
            \la \nabla f(x_i^{j+1}), \hv_i^j-v_i^j\ra
            \le
            2\Lavg_1\Delta_{x,j+1}^{\mathrm{raw}}\bR \\
            &\le 2\ls \tfrac{2\Lavg_1}{\mumin}
                \ls
                    2\hDeH{\hv}{j}  \bR
                    + \hDeg{\hv}{j}
                \rs
            + 2\Lavg_1 \hDex{\hv}{j} \rs \bR.
        \end{aligned}
    \end{equation*}

    Finally, by~\eqref{eq:berr}, our choice of parameters~\eqref{eq:acc_params} and Assumption~\ref{as:boundnes}, we get
    \begin{equation*}
        \begin{aligned}
            \berr_i^{k+1}  = & \tsum_{j=1}^{k} \left( \tfrac{1}{A_{j}}   \right.   \left.\|\hv_i^j - x_i^{j+1}\| + \tfrac{\alpha_j}{A_j} \|y_i^{j+1} - y_i^j\|\right)\\ 
            &\qquad  \times
                \ls
                    \max\left\{
                        \tfrac{\delta_{2,j}}{\Dtwo{\hv}{j}},
                        \tfrac{L}{\Lavg_2}
                    \right\}
                    + 2
                \rs \Done{\hv}{j} \\
            \le & 10 \bR \tsum_{j=1}^{k} \tfrac{1 + \alpha_j}{A_{j}}\Done{\hv}{j}.
        \end{aligned}
    \end{equation*}
    Therefore, 
    \begin{equation*}
        \tfrac{1}{\numMach} \tsum_{i=1}^\numMach \berr_i^{k} \le 10 \bR \tsum_{j=1}^{k} \tfrac{1 + \alpha_j}{A_{j}}\Done{\hv}{j}.
    \end{equation*}
    Combining all the bounds together with definition of $A_k = (1 - \alpha)^k$~\eqref{eq:As} and $\|\bx^0 - x^*\|^2 \leq 2(f(\bx^0) - f(x^*))/\muavg$ we get~\eqref{eq:acc_thm_rate1}.
\end{proof}

\begin{customthm}{\ref{thm:acc_delta_choice}}
Let Assumptions~\ref{as:strcnvxty},~\ref{as:L1},~\ref{as:L2}, and Assumption~\ref{as:boundnes}
 hold. Let $\e>0$ be desired accuracy, $\numIter$ be total number of iterations, and 
\begin{equation*}
\alpha = \alpha_k = \min \lb
    \tfrac45,\,
    \ls\tfrac{3\muavg\bR^{-1}}{160\Lavg_2}\rs^{1/3}
\rb.
\end{equation*}
Let inexactness levels be chosen for every $k=0,\dots,\numIter$ as
\begin{subequations}
    \begin{align}
        \hDeH{\hv}{k} &~ \eqdef \hDe_{H \vert \hv} 
        \le
        \min \lb
            \tfrac{\muavg}{60\sqrt5 \alpha^2},
            \tfrac{\alpha\mumin\e}{320\Lavg_1\bR^2}
        \rb,
        \label{eq:acc_cor_DH}
        \\
        \hDeg{\hv}{k} &~ \eqdef \hDe_{g \vert \hv}
        \le
        \min \lb
            \tfrac{\alpha\mumin\e}{160\Lavg_1\bR},
            \tfrac{\alpha\e}{160\bR}
        \rb,
        \label{eq:acc_cor_Dg}
        \\
        \hDex{\hv}{k} &~ \eqdef \hDe_{\hv}
        \le
        \min \lb
            \tfrac{\muavg}{120\sqrt5 \alpha^2\Lavg_2},
            \tfrac{\alpha\e}{320\Lavg_1\bR}
        \rb,
        \label{eq:acc_cor_Dv}
        \\
        \hDeg{x}{k+1} &~ \eqdef \hDe_{g \vert x}
        \le
        \tfrac{\e}{8\bR}.
        \label{eq:acc_cor_Dgx}
    \end{align}
\end{subequations}
Let Algorithm~\ref{alg:acc_cubic} run with parameters chosen as for all $k=0,\dots,\numIter$,
\begin{equation*}
    L = 3\Lavg_2,
    \quad 
    \delta_{2,k}=3\Dtwo{\hv}{k},
    \quad 
    \bk_{2,k}=\tfrac{\muavg}{2},
    \quad
    \bk_{3,k}=\tfrac{3}{2}\muavg\bR^{-1}.
\end{equation*}
Then after $\numIter + 1$ iterations with 
\begin{equation}
    \label{eq:acc_cor_K}
    \numIter
    \ge
    \left\lceil
        \frac{\log\!\ls2C(f(\bx^0) - f(x^*))/\e\rs}{\log\ls1/(1-\alpha)\rs}
    \right\rceil,
\end{equation}
where
\begin{equation*}
    C
    \eqdef
    \tfrac{4\Done{\hv}{0}}{\muavg R}
    + \tfrac{4\Done{\hv}{0}\bR}{\muavg R^2}
    + \tfrac{4\Dtwo{\hv}{0}}{\muavg}
    + \tfrac{1}{2}
    + \tfrac{8\Lavg_2 + 3\muavg\bR^{-1}}{6\muavg}R,
\end{equation*}
the method outputs an $\e$-solution, i.e.,
\begin{equation*}
    f(\bx^{\numIter+1})-f(x^*) \le \e.
\end{equation*}
\end{customthm}
\begin{proof}
By~\eqref{eq:acc_cor_DH} and~\eqref{eq:acc_cor_Dv}, for every $k=1,\dots,\numIter$,
\begin{equation*}
    \Dtwo{\hv}{k}
    =
    \hDeH{\hv}{k} + 2\Lavg_2\hDex{\hv}{k}
    \le
    \tfrac{\muavg}{60\sqrt{5}\alpha^2}
    +
    2\Lavg_2
    \tfrac{\muavg}{120\sqrt{5}\alpha^2\Lavg_2}
    =
    \tfrac{\muavg}{30\sqrt{5}\alpha^2}.
\end{equation*}
Hence,
\begin{equation*}
\ls\tfrac{\muavg}{30\sqrt{5}\,\Dtwoo{\hv}}\rs^{1/2}\ge \alpha,
\end{equation*}
and therefore Theorem~\ref{thm:acc_convergence} applies with the constant choice
$\alpha_k\equiv \alpha$.

Using~\eqref{eq:acc_cor_DH}--\eqref{eq:acc_cor_Dv}, we obtain for every $j=1,\dots,\numIter$,
\begin{equation*}
    \tfrac{4\Lavg_1\bR^2}{\mumin}\hDeH{\hv}{j}
    \le
    \tfrac{\alpha\e}{80},
    \quad
    \tfrac{2\Lavg_1}{\mumin}\hDeg{\hv}{j}\bR
    \le
    \tfrac{\alpha\e}{80},
    \quad
    2\Lavg_1\hDex{\hv}{j}\bR
    \le
    \tfrac{\alpha\e}{160},
\end{equation*}
hence
\begin{equation*}
\ls
\tfrac{2\Lavg_1}{\mumin}
\ls
2\hDeH{\hv}{j}\bR + \hDeg{\hv}{j}
\rs
+
2\Lavg_1\hDex{\hv}{j}
\rs\bR
 \le \tfrac{\alpha\e}{32}.
\end{equation*}

Similarly, by~\eqref{eq:acc_cor_Dg} and~\eqref{eq:acc_cor_Dv},
\begin{equation*}
\Done{\hv}{j}
=
\hDeg{\hv}{j} + 2\Lavg_1\hDex{\hv}{j}
\le
\tfrac{\alpha\e}{160\bR}
+
\tfrac{\alpha\e}{160\bR}
=
\tfrac{\alpha\e}{80\bR}.
\end{equation*}

Now apply Theorem~\ref{thm:acc_convergence}. By~\eqref{eq:acc_cor_K}, the geometric term is bounded by
\begin{equation*}
(1-\alpha)^{\numIter}C (f(\bx^0) - f(x^*)) \le \tfrac{\e}{2}.
\end{equation*}

Further, using~\eqref{eq:acc_cor_Dgx} and
\begin{equation*}
    \alpha\tsum_{j=1}^{\numIter}(1-\alpha)^{\numIter-j}\le 1,
\end{equation*}
we get
\begin{equation*}
    \tsum_{j=1}^{\numIter}
    \alpha(1-\alpha)^{\numIter-j}
    \hDeg{x}{j+1}\bR
    \le
    \tfrac{\e}{8}.
\end{equation*}
Next, we use $\tsum_{j=1}^{\numIter}(1-\alpha)^{\numIter-j}\le \tfrac{1}{\alpha}$, and obtain
\begin{equation*}
    \begin{aligned}
        \tsum_{j=1}^{\numIter}
        \ls
        \alpha(1-\alpha)^{\numIter-j}
        +
        2(1-\alpha)^{\numIter-j}
        \rs & \\
        \times \left(
        \tfrac{2\Lavg_1}{\mumin}
        \ls
        2\hDeH{\hv}{j}\bR + \hDeg{\hv}{j}
        \rs
        \right. & \left.
        +
        2\Lavg_1\hDex{\hv}{j}
        \right) \bR
        \\ 
        & \quad \le
        \tfrac{\alpha+2}{32}\e
        \le
        \tfrac{3\e}{32}.
    \end{aligned}
\end{equation*}

Finally,
\begin{equation*}
    \begin{aligned}
        10\bR\tsum_{j=1}^{\numIter}(1+\alpha)(1-&\alpha)^{\numIter-j}\Done{\hv}{j}
         \\
        & \le 10(1+\alpha)\tsum_{j=1}^{\numIter}(1-\alpha)^{\numIter-j}\tfrac{\alpha\e}{80} \\
        &\le 
        \tfrac{10(1+\alpha)}{80}\e 
        \le
        \tfrac{\e}{4},
    \end{aligned}
\end{equation*}
since $\alpha\le 1$.

Combining all bounds, we get
\begin{equation*}
f(\bx^{\numIter+1})-f(x^*)
\le
\tfrac{\e}{2}+\tfrac{\e}{8}+\tfrac{3\e}{32}+\tfrac{\e}{4}
=
\tfrac{31}{32}\e
< \e.
\end{equation*}
This proves the claim.
\end{proof}

\textbf{Proof of Lemma~\ref{lem:acc_comm_complexity}}
\begin{proof}
    The proof is similar to the proof of Lemma~\ref{lem:comm_complexity}. The core difference is that, by Assumption~\ref{as:boundnes}, all evaluated local sequences ($v_i^k$ and $x_i^{k+1}$) are bounded within a radius $\bR$ from the optimum $x^*$. 
    
    Plugging  $\hDe_{H \vert \hv}$, $\hDe_{g \vert \hv}$, $\hDe_{\hv}$ and $\hDe_{g \vert x}$ from \eqref{eq:acc_inacc_main} into the universal contraction bound~\eqref{eq:general_contraction} guarantees the communication complexities~\eqref{eq:acc_T_complexities}.
\end{proof}

\EOD

\end{document}